\documentclass[12pt]{article}

\pdfoutput=1

\usepackage{amsmath, amssymb, amsthm, enumerate, fullpage, bm}
\usepackage{verbatim, enumitem, bbm, etoolbox}
\usepackage{latexsym}

\newcommand{\std}{{\mbox{\scriptsize{std}}}}

\newcommand{\pre}{{\mbox{\scriptsize{pre}}}}

\apptocmd{\lim}{\limits}{}{}

\theoremstyle{definition}
\newtheorem{thm}{Theorem}[section]
\newtheorem{theorem}[thm]{Theorem}

\newtheorem{lemma}[thm]{Lemma}

\newtheorem{fact}[thm]{Fact}

\numberwithin{subcase}{case}

\theoremstyle{definition}
\newtheorem{definition}[thm]{Definition}
\newtheorem{corollary}[thm]{Corollary}
\newtheorem{remark}[thm]{Remark}
\newtheorem{example}[thm]{Example}

\def\forkindep{\mathrel{\raise0.2ex\hbox{\ooalign{\hidewidth$\vert$\hidewidth\cr\raise-0.9ex\hbox{$\smile$}}}}}

\def\Ind{\setbox0=\hbox{$x$}\kern\wd0\hbox to 0pt{\hss$\mid$\hss}
	\lower.9\ht0\hbox to 0pt{\hss$\smile$\hss}\kern\wd0}

\def\Notind{\setbox0=\hbox{$x$}\kern\wd0\hbox to 0pt{\mathchardef
		\nn=12854\hss$\nn$\kern1.4\wd0\hss}\hbox to 0pt{\hss$\mid$\hss}\lower.9\ht0
	\hbox to 0pt{\hss$\smile$\hss}\kern\wd0}

\def\phi{\varphi}

\def\lg{{\rm lg}}
\def\<{\langle}
\def\>{\rangle}

\makeatletter
\def\blfootnote{\xdef\@thefnmark{}\@footnotetext}
\makeatother

\begin{document}	

	\bibliographystyle{plain}
	
	\author{Douglas Ulrich\!\!\
	\thanks{Partially supported
by Laskowski's NSF grant DMS-1308546.}\\
Department of Mathematics\\University of California, Irvine}
	\title{Pseudosaturation and the Interpretability Orders}
	\date{\today} 
	
	
	\maketitle
	
	
\begin{abstract}
We streamline treatments of the interpretability orders $\trianglelefteq^*_\kappa$ of Shelah \cite{SH500}, the key new notion being that of pseudosaturation. Extending work of Malliaris and Shelah \cite{ShelahIso} \cite{InterpOrders} \cite{InterpNew}, we classify the interpretability orders on the stable theories. As a further application, we prove that for all countable theories $T_0, T_1$, if $T_1$ is unsupersimple, then $T_0 \trianglelefteq^*_1 T_1$ if and only if $T_0 \trianglelefteq^*_{\aleph_1} T_1$. We thus deduce that simplicity is a dividing line in $\trianglelefteq^*_{\aleph_1}$, and that consistently, $SOP_2$ characterizes maximality in $\trianglelefteq^*_{\aleph_1}$; previously these results were only known for $\trianglelefteq^*_1$ \cite{pEqualsT2} \cite{InterpOrders}. 
\end{abstract}

\section{Introduction}

Keisler proved the following fundamental theorem in \cite{Keisler}; an ultrafilter $\mathcal{U}$ on $\mathcal{P}(\lambda)$ is $\lambda$-regular if there is a family $\mathcal{X} \subseteq \mathcal{U}$ of size $\lambda$, such that every infinite intersection from $\mathcal{X}$ is empty.

\begin{theorem}\label{KeislerOrigSecond}
	Suppose $T$ is a complete countable theory, and $\mathcal{U}$ is a $\lambda$-regular ultrafilter on $\mathcal{P}(\lambda)$, and $M_0, M_1 \models T$. Then $M_0^\lambda/\mathcal{U}$ is $\lambda^+$-saturated if and only if $M_1^\lambda/\mathcal{U}$ is. 

\end{theorem}

Motivated by this theorem, Keisler investigated the following  pre-ordering $\trianglelefteq$ on complete first-order theories; $\trianglelefteq$ is now called Keisler's order.

\begin{definition}
	Suppose $\mathcal{U}$ is a $\lambda$-regular ultrafilter on $\mathcal{P}(\lambda)$. Then say that $\mathcal{U}$ $\lambda^+$-saturates $T$ if for some or every $M \models T$, $M^\lambda/\mathcal{U}$ is $\lambda^+$-saturated.

	Given complete countable theories $T_0, T_1$, say that $T_0 \trianglelefteq_\lambda T_1$ if whenever $\mathcal{U}$ is a $\lambda$-regular ultrafilter on $\mathcal{P}(\lambda)$, if $\mathcal{U}$ $\lambda^+$-saturates $T_1$ then $\mathcal{U}$ $\lambda^+$-saturates $T_0$. Say that $T_0 \trianglelefteq T_1$ if $T_0 \trianglelefteq_\lambda T_1$ for all $\lambda$.

\end{definition}

When proving positive reductions $T_0 \trianglelefteq T_1$ in Keisler's order, the ultrapower context is often just notational baggage. The interpetability orders $\trianglelefteq^*_{\lambda \kappa}$ are abstractions of Keisler's order  which eliminate this baggage; they were introduced by Shelah \cite{SH500}. As a convenient piece of notation (following \cite{InterpOrders}), we say that every structure $M$ is $1$-saturated.

\begin{definition}
	Suppose $T, T_*$ are complete countable theories, and $\mathfrak{C}_*$ is the monster model of $T_*$. Then an interpretation of $T$ in $T_*$ is given by some definable subset $X$ of $\mathbf{C}_*^n$, and for each $m$-ary relation symbol $R \in \mathcal{L}_T$, an $m$-ary definable subset $R_* \subseteq X^m$, and for each $m$-ary function symbol $f \in \mathcal{L}_T$, an $m$-ary definable function $f_*: X^m \to X$, such that $(X, \ldots) \models T$. Given $M_* \models T_*$ we always get an interpreted model $M \models T$. We depict interpretations as functions $I: M_* \mapsto M$. 
	
	Suppose $\kappa$ is an infinite cardinal or $1$, and $\lambda \geq \aleph_0$. Suppose $T_0, T_1$ are complete countable theories. Then say that $T_0 \trianglelefteq^*_{\lambda \kappa} T_1$ if there is a complete countable theory $T_*$ and interpretations $I_0, I_1$ of $M_*$ in $T_*$, such that for all $\kappa$-saturated $M_* \models T_*$, if we let $M_i = I_i(M_*)$ be the interpreted model of $T_i$ in $M_*$, then: if $M_1$ is $\lambda^+$-saturated, then so is $M_0$. Say that $T_0 \trianglelefteq^*_{ \kappa} T_1$ if $T_0 \trianglelefteq^*_{\lambda \kappa} T_1$ for all infinite $\lambda$.\footnote{We follow the indexing system of \cite{SH500}, where these orders were first introduced. In later papers, e.g. in the recent \cite{InterpOrders}, $\lambda^+$-saturation is replaced by $\lambda$-saturation, and $T_0 \trianglelefteq^*_{\kappa} T_1$ is defined to mean $T_0 \trianglelefteq^*_{\lambda \kappa} T_1$ for sufficiently large regular $\lambda$. Under this alternative definition, it is no longer necessarily true that $\trianglelefteq^*_{\aleph_1} \subseteq \trianglelefteq$. 
		
	In any case, conjecturally the two systems give the same ordering.}
	
\end{definition} 

Clearly, if $\kappa \leq \kappa'$ then $\trianglelefteq^*_{\lambda \kappa} \subseteq \trianglelefteq^*_{\lambda \kappa'}$ for all $\lambda$, and $\trianglelefteq^*_{\kappa} \subseteq \trianglelefteq^*_{\kappa'}$. It is consistent with what we know that $\trianglelefteq = \trianglelefteq^*_\kappa$ for every $\kappa \geq \aleph_1$. Malliaris and Shelah show in \cite{InterpOrders} that $\trianglelefteq^*_1 \subsetneq \trianglelefteq^*_{\aleph_1}$, and in the present work we show in fact that $\trianglelefteq^*_1 \subsetneq \trianglelefteq^*_{\aleph_0} \subsetneq \trianglelefteq^*_{\aleph_1}$.

Shelah observed in \cite{SH500} that $\trianglelefteq^*_{\lambda \aleph_1} \subseteq \trianglelefteq_{\lambda}$ for all $\lambda$, and hence $\trianglelefteq^*_{\aleph_1} \subseteq \trianglelefteq$; this is not hard to see. Indeed, suppose $T_0 \trianglelefteq^*_{\lambda \aleph_1} T_1$, via interpretations $I_0$ and $I_1$ of $T_0$ and $T_1$ in $T_*$. Let $\mathcal{U}$ be a $\lambda$-regular ultrafilter on $\mathcal{P}(\lambda)$. Let $M_* \models T_*$ be given and write $M_i = I_i(M_*)$. Then for each $i < 2$, $I_i(M_*^\lambda/\mathcal{U}) \cong M_i^\lambda/\mathcal{U}$. Hence, if $M_1^\lambda/\mathcal{U}$ is $\lambda^+$-saturated, then so is $M_0^\lambda/\mathcal{U}$. In other words, if $\mathcal{U}$ $\lambda^+$-saturates $T_1$, then $\mathcal{U}$ $\lambda^+$-saturates $T_0$. Thus $T_0 \trianglelefteq_\lambda T_1$.

 In Section \ref{SurveyInterpOrderSec}, we first restrict attention to the only interpretation that we need, namely set theory. Suppose $V$ is a transitive model of $ZFC^-$ ($ZFC$ without powerset), and $M \in V$ is an $\mathcal{L}$-structure. Given an elementary embedding $\mathbf{j}: V \preceq \hat{V}$, let $\mathbf{j}_{\std}(M)$ be the natural $\mathcal{L}$-structure with domain $\mathbf{j}[M]$. Say that $\hat{V} \models ZFC^-$ is $\omega$-nonstandard if it contains nonstandard natural numbers.
 
 The following lemma is straightforward:
 
 \begin{lemma}\label{SatAndPseudo0First}
 	Suppose $\lambda$ is an infinite cardinal, $\kappa$ is an infinite cardinal or $1$, and $T_0, T_1$ are complete countable theories. Then the following are equivalent:
 	\begin{itemize}
 		\item[(A)] $T_0 \trianglelefteq^*_{\lambda \kappa} T_1$;
 		\item[(B)] There is some countable transitive $V \models ZFC^-$ with $T_0, T_1 \in V$, and some $M_i \models T_i$ both in $V$, such that whenever $\mathbf{j}:V \preceq \hat{V}$, if $\hat{V}$ is $\kappa$-saturated and $\omega$-nonstandard, and if $\mathbf{j}_{\std}(M_1)$ is $\lambda^+$-saturated, then $\mathbf{j}_{\std}(M_0)$ is $\lambda^+$-saturated.
 	\end{itemize}
 \end{lemma}
 
 This is reminiscent of Keisler's order, if we view $\mathbf{j}_{\std}$ as a generalized ultrapower. However, the choices of $M_i \in V$ matter and are frequently delicate. Theorem~\ref{KeislerOrigSecond}, on the other hand, shows that if $\hat{V}$ is of the form $V^\lambda/\mathcal{U}$ for some $\lambda$-regular ultrafilter $\mathcal{U}$ on $\mathcal{P}(\lambda)$, then the choices of $M_0, M_1$ do not matter. This is because $\mathbf{j}_{\std}(M_i) \cong M_i^\lambda/\mathcal{U}$, where $\mathbf{j}: V \preceq \hat{V}$ is the {\L}o{\'s} embedding. One of our key observations is that this can be recaptured in a very general setting: the $\lambda$-regularity of $\mathcal{U}$ implies that every subset of $\hat{V}$ of size at most $\lambda$ is contained in some $\hat{X} \in \hat{V}$ which is finite in the sense of $\hat{V}$. This in and of itself is enough to imply that the choice of $M_i \models T_i$ does not matter. 
 
	In more detail, suppose $V \models ZFC^-$ is transitive and $\mathbf{j}: V \preceq \hat{V}$. Then say that $X \subseteq \hat{V}$ is pseudofinite if there is some $\hat{X} \in \hat{V}$ finite in the sense of $\hat{V}$, such that $X \subseteq \hat{X}$ (by this we mean that for all $x \in X$, $x \hat{\in} \hat{X}$, where $\hat{\in}$ is the membership relation of $\hat{V}$). If $M$ is a structure in $V$, then say that $\mathbf{j}_{\std}(M)$ is $\kappa$-pseudosaturated if for every pseudofinite $A \subseteq \mathbf{j}_{\std}(M)$ of cardinality less than $\kappa$, every type over $A$ is realized in $\mathbf{j}_{\std}(M)$. We can then prove the following theorem; this gives a new proof of Theorem~\ref{KeislerOrigSecond}(A).
 
 \begin{theorem}\label{ModKeislerOrderFirst}
 	Suppose $V \models ZFC^-$ is transitive and $\mathbf{j}: V \preceq \hat{V}$ is $\omega$-nonstandard, and $T \in V$ is a complete countable theory, and $M_0, M_1 \in V$ are two models of $T$. Suppose $\kappa$ is an uncountable cardinal. Then $\mathbf{j}_{\std}(M_0)$ is $\kappa$-pseudosaturated if and only if $\mathbf{j}_{\std}(M_1)$ is.
 	
 	Hence, if every subset of $\hat{V}$ of size less than $\kappa$ is pseudofinite (for example if $\hat{V} = V^\lambda/\mathcal{U}$ for some $\lambda$-regular ultrafilter $\mathcal{U}$ on $\mathcal{P}(\lambda)$), then $\mathbf{j}_{\std}(M_0)$ is $\lambda^+$-saturated if and only if $\mathbf{j}_{\std}(M_1)$ is.
 \end{theorem}

Inspired by this theorem, we define that $\hat{V}$ $\kappa$-pseudosaturates $T$ if for some or every $M \models T$ with $M \in V$, $\mathbf{j}_{\std}(M)$ is $\kappa$-pseudosaturated. It is then natural to define $T_0 \trianglelefteq^\times_{\lambda \kappa} T_1$ if there is some countable transitive $V \models ZFC^-$ with $T_0, T_1 \in V$, such that whenever $\mathbf{j}: V \preceq \hat{V}$, if $\hat{V}$ is $\kappa$-saturated and $\omega$-nonstandard, and if $\hat{V}$ $\lambda^+$-pseudosaturates $T_1$, then $\hat{V}$ $\lambda^+$-pseudosaturates $T_0$; and we define $T_0 \trianglelefteq^\times_{\kappa} T_1$ if $T_0 \trianglelefteq^\times_{\lambda \kappa} T_1$ for all infinite $\lambda$. 

In Section~\ref{FullBVModelsSec}, we connect the approach of pseudosaturation with the method of full Boolean-valued models from \cite{BVModelsUlrich}.  In Section~\ref{SurveyMaxSOP2Sec}, we follow Malliaris and Shelah's proof that $SOP_2$ theories are maximal in Keisler's order \cite{pEqualsTref} to show that $SOP_2$ theories are maximal in $\trianglelefteq^\times_1$. We remark that in \cite{pEqualsT2}, Malliaris and Shelah prove by a somewhat more involved argument that $SOP_2$ theories are maximal in $\trianglelefteq^*_1$.

 Over the next several sections, we prove that $\trianglelefteq^\times_\kappa \subseteq \trianglelefteq^*_\kappa$ for all $\kappa$ infinite or $1$ (in words: if $T_0 \trianglelefteq^\times_\kappa T_1$, then $T_0 \trianglelefteq^*_\kappa T_1$). Thus, reductions proved in $\trianglelefteq^\times_\kappa$ carry over to $\trianglelefteq^*_\kappa$. We note that many complications encountered when dealing with $\trianglelefteq^*_\kappa$ are absorbed into the proof that $\trianglelefteq^\times_\kappa \subseteq \trianglelefteq^*_\kappa$, and subsequently disappear when working with $\trianglelefteq^\times_\kappa$; for instance, this is the case in our proof that $SOP_2$ theories are maximal in $\trianglelefteq^\times_1$. Many more examples will be given in \cite{InterpOrders2Ulrich}.


In Section~\ref{InterpOnStableSec0}, we show that if $T_1$ is unstable and if $T_0 \trianglelefteq^\times_\kappa T_1$, then $T_0 \trianglelefteq^*_\kappa T_1$. The key point is to consider a certain expansion $ZFC^-_*$ of $ZFC^-$ which puts the universe in bijection with an initial segment of the natural numbers in a particularly nice way (note that necessarily, models of $ZFC^-_*$ have nonstandard natural numbers). As notation, given $\hat{V} \models ZFC^-$, we say that $\hat{V} \models ZFC^-_{\pre}$ if it can be expanded to a model of $ZFC^-_*$. The key property of models of $ZFC^-_{\pre}$ is: suppose $V \models ZFC^-$ is transitive, and $T \in V$ is a countable unstable theory and $M \models T$ with $M \in V$. If $\mathbf{j}: V \preceq \hat{V} \models ZFC^-_{\pre}$, then for all $\lambda$, $\mathbf{j}_{\std}(M)$ is $\lambda^+$-saturated if and only if it is $\lambda^+$-pseudosaturated.

To finish the proof that $\trianglelefteq^\times_\kappa \subseteq \trianglelefteq^*_\kappa$, we need to understand the behaviors of these orders on the stable theories. Shelah classified Keisler's order on the stable theories in \cite{ShelahIso}, and Malliaris and Shelah extended this classification to $\trianglelefteq^*_{\aleph_1}$ in \cite{InterpOrders}. In work in preparation \cite{InterpNew}, Malliaris and Shelah further extend this classification to $\trianglelefteq^*_1$. They privately communicated a sketch of the proof to me, and with their permission, we adapt their arguments to classify $\trianglelefteq^*_\kappa$ and $\trianglelefteq^\times_\kappa$ on stable theories for every cardinal $\kappa$ (infinite or $1$). 

To begin, in Sections~\ref{InterpOnStableSec1} and ~\ref{InterpOnStableSec2}, we give complete answers to the following questions:

\vspace{1 mm}

\noindent \textbf{Question 1.} Suppose $V \models ZFC^-$ is a transitive model of $ZFC^-$, and $T \in V$ is a countable stable theory, and $\mathbf{j}: V \preceq \hat{V}$ is $\omega$-nonstandard. For which cardinals $\lambda$ does $\hat{V}$ $\lambda$-pseudosaturate $T$?

\vspace{1 mm}

\noindent \textbf{Question 2.} Suppose $V \models ZFC^-$ is a transitive model of $ZFC^-$, and $T \in V$ is a countable stable theory, and $M \models T$ with $M \in V$, and $\mathbf{j}: V \preceq \hat{V} \models ZFC^-_{\pre}$. For which cardinals $\lambda$ is $\mathbf{j}_{\std}(M)$ $\lambda$-saturated?

\vspace{1 mm}

In Section~\ref{KeislerOnStableSection}, we apply the preceding computations to complete our analysis of the interpretability orders $\trianglelefteq^*_\kappa, \trianglelefteq^\times_\kappa, \trianglelefteq$ on the stable theories. We remark that there is a minor gap in Shelah's original argument for the classification of Keisler's order on the stable theories, which we circumvent (see Remark~\ref{GapRemark}).

In Section~\ref{DiscussionSec}, we remark on several consequences of our analysis. In particular, we extend Theorem~\ref{ModKeislerOrderFirst} to the case when $\kappa = \aleph_0$, and we prove the following analogue for saturation:

 \begin{theorem}\label{ModKeislerOrderSatFirst}
	Suppose $V \models ZFC^-$ is transitive and $\mathbf{j}: V \preceq \hat{V} \models ZFC^-_{\pre}$, and $T \in V$ is a complete countable theory, and $M_0, M_1 \in V$ are two models of $T$. Suppose $\kappa$ is an infinite cardinal. Then $\mathbf{j}_{\std}(M_0)$ is $\kappa$-saturated if and only if $\mathbf{j}_{\std}(M_1)$ is.
\end{theorem}

Hence, if $\mathbf{j}: V \preceq \hat{V} \models ZFC^-_{\pre}$ and $T \in V$ is a complete countable thoery, we can define that $\hat{V}$ $\kappa$-saturates $T$ if for some or every $M \models T$ with $M \in V$, we have that $\mathbf{j}_{\std}(M)$ is $\kappa$-saturated. We thus obtain the following equivalent formulation of the interpretability orders:

\begin{theorem}\label{SatAndPseudoCorFirst}
	Suppose $T_0, T_1$ are complete countable theories, and suppose $\kappa$ is infinite or $1$. Then the following are equivalent:
	
	\begin{itemize}
		\item[(A)] $T_0 \trianglelefteq^*_{\kappa} T_1$;
		\item[(B)] For all cardinals $\lambda$, there is some countable transitive $V \models ZFC^-$ with $T_0, T_1 \in V$, such that for all $\mathbf{j}: V \preceq \hat{V} \models ZFC^-_{\pre}$ with $\hat{V}$ $\kappa$-saturated, if $\hat{V}$ $\lambda^+$-pseudosaturates  $T_1$, then $\hat{V}$ $\lambda^+$-pseudosaturates $T_0$.
		\item[(C)] For all cardinals $\lambda$, there is some countable transitive $V \models ZFC^-$ with $T_0, T_1 \in V$, such that for all $\mathbf{j}: V \preceq \hat{V} \models ZFC^-_{\pre}$ with $\hat{V}$ $\kappa$-saturated, if $\hat{V}$ $\lambda^+$-saturates  $T_1$, then $\hat{V}$ $\lambda^+$-saturates $T_0$.
	\end{itemize}
\end{theorem}

Note that it follows immediately that if $T_0 \trianglelefteq^\times_\kappa T_1$, then $T_0 \trianglelefteq^*_\kappa T_1$.

We also prove the following, using ideas from Section~\ref{InterpOnStableSec1}:

\begin{theorem}\label{SatAndPseudoThm2First}
	Suppose $T_0, T_1$ are complete countable theories, and $T_1$ is unsupersimple. Then $T_0 \trianglelefteq^*_1 T_1$ if and only if $T_0 \trianglelefteq^*_{\aleph_1} T_1$.
\end{theorem}

We use this to strengthen two theorems on $\trianglelefteq^*_{1}$ to the context of $\trianglelefteq^*_{\aleph_1}$. First, it follows from results of 
D\v{z}amonja and Shelah \cite{SH692}, Shelah and Usvyatsov \cite{SOP2pt2}, and Malliaris and Shelah \cite{pEqualsTref}, as pieced together by Malliaris and Shelah \cite{pEqualsT2}, that $SOP_2$ consistently characterizes maximality in $\trianglelefteq^*_1$. Second, Malliaris and Shelah prove in \cite{InterpOrders} that simplicity is a dividing line in $\trianglelefteq^*_1$ (i.e. if $T_0$ is unsimple and $T_1$ is simple, then $T_0 \not \trianglelefteq^*_1 T_1$.) We deduce both of these theorems for $\trianglelefteq^*_{\aleph_1}$.

\vspace{1 mm}

\noindent \textbf{Acknowledgements.} We would like to thank Vincent Guingona, Alexei Kolesnikov, Chris Laskowski, Pierre Simon and Jindrich Zapletal for several helpful conversations.

\section{Model-theoretic Preliminaries}\label{Prelim}

We collect together various model-theoretic facts we will need throughout the paper.

\begin{definition}

Suppose $T$ is a complete countable theory with monster model $\mathfrak{C}$, and $\phi(x, \overline{y})$ is a formula of $T$. Then:

\begin{itemize}
	\item  $\phi(\overline{x}, \overline{y})$ has the strict order property of the second kind ($SOP_2$) if there are $(\overline{b}_s: s \in \omega^{<\omega})$ from $\mathfrak{C}^{|\overline{y}|}$, such that for each $\eta \in \omega^\omega$, $(\phi(\overline{x}, \overline{b}_{\eta \restriction_n}): n < \omega)$ is consistent, but whenever $s, t \in \omega^{<\omega}$ are incomparable, $\phi(\overline{x}, \overline{b}_s) \land \phi(\overline{x}, \overline{b}_t)$ is inconsistent. Otherwise, we say that $\phi(\overline{x}, \overline{y})$ has $NSOP_2$. We say that $T$ has $SOP_2$ if some formula of $T$ is $SOP_2$, and otherwise $T$ has $NSOP_2$.
	\item $\phi(\overline{x}, \overline{y})$ has the tree property $(TP)$ if there are $(\overline{b}_s: s \in \omega^{<\omega})$ from $\mathfrak{C}^{|\overline{y}|}$ such that for each $\eta \in \omega^\omega$, $(\phi(\overline{x}, \overline{b}_{\eta \restriction_n}): n < \omega)$ is consistent, but whenever $s \in \omega^{<\omega}$ and $i < j < \omega$, then $\phi(\overline{x}, \overline{b}_{s\,^\frown(i)}) \land \phi(\overline{x}, \overline{b}_s^\frown(j))$ is inconsistent. Otherwise $\phi(\overline{x}, \overline{y})$ has $NTP$.
	
	We say $T$ is unsimple if some formula of $T$ has the tree property; otherwise, $T$ is simple.
	
	\item $\phi(\overline{x}, \overline{y})$ has the independence property ($IP$) if there are $(\overline{b}_n: n < \omega)$ from $\mathfrak{C}^{|\overline{y}|}$ such that for all disjoint $u, v \subseteq \omega$, $\{\phi(\overline{x},\overline{b}_n: n \in )\} \cup \{\lnot \phi(\overline{x}, \overline{b}_n): n \in v\}$ is consistent. Otherwise, $\phi(\overline{x}, \overline{y})$ has $NIP$.
	
	$T$ has $IP$ if some formula of $T$ does; otherwise, $T$ has $NIP$. 
	\item 	$\phi(\overline{x}, \overline{y})$ is unstable if there are $(\overline{b}_n: n < \omega)$ from $\mathfrak{C}^{|\overline{y}|}$ such that for all $n < \omega$, $\{\phi(\overline{x}, \overline{b}_m: m < n)\} \cup \{\lnot \phi(\overline{x}, \overline{b}_m): m \geq n\}$ is consistent. Otherwise $\phi(\overline{x}, \overline{y})$ is stable.
	
	We say $T$ is unstable if some formula of $T$ is; otherwise we say $T$ is stable.
	
	\item $\phi(\overline{x}, \overline{y})$ has the finite cover property ($FCP$) if: for every $n$ there exists $m > n$ and $(\overline{b}_i: i < m)$ from $\mathfrak{C}^{|\overline{y}|}$ such that $\{\phi(\overline{x}, \overline{b}_i): i < m\}$ is inconsistent, but every $n$-element subset is consistent. Otherwise $\phi(\overline{x}, \overline{y})$ has $NFCP$.
	
	We say $T$ has $FCP$ if some formula of $T$ does; otherwise $T$ is $NFCP$.
\end{itemize}

\end{definition}

\begin{remark}
	The tree property of the first kind $TP_1$ is equivalent to $SOP_2$; we pick the term $SOP_2$ to use. See \cite{TP1} for a discussion. 
\end{remark}

\begin{remark}
In all of these definitions, it suffices to consider the case where $\overline{x}$ is a single variable $x$.
\end{remark}

We recall that in simple theories, forking (equivalently dividing) is a well-behaved independence relation. The precise definition will not be important for us.

\begin{definition}
	Suppose $T$ is a complete countable theory. Then:
	
	\begin{itemize}
		\item $T$ is supersimple if $T$ is simple and every type does not fork over a finite subset;
		\item $T$ is superstable if $T$ is supersimple and stable;
		\item $T$ is $\omega$-stable if for every $M \models T$, $|S^1(M)| \leq M$ (it is enough to check when $M$ is countable; this implies stability).
	\end{itemize}
\end{definition}

\begin{definition}
	Suppose $T$ is a complete countable theory, and $p(x) \in S^1(A)$. Then $p(x)$ is stationary if for all $B \supseteq A$, $p(x)$ has a unique nonforking extension to $B$. If $q(x) \in S^1(B)$ and $A \subseteq B$, then say that $q(x)$ is based on $A$ if $q(x)$ does not fork over $A$ and $q(x) \restriction_A$ is stationary.

\end{definition}

The following notation from \cite{AlephEpsilon} will be convenient for discussing superstability; the idea is $0 < \epsilon < 1$, so $\aleph_0 < \aleph_{\epsilon} < \aleph_1$.

\begin{definition}
	If $M$ is a structure, then say that $X \subseteq M$ is $\aleph_{\epsilon}$-finite if it is contained in the algebraic closure of a finite set. Say that the structure $M$ is $\aleph_{\epsilon}$-saturated if every type over an $\epsilon$-finite set is realized in $M$. Say that $T$ is $\epsilon$-small if for every $M \models T$ and for every $\epsilon$-finite $A \subseteq M$, $S^1(A)$ is countable. This is a strengthening of smallness, which asserts the same for every finite $A \subseteq M$.
\end{definition}

We now list several facts we will be using.

\begin{fact}\label{Fact1}
	Among complete countable theories $T$: $NFCP$ implies stable implies simple implies $NSOP_2$.
\end{fact}
\begin{proof}
	$NFCP$ implies stable is Theorem II.4.2 of \cite{ShelahIso}.
	For stable implies simple, see \cite{KimForking}. Simple implies $NSOP_2$ is one direction of Theorem 0.2 of \cite{ShelahSimple}.
\end{proof}

\begin{fact}\label{Fact2}
	Suppose $T$ is a complete countable theory $T$. 
	
	\begin{itemize}
		\item[(A)] $T$ is simple if and only if every type does not fork over a countable subset of its domain.		
		\item[(B)] $T$ is stable if and only if every type is based on a countable subset of its domain.
		\item[(C)] Suppose $T$ eliminates imaginaries. Then $T$ is superstable if and only if every type is based on an $\epsilon$-finite set.
		\item[(D)] $T$ is $\omega$-stable if and only if $T$ is small, and every type is based on a finite set. 
	\end{itemize}
\end{fact}
\begin{proof}
	(A) is frequently given as a definition of simplicity, e.g. as in \cite{KimForking}, where it is proved to be equivalent to no formula having the tree property.
	
	(B) follows from the fact that stability can be alternatively characterized by type-counting, as in \cite{KimForking}.
	
	For the remaining facts, see Chapter 1 of \cite{pillay}.
\end{proof}

The following dichotomy theorem of Shelah is well-known.

\begin{theorem}\label{NonStableDichotomy}
	$T$ is unstable if and only if either $T$ has $IP$ or $SOP_2$.
\end{theorem}
\begin{proof}
	Theorem II.4.7 of \cite{ShelahIso} states that $T$ is unstable if and only if either $T$ has $IP$ or else $T$ has $SOP$. But if $T$ has $SOP$ then $T$ has $SOP_2$, and if $T$ has $SOP_2$ then $T$ is unstable, so the theorem follows. (The definition of $SOP$ won't be important for us.)
\end{proof}

The following characterization of supersimplicity is the conjunction of Lemmas 2.1 and 3.1 of \cite{Casanovas}.

\begin{lemma}\label{SupersimpleCharacterization}
	Suppose $T$ is a complete theory. Then $T$ is unsupersimple if and only if there are formulas $(\phi_n(\overline{x}, \overline{y}_n): n < \omega)$ and tuples $(\overline{a}_s: s \in \omega^{<\omega})$ from some $M \models T$ such that:
	\begin{itemize}
		\item Each $|\overline{a}_s| = |\overline{y}_{|s|}|$;
		\item For all $\eta \in \omega^\omega$, $\{\phi_n(\overline{x}, \overline{a}_{\eta \restriction_n}): n < \omega\}$ is consistent,
		\item For all $s \in \omega^{<\omega}$, and for all $n < m < \omega$, $\phi_{|s|+1}(\overline{x}, \overline{a}_{s\,^\frown\,(n)}) \land \phi_{|s|+1}(\overline{x}, \overline{a}_{s\,^\frown\,(m)})$ is inconsistent.
	\end{itemize}
\end{lemma}

We will need some machinery on indiscernible sets in stable theories. As notation:

\begin{definition}
	Let $T$ be a complete countable theory. Then $\Delta$ is a set of partitioned formulas if $\Delta = \{\phi_i(\overline{x}, \overline{y}_i): i \in I\}$, i.e. it is a set of partitioned formulas with distinguished variables $\overline{x}$ . The arity of $\Delta$ is $|\overline{x}|$. Suppose $\Delta$ is a set of partitioned formulas; then a (positive) $\Delta$-formula is a (positive) boolean combination of formulas of the form $\phi(\overline{x}, \overline{a})$ for $\phi(\overline{x}, \overline{y}) \in \Delta$ and parameters $\overline{a} \in \mathfrak{C}$. A (positive) $\Delta$-type is a partial type $p(\overline{x})$, such that every $\phi(\overline{x}, \overline{a}) \in p(\overline{x})$ is a (positive) $\Delta$-formula. If $A$ is a set we let $S_\Delta(A)$ be the set of all maximal $\Delta$-types over $A$; so we have the obvious restriction map from $S_n(A)$ to $S_\Delta(A)$, where $n$ is the arity of $\Delta$. If $\overline{a}\in \mathfrak{C}$ then $tp_\Delta(\overline{a}/A) \in S_\Delta(A)$ is the set of all $\Delta$-formulas $\phi(\overline{x}, \overline{b})$ such that $\overline{b} \in A$ and $ \models \phi(\overline{a}, \overline{b})$. If $\Delta$ is the single formula $\phi(\overline{x}, \overline{y})$ we write $\phi(\overline{x}, \overline{y})$ instead of $\{\phi(\overline{x}, \overline{y})\}$.

	Suppose $I$ is an index set and $\Delta$ is a set of partitioned formulas of $T$ of arity $m$. For an index set $I$, a set $\{\overline{a}_i: i \in I\}$ is $(\Delta, n)$-indiscernible if each $\overline{a}_i$ has length $m$ and for every tuples of distinct elements $(i_0, \ldots, i_{n-1})$, $(j_0, \ldots, j_{n-1})$ from $I^n$, $tp_\Delta(\overline{a}_{i_0}/\overline{a}_{i_1} \ldots \overline{a}_{i_{n-1}}) = tp_\Delta(\overline{a}_{j_0}/ \overline{a}_{j_1} \ldots \overline{a}_{j_{n-1}})$.  $I$ is $\Delta$-indiscernible if it is $(\Delta, n)$-indiscernible for all $n$. $I$ is indiscernible if it is $\Delta$-indiscernible, where $\Delta$ is the set of all partitioned $\mathcal{L}$-formulas of $T$ of arity $m$.
\end{definition}

The following equivalents of the finite cover property are proved in \cite{ShelahIso}. (A) if and only if (B) is Theorem II.4.2, and (C) implies (A) is Theorem II.4.6. (B) implies (C) is trivial.
\begin{theorem}\label{fcpEquiv}
	For $T$ a countable stable theory, the following are equivalent:
	\begin{itemize}
		\item[(A)] $T$ has the finite cover property.
		\item[(B)] There is a formula $\phi(x, y, \overline{z})$ such that for every $\overline{c} \in \mathfrak{C}$ of length $|\overline{z}|$, $\phi(x, y, \overline{c})$ defines an equivalence relation $E_{\overline{c}}$, and for arbitrarily large $n$ there is $\overline{c}_n \in \mathfrak{C}^{|\overline{z}|}$ such that $E_{\overline{c}_n}$ has exactly $n$ classes.
		\item[(C)] There is some finite set $\Delta$ of partitioned formulas and some $M \models T$, such that for arbitrarily large $m$ there is a $(\Delta,n)$-indiscernible set $\{\overline{a}_i: i < m\}$ from $M$ which cannot be extended to an infinite $(\Delta,n)$-indiscernible set from $M$.
	\end{itemize}
\end{theorem}

The last bit of notation we need is that of average types of indiscernible sets.

\begin{definition}
	Suppose $T$ is a countable stable theory, and $M \models T$, and $I \subseteq M$ is indiscernible. Then the average type of $I$ over $M$, $\mbox{Av}(I, M)$, is the set of all $\phi(x, \overline{a})$ with $\overline{a} \in M^{<\omega}$, such that for all but finitely many $b \in I$, $M \models \phi(b, \overline{a})$; this is a complete type over $M$. Two indiscernible sequences $I, J$ are equivalent if they have the same average type. (See Definitions III.1.5, III.1.6 and Lemmas III.1.7 and III.1.8 of \cite{ShelahIso}.)
	
\end{definition}

\begin{lemma}\label{InterpKeislerLemma0}
	Suppose $T$ is a countable stable theory, $M  \models T$, $A \subseteq M$, and $p(x) \in S(A)$ is stationary; let $q(x)$ be the nonforking extension of $p(x)$ to $M$. Suppose $I \subseteq M$ is an independent set of realizations of $p(x)$. Then $I$ is indiscernible over $A$, and for every $b \in I$, $b$ realizes $q(x) \restriction_{A \cup (I \backslash \{b\})}$. Further, $\mbox{Av}(I, M) = q(x)$.
\end{lemma}
\begin{proof}
	This is essentially Lemma III.1.10 of \cite{ShelahIso}. More precisely, enumerate $I = (b_\alpha: \alpha \leq \alpha_*)$, so that $b_{\alpha_*} = b$. Since $I$ is independent, we have that each $b_\alpha$ realizes $q(x) \restriction_{A \cup \{b_\beta: \beta < \alpha\}}$, and hence $b$ realizes $q(x) \restriction_{A \cup (I \backslash \{b\})}$. Further, the hypotheses of Lemma III.1.10 are now met, and so we conclude.
\end{proof}

The following is a piece of Lemma III.3.10 of \cite{ShelahIso}. 
\begin{lemma}\label{ShelahStableSatLemma2}
	Suppose $T$ is a countable stable theory, and $M \models T$, and $p(x) \in S^1(M)$. Suppose that there is an indiscernible set $I \subseteq M$ such that $\mbox{Av}(I, M) = p(x)$.	Then for every $A \subseteq M$ with $|A|+ \aleph_0 < |I|$, $p(x) \restriction_A$ is realized in $M$.
\end{lemma}
\begin{proof}
	Let $q(x) = p(x) \restriction_A$. Note that $|q(x)| < |I|$. For each $\phi(x) \in q(x)$, let $I_\phi \subseteq I$ be the cofinite set of all $ a \in I_\phi$ such that $M \models \phi(a)$. Let $J = \bigcap_{\phi \in q(x)} I_\phi$; note that $|J| = |I|$ as we are removing fewer than $|I|$-many finite sets. But any $a \in J$ realizes $q(x)$.
\end{proof}

\section{Pseudosaturation}\label{SurveyInterpOrderSec}

Traditionally, to show that $T_0 \trianglelefteq T_1$, one carefully picks a theory $T_*$ interpreting both $T_0$ and $T_1$, and containing enough set theory to make the argument at hand work. We find it more convenient to just expand to a model of set theory, which gives us everything we could possibly need.

\begin{definition} $ZFC^-$ is $ZFC$ without powerset, and with replacement strengthened to collection, and with choice strengthened to the well-ordering principle; this is as in \cite{ZFCminus}. We like $ZFC^-$ since it is a reasonably strong fragment of $ZFC$ and it has many transitive models; in particular, for every hereditarily countable set $A$, there is a countable transitive $V \models ZFC^-$ with $A \in V$. The reader who is comfortable with mild large cardinals should feel free to replace $ZFC^-$ by $ZFC$ (or more) everywhere. In the other direction, really we just need weak fragments of arithmetic, as worked out by Malliaris and Shelah \cite{pEqualsT2} \cite{InterpOrders} \cite{pEqualsTref}. 
\end{definition}

\begin{remark}
	$ZFC^-$ is strong enough to prove most standard theorems which do not explicitly mention powersets. For instance, $ZFC^-$ proves that every consistent theory has a model, and every countable consistent theory has a countable model. Also, if $V$ is a transitive model of $ZFC^-$, then $\Pi^1_1$ statements are absolute to $V$, so (for instance) if $T \in V$ is a consistent theory, then $T$ is consistent in $V$, and so $T$ has a model in $V$; further, if $T$ is countable in $V$ then $T$ has a countable model in $V$. 
\end{remark}

\begin{definition}
	Say that $\hat{V} \models ZFC^-$ is an $\omega$-model, or is $\omega$-standard, if every natural number of $\hat{V}$ is standard (i.e. has finitely many $\hat{\in}$-predecessors).
	
	$V$ will denote a transitive model of $ZFC^-$. $\hat{V}$ will denote an $\omega$-nonstandard model of $ZFC^-$. Frequently $\hat{V}$ will come from an embedding $\mathbf{j}:V \preceq \hat{V}$, where $V$ is transitive. Whenever $\hat{V} \models ZFC^-$, we will identify $HF$ (the hereditarily finite sets) with its copy in $\hat{V}$. Other elements of $\hat{V}$ will usually be decorated with a hat, for instance we write $\hat{\omega}$ rather than $(\omega)^{\hat{V}}$; but sometimes readability takes precedence. Given $X \subseteq \hat{V}$, we say that $X$ is an internal subset of $\hat{V}$ if there is some $\hat{X} \in \hat{V}$ such that $X = \{\hat{y} \in \hat{V}: \hat{y} \hat{\in} \hat{X}\}$. In this case, we usually identify $X$ with $\hat{X}$ and will write that $X \in \hat{V}$. 

	Suppose $V \models ZFC^-$ is transitive, and $\mathbf{j}: V \preceq \hat{V}$, and $M$ is an $\mathcal{L}$-structure in $V$. Note that $\mathbf{j}(M)$ is a $\mathbf{j}(\mathcal{L})$-structure, where possibly some of the symbols of $\mathbf{j}(\mathcal{L})$ are nonstandard; let $\mathbf{j}_{\std}(M)$ be the ``reduct" to $\mathcal{L}$. Formally, $\mathbf{j}_{\std}(M)$ is the structure with universe $\{a \in \hat{V}: a \in \mathbf{j}(M)\}$, and where for each $n$-ary relation symbol $R \in \mathcal{L}$, $R^{\mathbf{j}_{\std}(M)}$ is $\{\overline{a} \in (\mathbf{j}_{\std}(M))^n: \hat{V} \models \mathbf{j}(R^M)(\overline{a})\}$, etc.
\end{definition}

\begin{example}
	Suppose $\mathcal{U}$ is an ultrafilter on $\mathcal{P}(\lambda)$; let $V$ be any transitive model of $ZFC^-$, write $\hat{V} = V^\lambda/\mathcal{U}$ and let $\mathbf{j}: V \preceq \hat{V}$ be the {\L}o{\'s} embedding. Then for any structure $M \in V$, $\mathbf{j}_{\std}(M) \cong M^\lambda/\mathcal{U}$.

\end{example}

Thus we can view the $\mathbf{j}_{\std}$ operator as a generalized ultrapower.

\vspace{2 mm}

\noindent \textbf{Conventions.} We operate entirely in $ZFC$; thus everything is a set, including formulas. Whenever $T$ is a complete countable theory, we will suppose that $T$ comes equipped with an injection from the symbols of $T$ into $\omega$. In particular, if $V$ is a transitive model of $ZFC^-$ with $T \in V$, then $T$ is countable in $V$. The advantage of this is that our theorems become easier to state; without this convention we would just insert the hypothesis that $T$ is countable in $V$ everywhere. 

\begin{example}
	Suppose $M, N$ are $\mathcal{L}$-structures. The following are equivalent, by an alternating chains argument:
	\begin{itemize}
		\item $M \equiv N$.
		\item For some transitive $V \models ZFC^-$ with $M, N \in V$, there is some $\mathbf{j}: V \preceq \hat{V}$ such that $\mathbf{j}_{\std}(M) \cong \mathbf{j}_{\std}(N)$.
		\item For every transitive $V \models ZFC^-$ with $M, N \in V$, there is some $\mathbf{j}: V \preceq \hat{V}$ such that $\mathbf{j}_{\std}(M) \cong \mathbf{j}_{\std}(N)$.
	\end{itemize}
	
	Note this is a baby version of the Keisler-Shelah theorem \cite{KeislerShelah}, which says we can in fact arrange $\hat{V} = V^\lambda/\mathcal{U}$, for some $\lambda$ and some ultrafilter $\mathcal{U}$ on $\mathcal{P}(\lambda)$.
\end{example}

\vspace{2 mm}

We can phrase $\trianglelefteq^*_{\lambda \kappa}$ in terms of models of $ZFC^-$ as follows. Note that the assertion ``$\hat{V}$ is $\omega$-nonstandard" is redundant except when $\kappa = 1$, as otherwise it follows from $\kappa$-saturation.

\begin{lemma}\label{SatAndPseudo0}
	Suppose $\lambda$ is an infinite cardinal, $\kappa$ is an infinite cardinal or $1$, and $T_0, T_1$ are complete countable theories. Then the following are equivalent:
	\begin{itemize}
		\item[(A)] $T_0 \trianglelefteq^*_{\lambda \kappa} T_1$;
		\item[(B)] There is some countable transitive $V \models ZFC^-$ with $T_0, T_1 \in V$, and some $M_i \models T_i$ both in $V$, such that whenever $\mathbf{j}:V \preceq \hat{V}$, if $\hat{V}$ is $\kappa$-saturated and $\omega$-nonstandard and if $\mathbf{j}_{\std}(M_1)$ is $\lambda^+$-saturated, then $\mathbf{j}_{\std}(M_0)$ is $\lambda^+$-saturated.
	\end{itemize}
\end{lemma}
\begin{proof}
	(A) implies (B): Suppose $T_*$ and interpretations $I_0, I_1$ witness that $T_0 \trianglelefteq^*_{\lambda \kappa} T_1$. Let $M_* \models T_*$. Choose a countable transitive $V \models ZFC^-$ with $T_*, M_*, T_0, T_1, I_0, I_1 \in V$. This works, because if $\mathbf{j}:V \preceq \hat{V}$ is such that $\hat{V}$ is $\kappa$-saturated, then $\mathbf{j}_{\std}(M_*)$ is $\kappa$-saturated, and $\mathbf{j}_{\std}(M_i) \cong I_i(\mathbf{j}_{\std}(M_*))$.
	
	(B) implies (A): Let $T_*$ be the elementary diagram of $V$. We get natural interpretations  $I_i$ of $T_i$ in $T_*$, using the constant symbols for $M_i$.  Then this witnesses $T_0 \trianglelefteq^*_{\lambda \kappa} T_1$.
\end{proof}

\begin{remark}
	In the definition of $\trianglelefteq^*_\kappa$, we don't know of any cases where the particular choice of $V$ matters. If $\kappa$ is infinite, we also don't know if it matters if we allow $V$ to be of arbitrary cardinality. If $\kappa = 1$, then this does matter, see Remark~\ref{LanguageCardRemark}.
\end{remark}

A serious annoyance in dealing with the interpretability orders is that the choices of $M_0, M_1 \in V$ in Lemma~\ref{SatAndPseudo0} are often delicate; this is in contrast with Keisler's order, where Keisler's fundamental theorem \cite{Keisler} says that the choice of $M$ does not matter. We recapture this phenomenon by restricting to pseudofinite partial types. Some notation:

\begin{definition}
	Suppose $V \models ZFC^-$ is transitive, and $\mathbf{j}: V \preceq \hat{V}$. Say that $X \subseteq \hat{V}$ is pseudofinite (with respect to $\hat{V}$) if there is some $\hat{X} \in \hat{V}$ finite in the sense of $\hat{V}$, with $X \subseteq \hat{X}$. So if $\hat{X} \in \hat{V}$, then $\hat{X}$ is pseudofinite if and only if it is finite in the sense of $\hat{V}$.
\end{definition}

The following is the fundamental consequence of regularity of ultrafilters that are needed for Keisler's order. For the converse, we don't know if the assumption that $|V| \geq \lambda$ is necessary.

\begin{theorem}\label{RegularUltsThm1}
	Suppose $\mathcal{U}$ is a $\lambda$-regular ultrafilter on $\mathcal{P}(\lambda)$, $V \models ZFC^-$ is transitive. Then every subset of $V^\lambda/\mathcal{U}$ of size at most $\lambda$ is pseudofinite. 
	
	Conversely, if $|V| \geq \lambda$ and every subset of $V^\lambda/\mathcal{U}$ of size at most $\lambda$ is pseudofinite, then $\mathcal{U}$ is $\lambda$-regular.
\end{theorem}
\begin{proof}
	Write $\hat{V} := V^\lambda/\mathcal{U}$, and let $\mathbf{j}: V \preceq \hat{V} := V^\lambda/\mathcal{U}$ is the {\L}o{\'s} embedding.
	
	First, suppose $\mathcal{U}$ is $\lambda$-regular, and $X \subseteq \hat{V}$ has size at most $\lambda$; enumerate $X = \{\hat{x}_\alpha: \alpha< \lambda\}$ (with repetitions if necessary). Let $(A_\alpha: \alpha < \lambda)$ be a regular family from $\mathcal{U}$, i.e. the intersection of every infinite subset of $\overline{A}$ is empty.
	
	Write each $\hat{x}_\alpha = [f_\alpha/\mathcal{U}]$ where $f_\alpha: \lambda \to V$. We define $g: \lambda \to [V]^{<\aleph_0}$. Namely, let $\gamma< \lambda$ be given; let $I_\gamma$ be the finite set of all $\alpha < \lambda$ with $\gamma \in A_\alpha$. Let $g(\gamma) = \{f_\alpha(\gamma): \alpha \in I_\gamma\}$. Clearly this works.
	
	Conversely, suppose $|V| \geq \lambda$ and every subset of $\hat{V}$ of size at most $\lambda$ is pseudofinite. Choose $(a_\alpha: \alpha < \lambda)$ a sequence of distinct elements of $V$. Choose $\hat{X} \in \hat{V}$, finite in the sense of $\hat{V}$, such that each $\mathbf{j}(a_\alpha) \in \hat{X}$. Write $\hat{X} = (X_\beta: \beta < \lambda)/\mathcal{U}$, where each $X_\alpha \in [V]^{<\aleph_0}$. For each $\alpha < \lambda$, let $A_\alpha = \{\beta < \lambda: a_\alpha \in X_\beta\}$. Then each $A_\alpha \in \mathcal{U}$. Further, given $I \subseteq \lambda$ infinite, $\bigcap_{\alpha \in I} A_\alpha = \emptyset$, since if $\beta$ were in the intersection then we would have $\{a_\alpha: \alpha \in I\} \subseteq X_\beta$, contradicting that $X_\beta$ is finite.
\end{proof}

The following characterization of pseudofinite types will be used frequently. 

\begin{lemma}\label{PseudoTypeChar}
	Suppose $V \models ZFC^-$ is transitive, and $\mathbf{j}: V \preceq \hat{V}$, and $M$ is an $\mathcal{L}$-structure in $V$ for some language $\mathcal{L}$ which is countable in $V$. Suppose $\hat{V}$ is $\omega$-nonstandard, and $p(\overline{x})$ is a partial type over $\mathbf{j}_{\std}(M)$. Then $p(\overline{x})$ is pseudofinite if and only if there is some pseudofinite $X \subseteq \mathbf{j}_{\std}(M)$ such that $p(\overline{x})$ is over $X$.
\end{lemma}
\begin{proof}
	Suppose $p(\overline{x})$ is pseudofinite; let $X$ be the set of all parameters used in $p(\overline{x})$. We wish to show $X$ is pseudofinite. By hypothesis we can find some set $\hat{\Delta} \in \hat{V}$ finite in the sense of $\hat{V}$, with $p(\overline{x}) \subseteq \hat{\Delta}$; we can suppose each element of $\hat{V}$ is a formula over $\mathbf{j}(M)$ in the sense of $\hat{V}$. Then define $\hat{X} \in \hat{V}$ to be the set of all elements of $\mathbf{j}(M)$ which occur as a parameter in a formula in $\hat{\Delta}$. Since $\hat{\Delta}$ is finite in $\hat{V}$, so is $\hat{X}$.
	
	Conversely, suppose $p(\overline{x})$ is a partial type over $X$ with $X$ pseudofinite. Let $\overline{z} = (z_i: i < \omega)$ be variables and let $(\psi_n(\overline{x}, \overline{z}): n < \omega)$ enumerate all $\mathcal{L}$-formulas in these variables; we can choose the enumeration. After rearranging, we can suppose each $\psi_n(\overline{x}, \overline{z})$ only uses the variables $(\overline{x}, \overline{z}_n)$, where $\overline{z}_n =(z_i: i < n)$. Moreover, we can arrange that the sequence $(\psi_n(\overline{x}, \overline{z}): n < \omega)$ is in $V$. Let $(\hat{\psi}_{\hat{n}}(\overline{x}, \hat{\overline{z}}_{\hat{n}}): \hat{n}< \omega) = \mathbf{j}(\psi_n(\overline{x}, \overline{z}_n): n < \omega)$.  Choose $\hat{n} \in \hat{\omega}$ nonstandard, and let $\hat{\Delta}= \{\hat{\psi}_{\hat{m}}(\overline{x}, {\overline{a}}): \hat{m} < \hat{n}, \overline{a} \in \mathbf{j}_{\std}(M)^{\hat{m}}\}$. So $\hat{\Delta} \in \hat{V}$ is pseudofinite, and $p(\overline{x}) \subseteq \hat{\Delta}$.
\end{proof}

We now make the key definition of the paper, motivated by Theorem~\ref{RegularUltsThm1} and also (retrospectively) by Theorem~\ref{ModKeislerOrder}.

\begin{definition}
	Suppose $V \models ZFC^-$ is transitive, and $\mathbf{j}: V \preceq \hat{V}$ and $\kappa$ is an infinite cardinal.  If $M$ is a structure in $V$, then say that $\mathbf{j}_{\std}(M)$ is $\kappa$-pseudosaturated if for every pseudofinite $A \subseteq \mathbf{j}_{\std}(M)$ with $|A| < \kappa$, and for every $n < \omega$, every type $p(\overline{x}) \in S^n(A)$ is realized in $\mathbf{j}_{\std}(M)$. 
\end{definition}

Note that in the definition of $\kappa$-pseudosaturation, it is enough to consider types $p(x)$ of arity $1$. Also, by Lemma \ref{PseudoTypeChar}, when $\kappa > \aleph_0$ it is equivalent to quantify over pseudofinite types $p(\overline{x})$ of cardinality less than $\kappa$ (when $\kappa = \aleph_0$, $\kappa$-pseudosaturation is the same as $\aleph_0$-pseudosaturation, since every finite set is pseudofinite).

\begin{example}
	Suppose $V \models ZFC^-$ is transitive, and $\mathbf{j}: V \preceq \hat{V}$ is $\omega$-standard. Then pseudofinite subsets of $\hat{V}$ are the same as finite subsets of $\hat{V}$, so whenever $M \in V$ is $\aleph_0$-saturated, $\mathbf{j}_{\std}(M)$ is $\lambda^+$-pseudosaturated for all $\lambda$. But we will always exclude the case where $\hat{V}$ is $\omega$-standard.
\end{example}

The following theorem, combined with Theorem~\ref{RegularUltsThm1}, gives a new proof of Keisler's fundamental Theorem~\ref{KeislerOrigSecond}. We will eventually prove the corresponding statement for $\kappa = \aleph_0$ as well (Corollary~\ref{ModKeislerOrder}).

\begin{theorem}\label{ModKeislerOrder}
	Suppose $V \models ZFC^-$ is transitive and $\mathbf{j}: V \preceq \hat{V}$ is $\omega$-nonstandard, and $T \in V$ is a complete countable theory, and $M_0, M_1 \in V$ are two models of $T$. Suppose $\kappa$ is an uncountable cardinal. Then $\mathbf{j}_{\std}(M_0)$ is $\kappa$-pseudosaturated if and only if $\mathbf{j}_{\std}(M_1)$ is.
\end{theorem}
\begin{proof}
	Per our convention, note that $T$ is countable in $V$. Suppose $M_1$ is $\lambda^+$-pseudosaturated; we show that $M_0$ is also. As remarked above, it suffices to consider types of arity $1$ (the only effect of this is to increase readability). 
	
	Let $p(x) = \{\phi_\alpha(x, \overline{a}_\alpha): \alpha < \lambda)\}$ be a pseudofinite type over $M_0$ of cardinality $\lambda < \kappa$; we show $p(x)$ is realized in $M_0$. Choose some pseudofinite $\hat{\Delta} \in \hat{V}$, such that $p(x) \subseteq \hat{\Delta}$. By separation in $\hat{V}$, we can suppose $\hat{\Delta} = \hat{\Delta}(x)$ is a set of $\mathbf{j}(\mathcal{L})$-formulas over $\mathbf{j}(M_0)$ in the free variable $x$. 
	
	Since $\hat{V}$ believes $\mathbf{j}(M_0) \equiv \mathbf{j}(M_1)$ (by elementarity of $\mathbf{j}: V \preceq \hat{V}$), we can find a set $\hat{\Gamma}(x)$ of $\mathbf{j}(\mathcal{L})$-formulas over $\mathbf{j}(M_1)$ in the variable $x$ and a bijection $\hat{f}: \hat{\Delta}(x) \to \hat{\Gamma}(x)$, such that the following are true in $\hat{V}$: 
	
	\begin{itemize}
		\item[(I)] For every $\hat{\phi}(x, \hat{\overline{a}}) \in \hat{\Delta}(x)$, $\hat{f}(\hat{\phi}(x, \hat{\overline{a}})) = \hat{\phi}(x, \hat{\overline{b}})$ for some $\hat{\overline{b}} \in \mathbf{j}(M_1)^{|\hat{\overline{a}}|}$;
		\item[(II)] For every $\hat{\Delta}_0(x) \subseteq \hat{\Delta}(x)$, $\mathbf{j}(M_0) \models \exists x \bigwedge \hat{\Delta}_0(x)$ if and only if $\mathbf{j}(M_1) \models \exists x \bigwedge \hat{f}[\hat{\Delta}_0(x)]$.
	\end{itemize}

	Let $q(x)$ be the image of $p(x)$ under $\hat{f}$. By (I), $q(x)$ is a set of $\mathcal{L}$-formulas over $\mathbf{j}_{\std}(M_1)$ in the free variable $x$. By (II), $q(x)$ is consistent. Visibly $q(x) \subseteq \hat{\Gamma}(x)$ is pseudofinite and has cardinality less than $\kappa$.
	
	Thus $q(x)$ has a realization $b \in \mathbf{j}_{\std}(M)$. Let $\hat{\Gamma}_0(x)$ be defined in $\hat{V}$ as the set of all $\hat{\phi}(x) \in \hat{\Gamma}$ such that $\mathbf{j}(M_1) \models \hat{\phi}(b)$. Let $\hat{\Delta}_0(x) = \hat{f}^{-1}[\hat{\Gamma}_0(x)]$. By (II), in $\hat{V}$, $\mathbf{j}(M_0) \models \exists x \bigwedge \hat{\Delta}_0(x)$, so we can find $a \in \mathbf{j}_{\std}(M_0)$ such that in $\hat{V}$, $\mathbf{j}(M) \models \hat{\Delta}_0(a)$. Since $p(x) \subseteq \hat{\Delta}_0(x)$ we conclude that $a$ realizes $p(x)$.
\end{proof}

For this reason, the following interpretability orders are more convenient to work with than $\trianglelefteq^*_\kappa$ when proving positive reductions. 

\begin{definition}
	
	Suppose $V \models ZFC^-$ is transitive, $\mathbf{j}: V \preceq \hat{V}$ is $\omega$-nonstandard, and suppose $T$ is a complete countable theory with $T \in V$. Suppose $\kappa$ is an uncountable cardinal (typically $\kappa = \lambda^+$ for some cardinal $\lambda$). Then say that $\hat{V}$ $\kappa$-pseudosaturates $T$ if for some or every $M \models T$ with $M \in V$, $\mathbf{j}_{\std}(M)$ is $\kappa$-pseudosaturated. (This also depends on the embedding $\mathbf{j}: V \preceq \hat{V}$; if there is ambiguity we will write $(\mathbf{j}, \hat{V})$ $\kappa$-pseudosaturates $T$.) 
	
	Suppose $\lambda$ is infinite and $\kappa$ is infinite or $1$. Then say that $T_0 \trianglelefteq^\times_{\lambda \kappa} T_1$ if there is some countable transitive $V \models ZFC^-$ containing $T_0, T_1$ such that whenever $\mathbf{j}: V \preceq \hat{V}$, if $\hat{V}$ is $\kappa$-saturated and $\omega$-nonstandard, and if $\hat{V}$ $\lambda^+$-pseudosaturates $T_1$, then also it $\lambda^+$-pseudosaturates $T_0$. Say that $T_0 \trianglelefteq^\times_\kappa T_1$ if $T_0 \trianglelefteq^\times_{\lambda \kappa} T_1$ for all infinite $\lambda$.
\end{definition}

As for $\trianglelefteq^*_\kappa$, the two main cases of interest are $\kappa \in \{1, \aleph_1\}$. Note that for $\kappa < \kappa'$, $\trianglelefteq^\times_\kappa \subseteq \trianglelefteq^\times_{\kappa'}$. Corollary~\ref{SatAndPseudoCor} states that $\trianglelefteq^\times_{\kappa} \subseteq \trianglelefteq^*_{\kappa}$, and we suspect equality holds. 

 Also, note that $\trianglelefteq^{\times}_{\lambda \kappa}$ is transitive: Suppose $T_0 \trianglelefteq^{\times}_{\lambda \kappa} T_1 \trianglelefteq^\times_{\lambda \kappa} T_2$, and $T_0, T_1 \in V_0$, $T_1, T_2 \in V_1$ are countable transitive models of $ZFC^-$ witnessing this. Then choose a countable transitive model $V$ of $ZFC^-$ with $V_0, V_1 \in V$; easily, this witnesses $T_0 \trianglelefteq^\times_\kappa T_2$.

%

We now note that $\trianglelefteq^\times_{\aleph_1} \subseteq \trianglelefteq$. This will eventually follow from $\trianglelefteq^\times_{\aleph_1} \subseteq \trianglelefteq^*_{\aleph_1}$, seeing as $\trianglelefteq^*_{\aleph_1} \subseteq \trianglelefteq$, but there is no reason to wait:

\begin{corollary}\label{KeislerAndInterpUlt1}
	$\trianglelefteq^\times_{\aleph_1} \subseteq \trianglelefteq$.
\end{corollary}
\begin{proof}
	Suppose $T_0 \trianglelefteq^\times_{\lambda \aleph_1} T_1$, as witnessed by $V \models ZFC^-$ countable and transitive with $T_0, T_1 \in V$. Let $\mathcal{U}$ be any $\lambda$-regular ultrafilter on $\mathcal{P}(\lambda)$; write $\hat{V} = V^\lambda/\mathcal{U}$ and let $\mathbf{j}: V \preceq \hat{V}$ be the {\L}o{\'s} map. Then $\hat{V}$ is $\aleph_1$-saturated, since $\mathcal{U}$ is $\aleph_1$-good (since it is $\aleph_1$-incomplete). Hence, if $\hat{V}$ $\lambda^+$-pseudosaturates $T_1$, then $\hat{V}$ $\lambda^+$-pseudosaturates $T_0$. We conclude by Theorem~\ref{RegularUltsThm1}.
\end{proof}

\section{Full Boolean-Valued Models}\label{FullBVModelsSec}
We recall the setup of \cite{BVModelsUlrich}.

As a convention, if $X$ is a set and $\mathcal{L}$ is a language, then $\mathcal{L}(X)$ is the set of formulas of $\mathcal{L}$ with parameters taken from $X$.

Suppose $\mathcal{B}$ is a complete Boolean algebra. A $\mathcal{B}$-valued structure is a pair $(\mathbf{M}, \| \cdot \|_{\mathbf{M}})$ where:

\begin{enumerate}
	\item $\mathbf{M}$ is a set;
	\item  $\phi \mapsto \|\phi\|_{\mathbf{M}}$ is a map from $\mathcal{L}( \mathbf{M})$ to $\mathcal{B}$;
	\item If $\phi$ is a logically valid sentence then $\|\phi\|_{\mathbf{M}} = 1$;
	\item For every formula $\phi \in \mathcal{L}( \mathbf{M})$, we have that $\|\lnot \phi\|_{\mathbf{M}}= \lnot\|\phi\|_{\mathbf{M}}$;
	\item For all $\phi, \psi$, we have that $\|\phi \land \psi\|_{\mathbf{M}} = \|\phi\|_{\mathbf{M}} \land \|\psi\|_{\mathbf{M}}$;
	\item For every formula $\phi(x)$ with parameters from $\mathbf{M}$, $\|\exists x \phi(x)\|_{\mathbf{M}} = \bigvee_{a \in \mathbf{M}} \|\phi(a)\|_{\mathbf{M}}$;
	\item For all $a, b \in \mathbf{M}$ distinct, $\|a = b\|_{\mathbf{M}} < 1$.
\end{enumerate}

We are only interested in the case when $\mathbf{M}$ is full, i.e. when in fact  $\|\exists x \phi (x, \overline{a}) \|_{\mathbf{M}} = \mbox{max}_{a \in \mathbf{M}} \|\phi(a, \overline{a})\|_{\mathbf{M}}$. If $T$ is a theory, then we say that $\mathbf{M}$ is a full $\mathcal{B}$-valued model of $T$, if $\|\phi\|_{\mathbf{M}} = 1$ for all $\phi\in T$.

For example, (ordinary) $\mathcal{L}$-structures are the same as full $\{0, 1\}$-valued $\mathcal{L}$-structures, which can thus be viewed as full $\mathcal{B}$-valued structures for any $\mathcal{B}$. Also, if $M$ is an $\mathcal{L}$-structure and $\lambda$ is a cardinal, then $M^\lambda$ is a $\mathcal{P}(\lambda)$-valued $\mathcal{L}$-structure; moreover, we have the canonical elementary embedding $\mathbf{i}: M \preceq M^\lambda$, given by the diagonal map. We call this the pre-{\L}o{\'s} embedding. More generally, for any complete Boolean algebra $\mathcal{B}$ we can define the $\mathcal{B}$-valued structure $M^{\mathcal{B}}$.

If $\mathbf{M}$ is a full $\mathcal{B}$-valued model of $T$ and $\mathcal{U}$ is an ultrafilter on $\mathcal{B}$, then we can form the specialization $\mathbf{M}/\mathcal{U} \models T$, which comes equipped with a canonical surjection $[\cdot]_{\mathcal{U}}: \mathbf{M}\to \mathbf{M}/\mathcal{U}$, satisfying that for all $\phi(\overline{a}) \in \mathcal{L}(\mathbf{M})$, $\mathbf{M}/\mathcal{U}\models \phi([\overline{a}]_{\mathcal{U}})$ if and only if $\|\phi(\overline{a})\|_{\mathbf{M}} \in \mathcal{U}$. This generalizes the ultrapower construction $M^\lambda/\mathcal{U}$; note that the {\L}o{\'s} embedding of $M$ into $M^\lambda/\mathcal{U}$ is the composition of the pre-{\L}o{\'s} embedding with $[\cdot]_{\mathcal{U}}$.

In \cite{BVModelsUlrich}, we prove the following compactness theorem for  full Boolean-valued models:

\begin{theorem}\label{Compactness}
Suppose $\mathcal{B}$ is a complete Boolean algebra, $X$ is a set, $\Gamma \subseteq \mathcal{L}( X)$, and $F_0, F_1: \Gamma \to  \mathcal{B}$ with $F_0(\phi(\overline{a})) \leq F_1(\phi(\overline{a}))$ for all $\phi(\overline{a}) \in \Gamma$. Then the following are equivalent:

\begin{itemize}
	\item[(A)] There is some full $\mathcal{B}$-valued structure $\mathbf{M}$ and some map $\tau: X \to \mathbf{M}$, such that for all $\phi(\overline{a}) \in \Gamma$, $F_0(\phi(\overline{a})) \leq \|\phi(\tau(\overline{a}))\|_{\mathbf{M}} \leq F_1(\phi(\overline{a}))$;

	\item[(B)] For every finite $\Gamma_0 \subseteq \Gamma$ and for every $\mathbf{c} \in \mathcal{B}_+$, there is some $\{0, 1\}$-valued $\mathcal{L}$-structure $M$ and some map $\tau: X \to M$, such that for every $\phi(\overline{a}) \in \Gamma$, if $\mathbf{c} \leq F_0(\phi(\overline{a}))$ then $M \models \phi(\tau(\overline{a}))$, and if $\mathbf{c} \leq \lnot F_1(\phi(\overline{a}))$ then $M \models \lnot \phi(\tau(\overline{a}))$. 
\end{itemize}
\end{theorem}

Given $\mathcal{B}$-valued models $\mathbf{M}\subseteq \mathbf{N}$, say that $\mathbf{M} \preceq \mathbf{N}$ if $\|\cdot\|_{\mathbf{M}} \subseteq \| \cdot \|_{\mathbf{N}}$. Say that $\mathbf{N}$ is $\lambda^+$-saturated if for every $\mathbf{M}_0 \preceq \mathbf{N}$ with $|\mathbf{M}_0| \leq \lambda$ and for every $\mathbf{M}_1 \succeq \mathbf{M}_0$ with $|\mathbf{M}_1| \leq \lambda$, there is some elementary embedding $f: \mathbf{M}_1 \preceq \mathbf{N}$ extending the inclusion from $\mathbf{M}_0$ into $\mathbf{N}$. As a first application of the compactness theorem, we prove in \cite{BVModelsUlrich} that for every $\mathcal{B}$-valued structure $\mathbf{M}$ and for every $\lambda$, there is an elementary extension $\mathbf{N} \succeq \mathbf{M}$ such that $\mathbf{N}$ is full and moreover $\lambda^+$-saturated.

Suppose $T$ is a complete countable theory, and $\mathcal{U}$ is an ultrafilter on the complete Boolean algebra $\mathcal{B}$. We observe in \cite{BVModelsUlrich} that if there is some $\lambda^+$-saturated $\mathbf{M} \models^{\mathcal{B}} T$ with $\mathbf{M}/\mathcal{U}$ $\lambda^+$-saturated, then for every $\lambda^+$-saturated $\mathbf{M} \models^{\mathcal{B}} T$, $\mathbf{M}/\mathcal{U}$ is $\lambda^+$-saturated. We define that $\mathcal{U}$ $\lambda^+$-saturates $T$ in this case. 

\begin{example}
Suppose $\lambda$ is an infinite cardinal, $T$ is a complete countable theory, and $M \models T$ is $\lambda^+$-saturated. Then $M^\lambda$ is a $\lambda^+$-saturated $\mathcal{P}(\lambda)$-valued model of $T$. Thus, if $\mathcal{U}$ is a $\lambda$-regular ultrafilter on $\mathcal{P}(\lambda)$, then the two definitions we have of $\mathcal{U}$ $\lambda^+$-saturating $T$ (the standard definition of Keisler \cite{Keisler}, and the new definition above) are both equivalent to: $M^\lambda/\mathcal{U}$ is $\lambda^+$-saturated, for some or every $\lambda^+$-saturated $M \models T$. Thus we have not created any conflicts with this definition.

\end{example}

More generally, we show in \cite{BVModelsUlrich} that whenever $\lambda$ is an infinite cardinal and whenever $\mathcal{U}$ is an ultrafilter on the complete Boolean algebra $\mathcal{B}$, then $\mathcal{U}$ $\lambda^+$-saturates $T$ if and only if $\mathcal{U}$ is $(\lambda, \mathcal{B}, T)$-moral, in the sense of Malliaris and Shelah.

Finally, in \cite{BVModelsUlrich} we give the following convenient characterization of Keisler's order:

\begin{theorem}\label{KeislerChar}
	Suppose $T_0, T_1$ are theories. Then $T_0 \trianglelefteq T_1$ if and only if for every $\lambda$, for every complete Boolean algebra $\mathcal{B}$ with the $\lambda^+$-c.c., and for every ultrafilter $\mathcal{U}$ on $\mathcal{B}$, if $\mathcal{U}$ $\lambda^+$-saturates $T_1$, then $\mathcal{U}$ $\lambda^+$-saturates $T_0$.
\end{theorem}

We now adapt this to our context.

\noindent \textbf{Convention.} $\mathbf{V}$ will denote a full $\mathcal{B}$-valued model of $ZFC^-$ for some complete Boolean algebra $\mathcal{B}$, often associated with an elementary embedding $\mathbf{i}: V \preceq \mathbf{V}$ for some transitive $V$ (for example, if $\mathbf{V} = V^{\mathcal{B}}$ is the Boolean ultrapower, then $\mathbf{i}$ would be the pre-{\L}o{\'s} embedding).

\begin{definition}
	Suppose $\mathcal{B}$ is a complete Boolean algebra, $V \models ZFC^-$ is transitive, and $\mathbf{i}: V \preceq \mathbf{V} \models^{\mathcal{B}} ZFC^-$. Given $X \subseteq V$, let $\mathbf{i}_{\std}(X) = \{\mathbf{a}\in V: \|\mathbf{a} \in \mathbf{i}(X)\|_{\mathcal{B}}\} = 1$. Suppose $M \in V$ is a structure in the countable language $\mathcal{L}$, with domain $\mbox{dom}(M)$. Then let $\mathbf{i}_{\std}(M)$ be the full $\mathcal{B}$-valued $\mathcal{L}$-structure defined as follows. Its domain is $\mathbf{i}_{\std}(\mbox{dom}(M))$. Given a formula $\phi(\mathbf{a}_i: i < n) \in \mathcal{L}(\mathbf{i}_{\std}(\mbox{dom}(M)))$, let $\|\phi(\mathbf{a}_i: i < n)\|_{\mathbf{i}_{\std}(M)} = \|\mathbf{i}(M) \models \phi(\mathbf{a}_i: i < n)\|_{\mathbf{V}}$. (Note that in practice, we usually denote $M$ and $\mbox{dom}(M)$ by the same symbol $M$.)
\end{definition}

\begin{example}
	Suppose $V \models ZFC^-$ is transitive and $\lambda$ is a cardinal; write $\mathbf{V} = V^\lambda$ and let $\mathbf{i}: V \preceq \mathbf{V}$ be the pre-{\L}o{\'s} embedding and let $\mathbf{j}: V \preceq \mathbf{V}/\mathcal{U}$ be the composition $[\cdot]_{\mathcal{U}} \circ \mathbf{i}$, i.e. the {\L}o{\'s} embedding. Suppose $M \in V$. Then $\mathbf{i}_{\std}(M) = M^\lambda$. If $\mathcal{U}$ is an ultrafilter on $\mathcal{P}(\lambda)$, then $\mathbf{j}_{\std}(M) \cong M^\lambda/\mathcal{U} = \mathbf{i}_{\std}(M)/\mathcal{U}$.
\end{example}

\begin{example}
	More generally, suppose $\mathcal{U}$ is an ultrafilter on $\mathcal{B}$, and $V \models ZFC^-$ is transitive, and $\mathbf{i}: V \preceq \mathbf{V}$ where $\mathbf{V} \models^{\mathcal{B}} ZFC^-$. Let $\mathbf{j}: V \preceq \mathbf{V}/\mathcal{U}$ be the composition $[\cdot]_{\mathcal{U}} \circ \mathbf{i}$. Then given $M \models T$ with $M \in V$, $\mathbf{j}_{\std}(M) \cong \mathbf{i}_{\std}(M)/\mathcal{U}$. This is because if $\mathbf{a} \in \mathbf{V}$ and $\| \mathbf{a} \in \mathbf{i}(M) \|_{\mathbf{V}} \in \mathcal{U}$, then we can find $\mathbf{b} \in \mathbf{i}_{\std}(\mathbf{M})$ such that $\| \mathbf{a} = \mathbf{b} \|_{\mathbf{V}} \in \mathcal{U}$. 
\end{example}

We aim for a characterization of Keisler's order along these lines. The following two lemmas are typical applications of Theorem~\ref{Compactness}.

\begin{lemma}\label{SatAndPseudo3First}
	Suppose $\mathcal{B}$ is a complete Boolean algebra, $V \models ZFC^-$ and $\mathbf{i}: V \preceq \mathbf{V} \models^{\mathcal{B}} ZFC^-$. Suppose $M \in V$ is a structure. If $\mathbf{V}$ is $\lambda^+$-saturated, then so is $\mathbf{i}_{\std}(M)$. 
\end{lemma}
\begin{proof}
Write $\mathbf{M} = \mathbf{i}_{\std}(M)$. Let $\mathbf{M}_0 \preceq \mathbf{M}$ have $|\mathbf{M}_0| \leq \lambda$, and let $\mathbf{M}_1 \succeq \mathbf{M}_0$ have $|\mathbf{M}_1| \leq \lambda$; we suppose that $\mathbf{M}_1 \cap \mathbf{V} = \mathbf{M}_0$. In the notation of Theorem~\ref{Compactness}, let $X = \mathbf{V} \cup \mathbf{M}_1$. Let $\Gamma = \mathcal{L}(\mathbf{V}) \cup \{``a \in \mathbf{i}(M)": a \in \mathbf{M}_1\} \cup \{``\mathbf{i}(M) \models \phi(\overline{a})": \phi(\overline{a}) \in \mathcal{L}(\mathbf{M}_1)\}$. Define $F: \Gamma \to \mathcal{B}$ via $F \restriction_{\mathcal{L}(\mathbf{V})} = \|\cdot\|_{\mathbf{V}}$, and $F(``a \in \mathbf{i}(M)") = 1$, and $F(``\mathbf{i}(M) \models \phi(\overline{a})") = \|\phi(\overline{a})\|_{\mathbf{M}_1}$. 

We wish to apply Theorem~\ref{Compactness} with $F_0 = F_1 = F$. We show that (B) there holds. Indeed, let $\mathbf{c} \in \mathcal{B}_+$ and let $\Gamma_0 \subseteq \Gamma$ be given; let $\mathcal{U}$ be an ultrafilter on $\mathcal{B}$ with $\mathbf{c} \in \mathcal{U}$. Let $\phi_0(\overline{a}_0, \overline{b}_0), \ldots, \phi_{n-1}(\overline{a}_{n-1}, \overline{b}_{n-1})$ list all of the $\mathcal{L}(\mathbf{M}_1)$-formulas $\phi(\overline{a}, \overline{b})$ such that $``\mathbf{i}(M) \models \phi(\overline{a}, \overline{b})" \in \Gamma_0$, with the variables partitioned so that $\overline{a} \subseteq \mathbf{M}_0$ and $\overline{b} \subseteq \mathbf{M}_1 \backslash \mathbf{M}_0$. Let $\overline{b}$ be an enumeration of $(\overline{b}_i: i < i_*)$ without repetitions, and write each $\phi_i$ as $\phi_i(\overline{a}_i, \overline{b})$. We can suppose $i_* \leq n$ is such that $\mathbf{M}_1/\mathcal{U} \models \phi_i([\overline{a}_i]_{\mathcal{U}}, [\overline{b}]_{\mathcal{U}})$ if and only if $i < i_*$. 

Since $\mathbf{M}/\mathcal{U} \preceq \mathbf{M}_1/\mathcal{U}$, we can find $\overline{c} \in \mathbf{M}$ such that $\mathbf{M}/\mathcal{U} \models \bigwedge_{i < i_*} \phi_i([\overline{a}_i]_{\mathcal{U}}, [\overline{c}]_{\mathcal{U}})$. 

Write $W = \mathbf{V}/\mathcal{U}$, and choose $\tau: X \to W$ so that $\tau \restriction_{\mathbf{V}} = [\cdot]_{\mathcal{U}}$ is the canonical surjection, and $\tau(\overline{b}) = (\overline{c})$, and $\tau[\mathbf{M}_1] \subseteq \mathbf{M}$. Then $\tau$ witnesses Theorem~\ref{Compactness}(B) holds for $\Gamma, F_0, F_1$.

Thus Theorem~\ref{Compactness}(A) holds for $\Gamma, F_0, F_1$. By $\lambda^+$-saturation of $\mathbf{V}$, this means we can find $\tau: \mathbf{M}_1 \to \mathbf{M}$ which is the identity on $\mathbf{M}_0$ such that for all $\phi(\overline{a}) \in \mathcal{L}(\mathbf{M}_1)$, $\|\phi(\overline{a})\|_{\mathbf{M}_1} = \|\mathbf{i}(M) \models \phi(\overline{a})\|_{\mathbf{V}}$. This means exactly that $\tau: \mathbf{M}_1 \preceq \mathbf{M}$, as desired.
\end{proof}

\begin{lemma}\label{SurveyInterpBVLarge}
	Suppose $\mathcal{B}$ is a complete Boolean algebra, $V \models ZFC^-$ is transitive and $\mathbf{i}: V \preceq \mathbf{V}$. Suppose $\mathcal{U}$ is an ultrafilter on $\mathcal{B}$; write $\hat{V} = \mathbf{V}/\mathcal{U}$. If $\mathbf{V}$ is $\lambda^+$-saturated, then every subset of $\hat{V}$ of size at most $\lambda$ is pseudofinite.
\end{lemma}
\begin{proof}
	Suppose $X \in [\hat{V}]^{\leq \lambda}$. Write $X = \{[\mathbf{a}]_{\mathcal{U}}: \mathbf{a} \in \mathbf{X}\}$ for some $\mathbf{X} \in [\mathbf{V}]^{\leq \lambda}$. Choose some full $\mathbf{V}_0 \preceq \mathbf{V}$ with $\mathbf{X} \subseteq \mathbf{V}_0$ and $|\mathbf{V}_0| \leq \lambda$ (this is possible since the axioms for full Boolean-valued models are finitary, see \cite{BVModelsUlrich}).
	
	We claim that Theorem~\ref{Compactness} implies there is $\mathbf{V}_1 \succeq \mathbf{V}_0$ and $\mathbf{Y} \in \mathbf{V}_1$ such that $\|\mathbf{a} \in \mathbf{Y}\|_{\mathbf{V}_1} = 1$ for each $\mathbf{a} \in \mathbf{X}$, and $\|\mathbf{Y} \mbox{ is finite}\|_{\mathbf{V}_1} = 1$. Indeed, let $X = \mathbf{V}_0 \cup \{\mathbf{Y}\}$, where $\mathbf{Y}$ is a new symbol, and let $\Gamma = \mathcal{L}(\mathbf{V}_0) \cup \{``\mathbf{a} \in \mathbf{Y}": \mathbf{a} \in \mathbf{X}\} \cup \{``\mathbf{Y} \mbox{ is finite}"\}$. We define $F_0 = F_1 = F: \Gamma \to \mathcal{B}$ via $F \restriction_{\mathcal{L}(\mathbf{V}_0)} = \|\cdot\|_{\mathbf{V}_0}$, and $F(``\mathbf{a} \in \mathbf{Y}") = F(``\mathbf{Y} \mbox{ is finite}") = 1$. Then clearly Theorem~\ref{Compactness} applies and gives $\mathbf{V}_1$ as desired.
	
	Since $\mathbf{V}$ is $\lambda^+$-saturated, we can in fact choose such $\mathbf{Y}\in \mathbf{V}$. Let $\hat{Y} = [\mathbf{Y}]_{\mathcal{U}}$. Then $\hat{Y}$ is pseudofinite and $X \subseteq \hat{Y}$.
\end{proof}

We immediately get the following corollary.

\begin{corollary}\label{KeislerAndInterp2}
	Suppose $\mathcal{B}$ is a complete Boolean algebra, $\mathcal{U}$ is an ultrafilter on $\mathcal{B}$, $\lambda$ is a cardinal, and $T$ is a complete countable theory. Then the following are equivalent:
	
	\begin{itemize}
		\item[(A)] $\mathcal{U}$ $\lambda^+$-saturates $T$.
		\item[(B)] For some or every transitive $V \models ZFC^-$ with $T \in V$, and for some or every $\mathbf{i}: V \preceq \mathbf{V}$ with $\mathbf{V}$ $\lambda^+$-saturated, and for some or every $M \models T$ with $M \in V$, $\mathbf{i}_{\std}(M)/\mathcal{U}$ is $\lambda^+$-saturated.
		\item[(C)] For some or every transitive $V \models ZFC^-$, and for some or every $\mathbf{i}: V \preceq \mathbf{V}$ with $\mathbf{V}$ $\lambda^+$-saturated, $\mathbf{V}/\mathcal{U}$ $\lambda^+$-pseudosaturates $T$.
	\end{itemize}
\end{corollary}
\begin{proof}
	(A) if and only if (B) is Lemma~\ref{SatAndPseudo3First}. (B) implies (C) is trivial. (C) implies (B) follows from Lemma~\ref{SurveyInterpBVLarge}.
\end{proof}

%
%
%

\section{Maximality of $SOP_2$}\label{SurveyMaxSOP2Sec}

In \cite{Keisler}, Keisler proved that $\mbox{Th}([\omega]^{<\aleph_0}, \subseteq)$ is $\trianglelefteq$-maximal. In \cite{ShelahIso}, Shelah showed that in fact every $SOP$ theory is maximal, in particular $\mbox{Th}(\mathbb{Q}, <)$ is maximal. Later in \cite{SH500}, Shelah improved this to show that every $SOP_3$ theory is maximal.

Then, in \cite{pEqualsTref}, Malliaris and Shelah proved that in fact every $SOP_2$ theory is maximal in Keisler's order; this is substantially harder than for $SOP_3$. In \cite{pEqualsT2}, they prove that every $SOP_2$ theory is maximal in $\trianglelefteq^*_1$, with some minor complications to the argument. 

In this section, we sketch some of the main concepts from \cite{pEqualsTref}, and prove that $SOP_2$ theories are maximal in $\trianglelefteq^\times_1$.

One large difference in our treatment versus \cite{pEqualsTref} is that there, Malliaris and Shelah use cofinality spectrum problems as their base set theory. This is a weak fragment of arithmetic, not even strong enough for exponentiation. We stick to $ZFC^-$, thus avoiding many difficulties.

We begin with some definitions.

\begin{definition}
	
	If $(L, <)$ is a linear order, and $\kappa, \theta$ are infinite regular cardinals, then a $(\kappa, \theta)$-pre-cut in $L$ is a pair of sequences $(\overline{a}, \overline{b}) = (a_\alpha: \alpha < \kappa)$, $(b_\beta: \beta < \theta)$ from $L$, such that for all $\alpha < \alpha'$, $\beta < \beta'$, we have $a_\alpha < a_{\alpha'} < b_{\beta'} < b_{\beta}$. $(\overline{a}, \overline{b})$ is a cut if there is no $c \in L$ with $a_\alpha < c < b_\beta$ for all $\alpha, \beta$. Let the cut spectrum of $(L, <)$ be $\mathcal{C}(L, <) := \{(\kappa, \theta): L \mbox{ admits a } (\kappa, \theta) \mbox{ cut}\}$. Define $\mbox{cut}(L, <) = \mbox{min}\{\kappa + \theta:  (\kappa, \theta) \in \mathcal{C}(L, <)\}$.

	Suppose $\hat{V}$ is an $\omega$-nonstandard model of $ZFC^-$. Then define $\mathcal{C}_{\hat{V}}= \mathcal{C}(\hat{\omega}, \hat{<})$, and define  $\mathfrak{p}_{\hat{V}} = \mbox{cut}(\hat{\omega}, \hat{<})$. Also, let $\mathfrak{t}_{\hat{V}}$ be the least $\kappa$ such that there is some $\hat{n}< \hat{\omega}$ and some increasing sequence $(\hat{s}_\alpha: \alpha < \kappa)$ from $\hat{n}^{<\hat{n}}$, with no upper bound in $\hat{n}^{<\hat{n}}$. 
\end{definition}

The following theorem corresponds to Theorem 4.1 of \cite{pEqualsTref}, although there the authors must also assume $\lambda < \mathfrak{t}_{\hat{V}}$, since in the context of cofinality spectrum problems it may be the case that $\mathfrak{p}_{\hat{V}} > \mathfrak{t}_{\hat{V}}$. Note that since models of $ZFC^-$ admit pairing functions, there is no loss in only considering types in a single variable, in which each formula has only a singleton parameter. For a simplified proof in the context of models of $ZFC^-$, see \cite{pEqualsTUlrich}.

\begin{theorem}\label{localSaturation}
	Suppose $\hat{V} \models ZFC^-$ is $\omega$-nonstandard.
	Suppose  $p(x)= (\phi_\alpha(x, \hat{a}_\alpha): \alpha < \lambda)$ is a partial type over $\hat{V}$ of cardinality $\lambda < \mathfrak{p}_{\hat{V}}$. Suppose $\hat{X} \in \hat{V}$ is pseudofinite, and $\phi_0(x)$ is $``x \in \hat{X}$." Then $p(x)$ is realized in $\hat{V}$.
\end{theorem}

For what we intend, the following tweaks will be more convenient.

\begin{theorem}\label{localSaturation2}
	Suppose $\hat{V} \models ZFC^-$ is $\omega$-nonstandard and $\mathfrak{p}_{\hat{V}} \geq \aleph_1$. Suppose $p(x)  = \{\phi_\alpha(x, \hat{a}_\alpha): \alpha < \lambda\}$ is a type over $\hat{V}$ of cardinality $\lambda < \mathfrak{p}_{\hat{V}}$, and suppose $\{\hat{a}_\alpha: \alpha < \lambda\}$ is pseudofinite. Then $p(x)$ is realized in $\hat{V}$, provided either of the following conditions are met.
	
	\begin{itemize}
		\item[(A)] There is some $n < \omega$ such that each $\phi_\alpha(x, a_\alpha)$ is $\Sigma_n$.
		
		\item[(B)] Every countable subset of $\hat{V}$ is pseudofinite.
	\end{itemize}
\end{theorem}
\begin{proof}
	The fork in the argument will come later.
	
	Choose $\hat{X} \in \hat{V}$ finite in the sense of $\hat{V}$, such that each $\hat{a}_\alpha \in \hat{X}$. For each $n$, let $\psi_n(x, y,z )$ be a truth predicate for $\Sigma_n$ formulas; that is, for all $\Sigma_n$-formulas $\phi(x, y)$, $ZFC^-$ proves 
	
	$$\forall x \forall y (\psi(x, y, \phi(x, y)) \leftrightarrow \phi(x, y))$$ 
	
	(We are assuming formulas $\phi(x, y)$ of set theory are $0$-definable uniformly in all models of $ZFC^-$, so that it makes sense to plug them in as parameters.)  
	
	Let $\hat{n}_* < \hat{\omega}$ be nonstandard, and let $\hat{\Gamma}_n$ denote the set of all $\hat{\Sigma}_n$-formulas of set theory in $\hat{V}$ of length less than $\hat{n}_*$ (so $\hat{\Gamma}_n$ contains all true $\Sigma_n$ formulas of set theory). Finally, choose some $\hat{Y}_n \in \hat{V}$ finite in the sense of $\hat{V}$, and such that the following holds in $\hat{V}$: for every $\hat{\overline{a}} \in \hat{X}^{<\hat{n}_*}$ and for every $\hat{\phi}(x, y) \in \hat{\Gamma}_n$, if there is some $\hat{b}$ such that $\psi_n(\hat{b}, \hat{\overline{a}}, \hat{\phi}(x, y))$ holds, then there is some such $\hat{b} \in \hat{Y}_n$. 
	
	For each $n < \omega$, let $I_n$ be the set of all $\alpha < \lambda$ such that $\phi_\alpha(x, a_\alpha)$ is $\Sigma_n$, and let $p_n(x)$ be $\{\phi_\alpha(x): \alpha \in I_n\}$. So $(p_n(x): n < \omega)$ is an ascending chain of types with union $p(x)$. 
	
	\vspace{1 mm}
	
	\noindent \textbf{Claim.} For each $n < \omega$, $p_n(x) \cup \{x \in \hat{Y}_n\}$ is consistent.
	
	\noindent \emph{Proof.} Suppose $s \in [I_n]^{<\aleph_0}$ is finite. Enumerate $s = \{\alpha_i: i < i_*\}$ in increasing order, and write $\hat{\overline{a}} = (\hat{a}_{\alpha_i}: i < i_*) \in \hat{X}^{i_*}$. Let $\phi(x, y)$ be the $\Sigma_n$ formula asserting that $y = (y_0, \ldots, y_{i_*-1})$ is a tuple and $\phi_{\alpha_i}(x, y_i)$ holds for each $i < i_*$. Since $p(x)$ is consistent and by choice of $\hat{Y}_n$, we can find $\hat{b} \in \hat{Y}_n$ satisfying $\phi(\hat{b}, \hat{\overline{a}})$, hence realizing $\{\phi_\alpha(x, \hat{a}_{\alpha_i}): i < i_*\} \cup \{x \in \hat{Y}_n\}$ as desired. \qed
	
	\vspace{1 mm}
	
	Now, suppose we are in Case (A), so every formula $\phi_\alpha(x, \hat{a}_\alpha)$ is $\Sigma_n$ for some fixed $n < \omega$. Then $p(x) \cup \{x \in \hat{Y}_n\}$ is consistent by the claim, so we are in the case of Theorem~\ref{localSaturation}.
	
	Finally, suppose we are in Case (B). Then $\{\hat{Y}_n: n < \omega\}$ is pseudofinite, hence we can find some $\hat{Y}_* \in \hat{V}$ finite in the sense of $\hat{V}$, such that each $\hat{Y}_n \in \hat{Y}_*$. We can suppose that each element of $\hat{Y}_*$ is finite in the sense of $\hat{V}$. Let $\hat{Y} = \bigcup \hat{Y}_*$; this is finite in the sense of $\hat{V}$, and each $\hat{Y}_n \subseteq \hat{Y}$. Further, $p(x) \cup \{x \in \hat{Y}\}$ is consistent by the claim. Thus we are once again in the case of Theorem~\ref{localSaturation}. 
\end{proof}

The following corollaries follow immediately. Corollary~\ref{SaturationCharacterization} is modeled after Theorem 4.7 of \cite{pEqualsT2}, which gives an analogous characterization of saturation for models of arithmetic. 

\begin{corollary}\label{LocalSat0}
	Suppose $V \models ZFC^-$ is transitive and $\mathbf{j}: V \preceq \hat{V}$ is $\omega$-nonstandard and $\mathfrak{p}_{\hat{V}} \geq \aleph_1$. Then for all complete, countable theories $T$, $\hat{V}$ $\mathfrak{p}_{\hat{V}}$-pseudosaturates $T$.
\end{corollary}

\begin{corollary}\label{SaturationCharacterization}
	Suppose $V \models ZFC^-$ is transitive and $\mathbf{j}: V \preceq \hat{V}$ is $\omega$-nonstandard and $\kappa$ is an uncountable cardinal. Then $\hat{V}$ is $\kappa$-saturated if and only if $\mathfrak{p}_{\hat{V}} \geq \kappa$ and every subset of $\hat{V}$ of size less than $\kappa$ is pseudofinite.
\end{corollary}

Finally, the following theorem is Central Theorem 9.1 of \cite{pEqualsTref} (except in the context of cofinality spectrum problems, only $\geq$ is necessarily true; Malliaris and Shelah prove $\leq$ for cofinality spectrum problems with exponentiation in \cite{pEqualsT2}). We give a streamlined proof in \cite{pEqualsTUlrich}.

\begin{theorem}\label{pvEqualsTv}
	Suppose $\hat{V} \models ZFC^-$. Then $\mathfrak{p}_{\hat{V}} = \mathfrak{t}_{\hat{V}}$.
\end{theorem}

We now proceed to apply this to $SOP_2$.

\begin{theorem}\label{SOP2Max}
	Suppose $V \models ZFC^-$ is transitive and $\mathbf{j}: V \preceq \hat{V}$ is $\omega$-nonstandard, and suppose $T \in V$ is a countable complete theory with $SOP_2$. Then $\hat{V}$ does not $\mathfrak{p}_{\hat{V}}^+$-pseudosaturate $T$.
\end{theorem}
\begin{proof}
	Write $\lambda = \mathfrak{p}_{\hat{V}} = \mathfrak{t}_{\hat{V}}$. Suppose towards a contradiction that $\hat{V}$ did $\lambda^+$-pseudosaturate $T$. Let $M \models T$ with $M \in V$. Let $\phi(\overline{x}, \overline{y})$ be a formula of $T$ with $SOP_2$.
	
	Let $\hat{n}_* < \hat{\omega}$ and let $(\hat{s}_\alpha: \alpha < \lambda)$ be an increasing sequence from $\hat{n}_*^{<\hat{n}_*}$ with no upper bound in $\hat{n}_*^{<\hat{n}_*}$. 
	
	By elementarity, we can choose $\hat{f}: \hat{n}_*^{< \hat{n}_*} \to \mathbf{j}_{\std}(M)^{\lg(\overline{y})}$ in $\hat{V}$, such that $(\hat{f}(\hat{s}): \hat{s} \in \hat{n}_*^{<\hat{n}_*})$ is as in the definition of $\phi(\overline{x}, \overline{y})$ being $SOP_2$.
	
	By $\lambda^+$-pseudosaturation of $\mathbf{j}_{\std}(M)$ we can choose $\overline{a} \in \mathbf{j}_{\std}(M)^{|\overline{x}|}$ such that $\mathbf{j}_{\std}(M) \models \phi(\overline{a}, \hat{f}(\hat{s}_\alpha))$ for each $\alpha < \kappa$. Working in $\hat{V}$, let $\hat{s}_*$ be the union of all $\hat{s} \in \hat{n}_*^{<\hat{n}_*}$ such that $\mathbf{j}_{\std}(M) \models \phi(\overline{a}, \hat{f}(\hat{s}))$. Then $\hat{s} \in \hat{n}_*^{<\hat{n}_*}$ is an upper bound to $(\hat{s}_\alpha: \alpha < \lambda)$, contradiction.
\end{proof}

We immediately obtain the following. Note that we will eventually show that $\trianglelefteq^\times_\kappa \subseteq \trianglelefteq^*_\kappa$ for all $\kappa$, and in particular we will have recovered Malliaris and Shelah's theorem \cite{pEqualsT2} that $SOP_2$ theories are maximal in $\trianglelefteq^*_1$.

\begin{corollary}\label{SOP2MaxStatement}
	Suppose $T$ has $SOP_2$. Then $T$ is maximal in $\trianglelefteq^\times_{1}$. Hence $T$ is maximal in $\trianglelefteq^\times_\kappa$ for all $\kappa$ infinite or $1$, and $T$ is maximal in $\trianglelefteq$.
\end{corollary}

Malliaris and Shelah use these results to give a characterization of $\lambda^+$-goodness among $\lambda$-regular ultrafilters on $\mathcal{P}(\lambda)$. We give a similar characterization for arbitrary Boolean algebras, by the same argument. Recall from \cite{BVModelsUlrich} that an ultrafilter $\mathcal{U}$ on $\mathcal{B}$ is $\lambda^+$-good if $\mathcal{U}$ $\lambda^+$-saturates every complete countable theory $T$; this is equivalent to $\mathcal{U}$ $\lambda^+$-saturating $\mbox{Th}([\omega]^{<\aleph_0}, \subseteq)$, and is also equivalent to: every $\lambda$-distribution in $\mathcal{U}$ has a multiplicative refinement in $\mathcal{U}$.

\begin{theorem}\label{Goodness2}
	Suppose $\mathcal{U}$ is an ultrafilter on the complete Boolean algebra $\mathcal{B}$ and $\lambda$ is a cardinal. Then the following are equivalent:
	
	\begin{itemize}
		\item[(A)] $\mathcal{U}$ is $\lambda^+$-good, i.e. $\mathcal{U}$ $\lambda^+$-saturates every complete countable theory.
		\item[(B)] For some or every transitive $V \models ZFC^-$, and for some or every $\mathbf{i}: V \preceq \mathbf{V}$ with $\mathbf{V}$ $\lambda^+$-saturated, $\lambda < \mathfrak{p}_{\mathbf{V}/\mathcal{U}}$ (i.e. $\lambda < \mathfrak{t}_{\mathbf{V}/\mathcal{U}}$).
		\item[(C)] $\mathcal{U}$ $\lambda^+$-saturates some countable $SOP_2$-theory.
	\end{itemize}
\end{theorem}
\begin{proof}
	In (B), write $\hat{V} = \mathbf{V}/\mathcal{U}$. 
	
	(A) implies (B): by Theorem \ref{KeislerAndInterp2}, $\hat{V}$ $\lambda^+$-pseudosaturates $\mbox{Th}(\omega, <)$. Hence $(\hat{\omega}, <)$ is $\lambda^+$-pseudosaturated, so $\lambda < \mathfrak{p}_{\hat{V}}$.
	
	(B) implies (C): by Theorem~\ref{KeislerAndInterp2} and Corollary~\ref{LocalSat0}.
	
	(C) implies (A): by Corollary~\ref{SOP2MaxStatement}. 
\end{proof}

In view of this theorem, the following definition makes sense:

\begin{definition}
Suppose $\mathcal{U}$ is an ultrafilter on $\mathcal{B}$. Then let $\mathfrak{p}_{\mathcal{U}} = \mathfrak{t}_{\mathcal{U}} = $ the least cardinal $\lambda$ such that $\mathcal{U}$ is $\lambda^+$-good, or $\infty$ if $\mathcal{U}$ is $\lambda^+$-good for all $\lambda$.
\end{definition}

We prove in \cite{InterpOrders2Ulrich} that if $\mathcal{U}$ is nonprincipal (i.e. $\bigwedge \mathcal{U} = 0$), then $\mathfrak{p}_{\mathcal{U}} \leq \mbox{c.c.}(\mathcal{B})^+$, so in particular $\mathfrak{p}_{\mathcal{U}} = \infty$ precisely when $\mathcal{U}$ is principal.

The following summarizes what else is known about the maximal class of Keisler's order and the interpretability orders. Results of D\v{z}amonja and Shelah \cite{SH692} and Shelah and Usvyatsov \cite{SOP2pt2} together show that if $T$ is $NSOP_2$, and if suitable instances of $GCH$ hold, then $T$ is not maximal in $\trianglelefteq^*_1$. Therefore, consistently $SOP_2$ characterizes maximality in $\trianglelefteq^*_1$; the pieces for all of this are put together in \cite{pEqualsT2}. Corollary~\ref{supersimpleCor3} of the present work states that this is also true of $\trianglelefteq^*_{\aleph_1}$. In \cite{InterpOrders}, Malliaris and Shelah prove in $ZFC$ that simple theories are not maximal in $\trianglelefteq^*_1$; Corollary~\ref{supersimpleCor4} of the present work states that this is also true of $\trianglelefteq^*_{\aleph_1}$. Malliaris and Shelah have shown  that if there is a supercompact cardinal, then simple theories are non-maximal in $\trianglelefteq$ \cite{Optimals}. I have shown that low theories are non-maximal in $\trianglelefteq$ \cite{LowDividingLine}.\footnote{There are two definitions of ``low" in circulation. We are sticking with Buechler's original definition \cite{Buechler} in terms of $D$-rank. In particular, this implies simplicity.}

\section{The interpretability orders in terms of pseudosaturation}\label{InterpOnStableSec0}

In this section, we begin to clarify the relationship between $\trianglelefteq^\times_\kappa$ and $\trianglelefteq^*_{\kappa}$. In particular, we show that if $T_1$ is unstable, and if $T_0 \trianglelefteq^\times_\kappa T_1$, then $T_0 \trianglelefteq^*_\kappa T_1$; by a case analysis on the stable theories, we will eventually eliminate the restriction on $T_1$. The key is to consider a certain expansion $ZFC^-_*$ of $ZFC$. These arguments are an abstraction of Malliaris and Shelah's proof in \cite{InterpOrders} that $T_{rg}$ is the $\trianglelefteq^*_{1}$-minimal unstable theory, where $T_{rg}$ is the theory of the random graph.

 The following definition should be thought of as analogous to the bounding number $\mathfrak{b}$ from cardinal characteristics of the continuum.

\begin{definition}
	Suppose $\hat{V} \models ZFC^-$. Then let $\mathfrak{b}_{\hat{V}}$ be the least cardinality of a non-pseudofinite subset of $\hat{V}$.
\end{definition}

For example, if $\hat{V} = V^\lambda/\mathcal{U}$ where $\mathcal{U}$ is a $\lambda$-regular ultrafilter on $\mathcal{U}$, then Theorem~\ref{RegularUltsThm1} states that $\mathfrak{b}_{\hat{V}} > \lambda$; similarly, if $\hat{V} = \mathbf{V}/\mathcal{U}$ where $\mathcal{U}$ is an ultrafilter on $\mathcal{B}$ and $\mathbf{V}$ is a $\lambda^+$-saturated $\mathcal{B}$-valued model of $ZFC^-$, then Lemma~\ref{SurveyInterpBVLarge} states that $\mathfrak{b}_{\hat{V}} > \lambda$. Theorem~\ref{SaturationCharacterization} states that if $\hat{V} \models ZFC^-$ and $\kappa$ is an uncountable cardinal, then $\hat{V}$ is $\kappa$-saturated if and only if $\mathfrak{p}_{\hat{V}} \geq \kappa$ and $\mathfrak{b}_{\hat{V}} \geq \kappa$.

Our strategy is to find conditions on $\hat{V} \models ZFC^-$ to guarantee that $\mathfrak{b}_{\hat{V}}$ is large. The reason we want this is the following simple lemma, which has been implicit in several of our arguments so far:

\begin{lemma}\label{SatAndPseudo1}
	Suppose $V \models ZFC^-$ is transitive and $\mathbf{j}: V \preceq \hat{V}$, and $T \in V$ is a complete countable theory. Then for every $M \models T$ with $M \in V$, and for every uncountable cardinal $\kappa \leq \mathfrak{b}_{\hat{V}}$, the following are equivalent:
	
	\begin{itemize}
		\item[(A)] $\hat{V}$ $\kappa$-pseudosaturates $T$;
		\item[(B)] $\mathbf{j}_{\std}(M)$ is $\kappa$-saturated.
	\end{itemize}
\end{lemma}
\begin{proof}
	Immediate, by the definitions.
\end{proof}

The following expansion is used implicitly by Malliaris and Shelah in \cite{InterpOrders}; it is designed to tie $\mathfrak{b}_{\hat{V}}$ down to pseudofinite invariants of $\hat{V}$.
\begin{definition}
	Let $\mathcal{L}_* = \{\in, I, F\}$ where $I$ is a unary relation symbol and $F$ is a unary function symbol. Let $ZFC^-_*$ be the $\mathcal{L}_*$ theory, such that $(\hat{V}, \hat{\in}, I^{\hat{V}}, F^{\hat{V}}) \models ZFC^-_*$ if:
	
	\begin{itemize}
		\item $(\hat{V}, \hat{\in}) \models ZFC^-$;
		\item $I^{\hat{V}}$ is a bounded subset of $\hat{\omega}$;
		\item $F^{\hat{V}}$ is a bijection from $I^{\hat{V}}$ onto $\hat{V}$, and for every $\hat{n} \in I^{\hat{V}}$, $F^{\hat{V}} \restriction \{\hat{m} \in I^{\hat{V}}: \hat{m} \leq \hat{n}\} \in \hat{V}$;
		\item For every $\hat{a} \in \hat{V}$, either $\{\hat{n} \in I^{\hat{V}}: \hat{n} \, \hat{\in} \, \hat{a}\}$ is bounded in $I^{\hat{V}}$, or its complement is.
	\end{itemize}
	
	Given $\hat{V} \models ZFC^-$, say that $\hat{V} \models ZFC^-_{\pre}$ if $\hat{V}$ can be expanded to a model $(\hat{V}, I^{\hat{V}}, F^{\hat{V}})$ of $ZFC^-_*$. In particular this implies $\hat{V}$ is $\omega$-nonstandard.
\end{definition}

This is not a vacuous definition:

\begin{lemma}\label{SatAndPseudo2}
	Suppose $V \models ZFC^-$ is transitive. Then there is some $\mathbf{j}: V \preceq \hat{V}$ with $\hat{V} \models ZFC^-_{\pre}$.
\end{lemma}
\begin{proof}
	We can suppose, by compactness, that $V$ is countable. Let $\mathbf{j}:V \preceq \hat{V}$ with $\hat{V}$ $\omega$-nonstandard and countable. We show that $\hat{V} \models ZFC^-_{\pre}$, that is, we find an expansion to a model of $ZFC^-_*$.
	
	Enumerate $\hat{V} = \{\hat{a}_m: m < \omega\}$. Let $Y_0 = \omega$; for each $m < \omega$, let $Y_{m+1}$ be either $\{n \in Y_m: n \in \hat{a}_m\}$ or else $\{n \in Y_m: n \not \in \hat{a}_m\}$, chosen so that $Y_{m+1}$ is infinite. Choose an increasing sequence $(b_n: n < \omega)$ so that $b_n \in Y_n$.
	
	Define $I^{\hat{V}} := \{b_n: n < \omega\}$, a bounded subset of $\hat{\omega}$. Choose a bijection $F^{\hat{V}}: I^{\hat{V}} \to \hat{V}$. Then $(\hat{V}, I^{\hat{V}}, F^{\hat{V}})$ works: since $[\hat{V}]^{<\aleph_0} \subseteq \hat{V}$, we see that $\hat{V}$ contains the initial segments of $F^{\hat{V}}$. Further, for every $\hat{a} \in \hat{V}$, if we write $\hat{a} = \hat{a}_m$, then $\hat{a} \cap I^{\hat{V}}$ is either bounded or cobounded in $I^{\hat{V}}$, depending on the choice we made for $Y_{m+1}$. 
\end{proof}

The definition of $ZFC^-_*$ was rigged to make the following work:

\begin{theorem}\label{BBound}
	Suppose $V \models ZFC^-$ is transitive and suppose $\mathbf{j}:V \preceq \hat{V} \models ZFC^-_{\pre}$. Then $\hat{V}$ does not $\mathfrak{b}_{\hat{V}}^+$-pseudosaturate any unstable theory.
\end{theorem}
\begin{proof}
	Let $(\hat{V},I^{\hat{V}}, F^{\hat{V}})$ be an expansion of $\hat{V}$ to a model of $ZFC^-_{*}$. Suppose $T \in V$ is unstable. By Theorem~\ref{NonStableDichotomy}, either $T$ has $SOP_2$ or else $T$ has the independence property.
	
	Suppose first $T$ has the independence property, say via $\phi(\overline{x}, \overline{y})$. Choose $M \models T$ and $(\overline{a}_n: n < \omega)$ a sequence from $M$ in $V$, such that for all disjoint $u, v \in [\omega]^{<\aleph_0}$, $\{\phi(\overline{x}, \overline{a}_n): n \in u\} \cup \{\lnot \phi(\overline{x}, \overline{a}_m): m \in v\}$ is consistent. Let $(\overline{a}_{\hat{n}}: \hat{n}< \hat{\omega}) = \mathbf{j}(\overline{a}_n: n < \omega)$.
	
	Choose $X \subseteq \hat{V}$ of cardinality $\mathfrak{b}_{\hat{V}}$, such that $X$ is not pseudofinite. Define $Y := (F^{\hat{V}})^{-1}(X)$; since $X$ is not pseudofinite, and since $\hat{V}$ contains the initial segments of $\hat{F}$, $Y$ must be unbounded in $I^{\hat{V}}$.
	
	Let $I_0, I_1$ be two unbounded, disjoint subsets of $\lambda$ each of size $\lambda$; for instance, we could let $I_0 = \{2\alpha: \alpha < \lambda\}$ and $I_1 = \{2\alpha+1: \alpha < \lambda\}$. Write $Y_i = \{\hat{n}_{\alpha}: \alpha \in I_i\}$. Let $p(x) = \{\phi(\overline{x}, \overline{a}_{\hat{n}}): \hat{n} \in Y_0\} \cup \{\lnot \phi(\overline{x}, \overline{a}_{\hat{n}}): \hat{n} \in Y_1\}$. $p(x)$ is a consistent pseudofinite partial type over $\mathbf{j}_{\std}(M)$ of cardinality $\mathfrak{b}_{\hat{V}}$. Note that whenever $\hat{Y}_0 \in \hat{V}$ contains $Y_0$, we must have that $\hat{Y}_0$ contains an end segment of $I^{\hat{V}}$; in particular $\hat{Y}_0$ contains elements of $Y_1$.  Hence $p(x)$ cannot be realized in $\mathbf{j}_{\std}(M)$, as if $\overline{b}$ were a realization, then we could set $\hat{Y}_0 = \{\overline{a}: \mathbf{j}(M) \models \phi(\overline{b}, \overline{a})\}$. 
	
	If $T$ has $SOP_2$, then by Corollary~\ref{LocalSat0} and Theorem~\ref{SOP2Max}, $\hat{V}$ $\lambda^+$-pseudosaturates $T$ if and only if $\lambda < \mathfrak{p}_{\hat{V}}$. By the preceding argument and Theorem~\ref{localSaturation2}, we get that $\mathfrak{p}_{\hat{V}} \leq \mathfrak{b}_{\hat{V}}$, and so we conclude.
\end{proof}

\begin{corollary}\label{SatAndPseudoCor0}
Suppose $V \models ZFC^-$ is transitive and suppose $\mathbf{j}: V \preceq \hat{V} \models ZFC^-_{\pre}$, and suppose $T$ is a complete countable unstable theory in $V$, and suppose $M \models T$ with $M \in V$. Then for all $\kappa > \aleph_0$, $\mathbf{j}_{\std}(M)$ is $\kappa$-saturated if and only if $\hat{V}$ $\kappa$-pseudosaturates $T$.
\end{corollary}
\begin{proof}
Immediate by Lemma~\ref{SatAndPseudo1} and Theorem~\ref{BBound}.
\end{proof}

Next, we easily get the following characterization of $\trianglelefteq^*_\kappa$; it only uses that $ZFC^-_{*}$ satisfies Lemma~\ref{SatAndPseudo2}.
\begin{theorem}\label{SatAndPseudoThm1}
	Suppose $\kappa$ is an infinite cardinal or $1$, and $T_0, T_1$ are complete countable theories. Then the following are equivalent:
	\begin{itemize}
		\item[(A)] $T_0 \trianglelefteq^*_{\lambda \kappa} T_1$;
		\item[(B)] There is some countable transitive $V \models ZFC^-$ with $T_0, T_1 \in V$, and some $M_i \models T_i$ both in $V$, such that for all $\mathbf{j}:V \preceq \hat{V} \models ZFC^-_{\pre}$, if $\hat{V}$ is $\kappa$-saturated and $\mathbf{j}_{\std}(M_1)$ is $\lambda^+$-saturated, then $\mathbf{j}_{\std}(M_0)$ is $\lambda^+$-saturated.
	\end{itemize}
\end{theorem}
\begin{proof}
	(A) implies (B): this follows from Lemma~\ref{SatAndPseudo0}(A) implies (B).
	
	(B) implies (A): Given $V$, choose some $\mathbf{j}_0: V \preceq \hat{V}_0 \models ZFC^-_{\pre}$. Let $(\hat{V}_0, \hat{I}^{\hat{V}_0}, \hat{F}^{\hat{V}_0})$ be an expansion of $\hat{V}_0$ to a model of $ZFC^-_{\pre}$. Let $T_*$ be the elementary diagram of $(\hat{V}_0, \hat{I}^{\hat{V}_0}, \hat{F}^{\hat{V}_0})$. We get natural interpretations  $I_i$ of $T_i$ in $T_*$, using the constant symbols for $\mathbf{j}(M_i)$.  Then this witnesses $T_0 \trianglelefteq^*_{\lambda \kappa} T_1$.
\end{proof}

We will want the following straightforward application of Theorem~\ref{Compactness}; it is proved similarly to Lemma~\ref{SatAndPseudo3First}.

\begin{lemma}\label{SatAndPseudo3}
	Suppose $T \subseteq T_*$ are complete theories in languages $\mathcal{L} \subseteq \mathcal{L}_*$. Suppose $\mathcal{B}$ is a complete Boolean algebra and $\mathbf{M}_*$ is a full $\mathcal{B}$-valued model of $T_*$; write $\mathbf{M} = \mathbf{M}_* \restriction_{\mathcal{L}}$, so $\mathbf{M}$ is a full $\mathcal{B}$-valued model of $T$. If $\mathbf{M}_*$ is $\lambda^+$-saturated, then so is $\mathbf{M}$.
\end{lemma}

We obtain the following. We will eventually remove the hypothesis that $T_1$ is unstable.
\begin{corollary}\label{SatAndPseudoCor1}
	Suppose $T_0, T_1$ are complete countable theories, and suppose $\kappa$ is infinite or $1$. If $T_1$ is unstable, then the following are equivalent:
	
	\begin{itemize}
		\item[(A)] $T_0 \trianglelefteq^*_{\kappa} T_1$;
		\item[(B)] For all cardinals $\lambda$, there is some countable transitive $V \models ZFC^-$ with $T_0, T_1 \in V$, such that for all $\mathbf{j}: V \preceq \hat{V} \models ZFC^-_{\pre}$ with $\hat{V}$ $\kappa$-saturated, if $\hat{V}$ $\lambda^+$-pseudosaturates  $T_1$, then it $\lambda^+$-pseudosaturates $T_0$.
	\end{itemize}
\end{corollary}
\begin{proof}
%
%

	(B) implies (A): Choose $M_i \models T_i$ with $M_i \in V$. We claim that $(V, M_0, V_1)$ witnesses Theorem~\ref{SatAndPseudoThm1} holds, and hence that $T_0 \trianglelefteq^*_{\lambda \kappa} T_1$. Indeed, suppose $\mathbf{j}: V \preceq \hat{V} \models ZFC^-_{\pre}$ with $\hat{V}$ $\kappa$-saturated. Suppose $\mathbf{j}_{\std}(M_1)$ is $\lambda^+$-saturated. Then in particular $\mathbf{j}_{\std}(M_1)$ is $\lambda^+$-pseudosaturated, hence $\mathbf{j}_{\std}(M_0)$ is $\lambda^+$-pseudosaturated. Since $T_1$ is unstable, we ge that $\lambda < \mathfrak{b}_{\hat{V}}$ by Theorem~\ref{BBound}. Thus $\mathbf{j}_{\std}(M_0)$ is $\lambda^+$-saturated.
	
	(A) implies (B): Choose some transitive countable $V \models ZFC^-$ and $M_i \models T_i$ with $M_i \in V$, as in Theorem~\ref{SatAndPseudoThm1}(B). Now suppose $\mathbf{j}: V \preceq \hat{V} \models ZFC^-_{\pre}$ satisfies that $\hat{V}$ is $\kappa$-saturated, and $\mathbf{j}_{\std}(M_1)$ is $\lambda^+$-pseudosaturated. Then by Lemma~\ref{SatAndPseudo1} and Theorem~\ref{BBound},  $\mathbf{j}_{\std}(M_1)$ is $\lambda^+$-saturated, so $\mathbf{j}_{\std}(M_0)$ is $\lambda^+$-saturated, in particular it is $\lambda^+$-pseudosaturated.
\end{proof}

\section{Baseline Saturation}\label{InterpOnStableSec1}

In this section, we prove some theorems of the following form: if $\mathbf{j}: V \preceq \hat{V}$ and $\hat{V}$ has some minimal amount of (pseudo)saturation, then for $M \in V$ satisfying various hypotheses, $\mathbf{j}_{\std}(M)$ has some minimal amount of (pseudo)saturation, and conversely. These theorems are adaptions of arguments of Malliaris and Shelah \cite{InterpNew}; many of these results also have antecedents in \cite{TransferringSaturation}. The key novelty of our approach is using models of $ZFC^-$ instead of interpretations. 

To begin, we consider what is required for unsupersimple theories to be $\aleph_1$-pseuodosaturated.
\begin{theorem}\label{BaselineSupersimple}
	Suppose $V$ is a transitive model of $ZFC^-$ and $\mathbf{j}:V \preceq \hat{V}$ is not $\omega$-standard. Then the following are equivalent:
	
	\begin{itemize}
		\item[(A)] There is some countable unsupersimple $T \in V$ such that $\hat{V}$ $\aleph_1$-pseudosaturates $T$;
		\item[(B)] $\mathfrak{p}_{\hat{V}} \geq \aleph_1$, i.e. $\hat{V}$ $\aleph_1$-pseudosaturates every countable theory $T \in V$.
	\end{itemize}
	
\end{theorem}
\begin{proof}
	(A) implies (B): suppose $T \in V$ is unsupersimple, and $\hat{V}$ $\aleph_1$-pseudosaturates $T$. It suffices to show that $(\aleph_0, \aleph_0) \not \in \mathcal{C}(\hat{V})$.

	By Lemma~\ref{SupersimpleCharacterization} (and using that $\Pi^1_1$-relations are absolute to models of $ZFC^-$) we can find some $M \models T$ in $V$, and a sequence $(\phi_n(\overline{x}, \overline{y}_n)) \in V$ of formulas of $T$, and a tree $(\overline{a}_s: s \in \omega^{<\omega}) \in V$ of tuples in $M$, as there.

	Now suppose that $(\hat{n}_i: i < \omega)$, $(\hat{m}_i: i < \omega)$ is a pre-cut in $\hat{\omega}$, i.e. $\hat{n}_i < \hat{n}_{j} < \hat{m}_{j} < \hat{m}_i$ for all $i < j < \omega$. It suffices to find $\hat{n} \in \hat{\omega}$ with $\hat{n}_i < \hat{n} < \hat{m}_i$ for all $i < \omega$. Let $\langle, \rangle: \hat{\omega} \times \hat{\omega} \to \hat{\omega}$ be a bijection in $\hat{V}$. For each $i < \omega$, let $\hat{k}_i = \langle \hat{n}_i, \hat{m}_i \rangle$. For each $n < \omega$ let $\hat{s}_{n} = (\hat{k}_i: i < n)$.

	Let $p(\overline{x}) = \{\phi_n(\overline{x}, \overline{a}_{\hat{s}_n}): n < \omega\}$, a countable partial type over $\mathbf{j}_{\std}(M)$. $p(\overline{x})$ is pseudofinite, since the $\hat{k}_i$'s are bounded (since $\langle \hat{n}, \hat{m} \rangle: \hat{n}, \hat{m} \leq \hat{m}_0 \rangle$ is bounded). Thus $p(\overline{x})$ is realized by some $\hat{\overline{b}} \in \hat{V}$. 
	
	Let $(\hat{\overline{a}}_{\hat{s}}: \hat{s} \in \hat{\omega}^{<\hat{\omega}}) = \mathbf{j}(\overline{a}_s: s \in \omega^{<\omega})$ and similarly define $(\hat{\phi}(\overline{x}, \overline{y}_{\hat{n}}): \hat{n}< \hat{\omega})$. By overflow, we can find some $\hat{s} \in \hat{\omega}^{<\hat{\omega}}$ of nonstandard length such that $\mathbf{j}(M) \models \hat{\phi}_{|\hat{t}|}(\hat{\overline{b}}, \hat{\overline{a}}_{\hat{t}})$ for all $\hat{t} \subseteq \hat{s}$. Now necessarily $\hat{s}_n \subseteq \hat{s}$ for all $n < \omega$. Write $\hat{n}_* = |\overline{\hat{s}}|$. For each $\hat{m} < \hat{n}_*$, write $\hat{s}(\hat{m}) = \langle \hat{s}_0(\hat{m}), \hat{s}_1(\hat{m}) \rangle$. By passing to an initial segment of $\hat{s}$, we can suppose that for all $\hat{n} < \hat{m} < \hat{n}_*$, $\hat{s}_0(\hat{n}) < \hat{s}_0(\hat{m}) < \hat{s}_1(\hat{m}) < \hat{s}_0(\hat{m})$. Then $\hat{s}_0(\hat{n}_*-1)$ realizes our given precut, as desired.
	
	(B) implies (A): trivial.
\end{proof}

\begin{corollary}\label{supersimpleCor1}
	Suppose $V \models ZFC^-$ and $\mathbf{j}:V \preceq \hat{V} \models ZFC^-_{\pre}$. Then $\hat{V}$ is $\aleph_1$-saturated if and only if $\hat{V}$ $\aleph_1$-pseudosaturates some unsupersimple theory.
\end{corollary}
\begin{proof}
	Clearly, if $\hat{V}$ is $\aleph_1$-saturated then $\hat{V}$ $\aleph_1$-pseudosaturates every complete countable theory $T \in V$. Conversely, suppose $\hat{V}$ $\aleph_1$-pseudosaturates some unsupersimple theory $T$. Then $\mathfrak{b}_{\hat{V}} \geq \mathfrak{p}_{\hat{V}} \geq \aleph_1$ by Theorem \ref{BBound} and Theorem \ref{BaselineSupersimple}.  In particular, $\mathfrak{p}_{\hat{V}}, \mathfrak{b}_{\hat{V}} \geq \aleph_1$. Thus, by Theorem~\ref{SaturationCharacterization}, $\hat{V}$ is $\aleph_1$-saturated.
\end{proof}

Next, we consider $\aleph_{\epsilon}$-saturation and $\aleph_0$-saturation. Recall the $\epsilon$-notation from the Section~\ref{Prelim}. To given an example of how we are using it:

\begin{example}Suppose $V \models ZFC^-$ is transitive, $M \in V$, and $\mathbf{j}: V \preceq \hat{V}$ is $\omega$-nonstandard. Then every $\epsilon$-finite subset of $\mathbf{j}_{\std}(M)$ is pseudofinite. Note that we are evaluating $\mbox{acl}$ in the sense of $\mathbf{j}_{\std}(M)$, not $\mathbf{j}(M)$; if we wanted the latter, we could write $\hat{\epsilon}$-finite, but this won't be needed.
	
	Hence, if $\hat{V}$ $\aleph_1$-pseudosaturates $\mbox{Th}(M)$, then $\mathbf{j}_{\std}(M)$ is $\aleph_{\epsilon}$-saturated.
\end{example}

\begin{definition}
	Suppose $\hat{V} \models ZFC^-$. Then say that $\hat{V}$ is $2^\omega$-rich if for every $\eta \in 2^\omega$, there is some $\hat{\eta} \in 2^{\hat{\omega}}$ such that each $\hat{\eta}(n) = \eta(n)$.
\end{definition}

We obtain the following:

\begin{theorem}\label{BaselineEpsilonSmall}
	Suppose $V \models ZFC^-$ is transitive, and $\mathbf{j}: V \preceq \hat{V}$ is $\omega$-nonstandard. Then the following are equivalent:
	
	\begin{itemize}
		\item[(A1)] There is some non-small theory $T \in V$ and some $M \models T$ such that $\mathbf{j}_{\std}(M)$ is $\aleph_0$-saturated;
		\item[(A2)] There is some countable non-$\epsilon$-small theory $T \in V$ and some $M \models T$ such that $\mathbf{j}_{\std}(M)$ is $\aleph_{\epsilon}$-saturated;
		\item[(B)] For every countable complete $T \in V$ and for every $M \models T$ with $M \in V$, $\mathbf{j}_{\std}(M)$ is $\aleph_{\epsilon}$-saturated;
		\item[(C)] Whenever $a \in \hat{V}$ and whenever $p(x, a)$ is a $\Sigma_n$-type over $a$ for some $n < \omega$, then $p(x, a)$ is realized in $\hat{V}$;
		\item[(D)] $\hat{V}$ is $2^\omega$-rich.
	\end{itemize} 
\end{theorem}
\begin{proof}
	(A1) implies (D): we can choose $M \models T$ with $M \in V$ and $\overline{a} \in M^{<\omega}$ with $S^1(\overline{a})$ uncountable, since $\Pi^1_1$-statements are absolute to transitive models. Suppose towards a contradiction that $\mathbf{j}_{\std}(M)$ is $\aleph_0$-saturated and yet $\hat{V}$ is not $2^\omega$-rich. Choose, in $V$, formulas $(\phi_s(x, \overline{a}):s  \in 2^{<\omega})$, where for every $\eta \in 2^\omega$, $\{\phi_s(x, \overline{a}): s \subseteq \eta\}$ is consistent, and for all $s, t \in 2^{<\omega}$ incompatible, $\phi_s(x, \overline{a}) \wedge \phi_t(x, \overline{a})$ is inconsistent. Write $(\hat{\phi}_{\hat{s}}(x, \overline{a}): \hat{s}\in 2^{<\hat{\omega}}) = \mathbf{j}(\phi_s(x, \overline{a}):s  \in 2^{<\omega})$, and choose $\eta \in 2^\omega$ witnessing that $\hat{V}$ is not $2^\omega$-rich. By hypothesis,  $\{\phi_s(x, \overline{a}): s \subseteq \eta\}$ is realized in $\mathbf{j}_{\std}(M)$, say by $a \in \mathbf{j}_{\std}(M)$.  Let $\hat{s} = \bigcup\{\hat{t} \in 2^{<\hat{\omega}}: \mathbf{j}(M) \models \hat{\phi}_{\hat{t}}(a, \overline{a})\}$, an element of $2^{\leq\hat{\omega}}$. Choose $\hat{\eta} \in 2^{\hat{\omega}}$ extending $\hat{s}$. Then $\hat{\eta}(n) = \eta(n)$ for each $n < \omega$, contradicting our choice of $\eta$.
	
	(A2) implies (D): similar, except our tree of formulas is $(\phi_s(x, \overline{b}_s): s \in 2^{<\omega})$, where each $\overline{b}_s \subseteq \mbox{acl}(\overline{a})$.
	
	(D) implies (C): Let $p(x, a)$ be given. Let $\psi(x, y, z)$ be a truth-predicate for $\Sigma_n$-formulas.
	
	Let $(\phi_n(x, y): n < \omega) \in V$ be an enumeration of the $\Sigma_n$-formulas of set theory in $2$ variables, and let $(\hat{\phi}_{\hat{n}}(x, y): \hat{n} < \hat{\omega}) = \mathbf{j}(\phi_n(x, y): n < \omega)$. Let $\eta \in 2^\omega$ be defined by $\eta(n) = 1$ if and only if $\phi_n(x, a) \in p_n(x,a)$. By hypothesis, we can find some $\hat{\eta} \in 2^{\hat{\omega}}$ such that each $\hat{\eta}(n) = \eta(n)$. By overspill, we can find some $\hat{n}_* < \hat{\omega}$ nonstandard such that $\{\psi(x, a, \hat{\phi}_{\hat{n}}(x, y)): \hat{n} < \hat{n}_*, \hat{\eta}(\hat{n}) = 1\}$ has a realization $b$. Clearly, $b$ realizes $p(x)$.
	
	(C) implies (B): Suppose (C) holds, and $T$ is given. We can suppose $T$ is Skolemized, and so we just need to show that for every $M \models T$ with $M \in V$, $\mathbf{j}_{\std}(M)$ is $\aleph_0$-saturated. But note that $\hat{V}$ admits pairing functions, so this follows from (C).

	(B) implies (A2), and (B) implies (A1): Trivial. 
\end{proof}

The following theorem completes the picture, showing that for small theories, $\aleph_0$-saturation is automatic:
\begin{theorem}\label{BaselineSmall}
	Suppose $V \models ZFC^-$ is transitive, and $\mathbf{j}: V \preceq \hat{V}$ is $\omega$-nonstandard. Suppose $T \in V$ is a countable small theory, and $M \models T$ with $M \in V$. Then $\mathbf{j}_{\std}(M)$ is $\aleph_0$-saturated.
\end{theorem}
\begin{proof}
	We first show that $\mathbf{j}_{\std}(M)$ is $\aleph_0$-homogeneous (this does not use the hypothesis on $T$). So suppose $\overline{a}, \overline{b} \in \mathbf{j}_{\std}(M)^{<\omega}$ satisfy that $\mbox{tp}_{\mathbf{j}_{\std}(M)}(\overline{a})=\mbox{tp}_{\mathbf{j}_{\std}(M)}(\overline{b})$, and let $a \in \mathbf{j}_{\std}(M)$ be given. Let $(\Delta_n: n < \omega) \in V$ be a filtration of the formulas of $\mathcal{L}$ into finite sets, and let $(\hat{\Delta}_{\hat{n}}: \hat{n} < \hat{\omega}) = \mathbf{j}(\Delta_n: n < \omega)\in \hat{V}$. For each $n < \omega$, we can find $b \in \mathbf{j}_{\std}(M)$ such that $tp_{\Delta_n}(a, \overline{a}) = tp_{\Delta_n}(b, \overline{b})$, hence by overflow there is some nonstandard $\hat{n} < \hat{\omega}$ and some $b \in \mathbf{j}_{\std}(M)$ such that $tp_{\hat{\Delta}_{\hat{n}}}(a, \overline{a}) = tp_{\hat{\Delta}_{\hat{n}}}(b, \overline{b})$. Then in particular, $\mbox{tp}_{\mathbf{j}_{\std}(M)}(a,\overline{a}) = \mbox{tp}_{\mathbf{j}_{\std}(M)}(b, \overline{b})$ as desired.
	
	Next, we show that $\mathbf{j}_{\std}(M)$ realizes every type over $\emptyset$. So suppose $q(x_i: i < n) \in S^n(\emptyset)$. Since $T$ is small, $q(x_i: i < n) \in V$. To see this, note that $S^n(\emptyset)$ has countable Cantor-Bendixson rank; prove by induction that types of rank $\alpha$ are in $V$.
	
	Moreover, the isolated types are dense in $S^n(\emptyset)$, and so we can find a sequence $(\phi_m(x_i: i < n): m < \omega) \in V$ of formulas, such that each $\phi_m(x_i: i < n)$ generates a complete type, and for every $\phi(x_i: i < n) \in q(x_i: i < n)$, we have that $\phi_m(x_i: i < n) \rightarrow \phi(x_i: i < n) \in T$ for sufficiently large $m$. Let $(\hat{\phi}_{\hat{m}}(x_i: i < n): \hat{m} < \hat{\omega}) = \mathbf{j}(\phi_m(x_i: i < n): n < \omega)$. Let $\hat{n} < \hat{\omega}$ be nonstandard, and choose $(a_i: i < n)$ such that $\mathbf{j}(M) \models \hat{\phi}_{\hat{n}}(a_i: i < n)$. Then clearly $(a_i: i < n)$ realizes $q(x_i: i < n)$.
\end{proof}

We note that this allows us to extend Theorem~\ref{ModKeislerOrder} to the case $\lambda = \aleph_0$, as promised. Recall that $\aleph_0$-pseudosaturation is the same as $\aleph_0$-saturation.

\begin{corollary}\label{ModKeislerOrder2} Suppose $V \models ZFC^-$ is transitive and $\mathbf{j}: V \preceq \hat{V}$ is $\omega$-nonstandard, and $T \in V$ is a complete countable theory, and $M \in V$ is a model of $T$. Then $\mathbf{j}_{\std}(M)$ is $\aleph_0$-saturated if and only if either $T$ is small or else $\hat{V}$ is $2^\omega$-rich. In particular, this does not depend on the choice of $M$. 
\end{corollary}

Thus we can extend the definition of pseudosaturation of theories:

\begin{definition}
	Suppose $V \models ZFC^-$ is transitive and $\mathbf{j}: V \preceq \hat{V}$ is $\omega$-nonstandard, and $T \in V$ is a complete countable theory. Then say that $\hat{V}$ $\aleph_0$-pseudosaturates $T$ if for some or any $M \models T$ with $M \in V$, $\mathbf{j}_{\std}(M)$ is $\aleph_0$-(pseudo)saturated. 
\end{definition}

\section{Levels of Saturation of Stable Theories}\label{InterpOnStableSec2}
In this section, we adapt arguments of Malliaris and Shelah \cite{InterpNew} to the context of models of $ZFC^-$, to perform the following computations: given some transitive model $V \models ZFC^-$, some stable theory $T \in V$, and some $\mathbf{j}: V \preceq \hat{V}$ with $\hat{V}$ $\omega$-nonstandard, we pin down precisely the set of cardinals $\lambda$ such that $\hat{V}$ $\lambda$-pseudosaturates $T$. Also, given a model $M \models T$ with $M \in V$, and given $\mathbf{j}: V \preceq \hat{V} \models ZFC^-_{\pre}$, we pin down precisely the set of cardinals $\lambda$ such that $\mathbf{j}_{\std}(M)$ is $\lambda$-saturated; this will not depend on the choice of $M$.

We now introduce a cardinal characteristic of models of $ZFC^-$, that will control how saturated models of theories with the finite cover property can be.

\begin{definition}
	Suppose $\hat{V} \models ZFC^-$ is nonstandard. Note that if $\hat{X} \in \hat{V}$, then $|\hat{X}|$ could refer to either the cardinality of $\hat{X}$ as computed in $\hat{V}$, or the cardinality of $\{a \in \hat{V}: a \in \hat{X}\}$ as computed in $\mathbb{V}$. To clarify this, we will write $|\hat{X}|_{\hat{V}}$ or $|\hat{X}|_{\mathbb{V}}$, respectively.
	
	If $(L, <)$ is a linear order with proper initial segment $\omega$, then let $\mu_{(L,<)}$ denote the least cardinality of an initial segment $L_0$ of $L$ which properly contains $\omega$. 
	
	Suppose $\hat{V} \models ZFC^-$ is $\omega$-nonstandard. Then let $\mu_{\hat{V}} = \mu_{(\hat{\omega}, <)}$. In other words, $\mu_{\hat{V}}$ is the minimum of $|\hat{n}|_{\mathbb{V}}$, for $\hat{n} < \hat{\omega}$ nonstandard.
\end{definition}

To state the main theorem of this section, it is convenient to introduce the following definitions. We consider $\infty$ to be larger than every cardinal. If $\mathbf{j}: V \preceq \hat{V}$ and $M \in V$, we say that $\mathbf{j}_{\std}(M)$ is $\infty$-pseudosaturated if every pseudofinite type over $\mathbf{j}_{\std}(M)$ is realized. Thus: for every structure $M$, there is a least infinite cardinal $\lambda$ such that $M$ is not $\lambda^+$-saturated, and if $T \in V$ and $\mathbf{j}: V \preceq \hat{V}$, then there is a least infinite cardinal $\lambda$ (possibly $\infty$) such that $\hat{V}$ does not $\lambda^+$-pseudosaturate $T$. (As a matter of pure notation, we define that $\hat{V}$ never $\infty^+$-pseudosaturates $T$.)

\begin{definition}\label{LongDef}
Suppose $V \models ZFC^-$ is transitive, and $T \in V$ is a countable stable theory. Suppose $\mathbf{j}: V \preceq \hat{V}$ is $\omega$-nonstandard.  Then define $\lambda_{\hat{V}}(T)$ to be the largest value consistent with the following clauses. Also, if $\hat{V} \models ZFC^-_{\pre}$, then define $\lambda'_{\hat{V}}$ similarly. $\lambda_{\hat{V}}(T)$ will measure the level of pseudosaturation, and $\lambda'_{\hat{V}}(T)$ will measure the level of saturation.
\begin{enumerate}
	\item If $T$ is not $\omega$-stable and $\hat{V}$ is not $2^\omega$-rich, then $\lambda_{\hat{V}}(T) =\lambda'_{\hat{V}}(T) = \aleph_0$.
	\item If $T$ is not superstable and $\mathfrak{p}_{\hat{V}} = \aleph_0$, then $\lambda_{\hat{V}}(T) = \aleph_0$. 
	\item If $T$ is not superstable and $\hat{V}$ is not $\aleph_1$-saturated, then $\lambda'_{\hat{V}}(T) = \aleph_0$.
	\item If $T$ has the finite cover property, then $\lambda_{\hat{V}}(T) \leq \mu_{\hat{V}}$ and $\lambda'_{\hat{V}}(T) \leq \mu_{\hat{V}}$.
	\item $\lambda_{\hat{V}}(T) \leq \infty$, and $\lambda'_{\hat{V}}(T) \leq |\hat{V}|$.
\end{enumerate}
\end{definition}

The rest of this section will be concerned with the proof of the following theorems.

\begin{theorem}\label{PseudosaturationLevelsThm}
	Suppose $V \models ZFC^-$ is transitive, and $T \in V$ is a countable stable theory. Suppose $\mathbf{j}:V \preceq \hat{V} \models ZFC^-$ is $\omega$-nonstandard. Then $\lambda_{\hat{V}}(T)$ is the least infinite cardinal such that $T$ does not $\lambda_{\hat{V}}(T)^+$-pseudosaturate $T$.
	
\end{theorem}

Note: if $\lambda_{\hat{V}}(T) > \aleph_0$, then it follows $\hat{V}$ does $\lambda_{\hat{V}}(T)$-pseudosaturate $T$. If $\lambda_{\hat{V}}(T) = \aleph_0$ then this may fail, but we know exactly when this happens: $\hat{V}$ $\aleph_0$-pseudosaturates $T$ if and only if either $T$ is small or else $\hat{V}$ is $2^\omega$ rich. Trying to incorporate this information into the definition of $\lambda_{\hat{V}}(T)$ would cause notational difficulties down the road.

\begin{theorem}\label{SaturationLevelsThm}
Suppose $V \models ZFC^-$ is transitive, and $T \in V$ is a countable stable theory, and $M \models T$ with $M \in V$. Suppose $\mathbf{j}:V \preceq \hat{V} \models ZFC^-_{\pre}$. Then $\lambda'_{\hat{V}}(T)$ is the least infinite cardinal such that $\mathbf{j}_{\std}(M)$ is not $\lambda'_{\hat{V}}(T)^+$-saturated.

\end{theorem}

Again, if $\lambda'_{\hat{V}}(T) > \aleph_0$, then it follows $\mathbf{j}_{\std}(M)$ is $\lambda'_{\hat{V}}(T)$-saturated, and we know exactly when $\mathbf{j}_{\std}(M)$ is $\aleph_0$-saturated. We note that $ZFC^-_{\pre}$ is overkill in many cases; see the following theorems for sharper results.

First off, we prove the following:

\begin{theorem}\label{InterpOnStable0}
	Suppose $V \models ZFC^-$ is transitive, and $\mathbf{j}: V \preceq \hat{V}$ is $\omega$-nonstandard. Suppose $T \in V$ is a countable stable theory with the finite cover property. Then $\hat{V}$ does not $\mu_{\hat{V}}^+$-pseudosaturate $T$.
\end{theorem}
\begin{proof}
	Suppose $M \models T$ with $M \in V$. We show that $\mathbf{j}_{\std}(M)$ is not $\mu_{\hat{V}}^+$-pseudosaturated. 
	
	Let $\phi(x, y, \overline{z})$ be a formula witnessing (B) of Theorem~\ref{fcpEquiv} holds. Let $X \subseteq \omega$ be the set of all $n < \omega$ such that for some $\overline{c} \in M^{|\overline{z}|}$, $E_{\overline{c}}^M$ has exactly $n$ classes. Then the nonstandard elements of $\mathbf{j}(X)$ are cofinal above $\omega$; thus we can choose $\hat{n}_* \in \mathbf{j}(X)$ nonstandard such that $|\hat{n}_*|_{\mathbb{V}} = \mu_{\hat{V}}$. Let $\overline{c} \in \hat{\omega}^{|\overline{z}|}$ be such that in $\hat{V}$, $E_{\overline{c}}^{\mathbf{j}(M)}$ has $\hat{n}_*$ classes.

	Let $\hat{f}: \hat{n}_* \to \mathbf{j}(M)$ choose a representative from each $E_{\overline{c}}$-class. Consider the partial type $p(x)$ over $\mathbf{j}_{\std}(M)$ which says $\lnot \phi(x, \hat{f}(\hat{n}), \overline{c})$ for all $\hat{n} < \hat{n}_*$. $p(x)$ is pseudofinite, and by choice of $\hat{n}_*$, $|p(x)| = \mu_{\hat{V}}$. But $p(x)$ is omitted in $\mathbf{j}_{\std}(M)$ by choice of $\hat{f}$.
\end{proof}

%

\noindent \textbf{Status Update 1.} Suppose $V \models ZFC^-$ is transitive, and $T \in V$ is a countable stable theory. Suppose $\mathbf{j}:V \preceq \hat{V}$ is $\omega$-nonstandard. Then $\hat{V}$ does not $\lambda_{\hat{V}}(T)^+$-pseudosaturate $T$. 

\begin{proof}

If Clause 1 is activated, then $\hat{V}$ does not $\aleph_1$-pseudosaturate $T$ by Theorem~\ref{BaselineEpsilonSmall}. To see this, note that we can suppose $T$ eliminates imaginaries, since $\hat{V}$ $\lambda$-pseudosaturates $T$ if and only if $\hat{V}$ $\lambda$-pseudosaturates $T^{eq}$. Thus, $T$ is not $\epsilon$-small.  Thus, $\mathbf{j}_{\std}(M)$ is not $\aleph_{\epsilon}$-saturated, and hence not $\aleph_1$-pseudosaturated (since $\epsilon$-finite sets are pseudofinite).

If Clause 2 is activated, then $\hat{V}$ does not $\aleph_1$-pseudosaturate $T$ by Theorem~\ref{BaselineSupersimple}.

Clause 3 is not relevant to $\lambda_{\hat{V}}$. 

If Clause 4 is activated, then $\hat{V}$ does not $\mu_{\hat{V}}^+$-pseudosaturate $T$, by Theorem~\ref{InterpOnStable0}.

For Clause 5, there is nothing to check (since $\infty^+$-pseudosaturation always fails).
\end{proof}

\noindent \textbf{Status Update 2.} Suppose $V \models ZFC^-$ is transitive, and $T \in V$ is a countable stable theory, and $M \models T$ with $M \in V$. Suppose $\mathbf{j}:V \preceq \hat{V}$. Then $\mathbf{j}_{\std}(M)$ is not $\lambda'_{\hat{V}}(T)^+$-saturated.

\begin{proof}
Similar to the proof of Status Update 1.
%
%
%
%
%
%
%
%
%
\end{proof}

Thus, to finish we need to prove that we have found the only obstacles to saturation of stable theories. We aim to apply Lemma~\ref{ShelahStableSatLemma2}; towards this, we explain how to find indiscernible sets in $\mathbf{j}_{\std}(M)$.

\begin{lemma}\label{KeislerOnStableKey}
Suppose $V \models ZFC^-$ is transitive, and $\mathbf{j}: V \preceq \hat{V}$ is $\omega$-nonstandard, and $T \in V$ is a countable stable theory. Suppose $M \models T$ with $M \in V$. Suppose $p(x)$ is a type over $\mathbf{j}_{\std}(M)$. Suppose at least one of the following holds:

\begin{itemize}
	\item[(A)] $\mathbf{p}_{\hat{V}} \geq \aleph_1$ and $p(x)$ is based on a pseudofinite set;
	\item[(B)] $T$ is supersimple and $\hat{V}$ is $2^\omega$-rich;
	\item[(C)] $T$ is $\omega$-stable. 
\end{itemize}

Then we can find some nonstandard $\hat{n}_* < \hat{\omega}$ and some sequence $(\hat{a}_{\hat{n}}: \hat{n} < \hat{n}_*) \in \hat{V}$, such that in $\mathbb{V}$, $\{\hat{a}_{\hat{n}}: \hat{n}< \hat{n}_*\}$ is an indiscernible set $I \subseteq \mathbf{j}_{\std}(M)$ with $\mbox{Av}(I, \mathbf{j}_{\std}(M)) = p(x)$.

If additionally $T$ does not have the finite cover property, then we can find some sequence $(\hat{a}_{\hat{n}}: \hat{n} < \hat{\omega}) \in \hat{V}$, such that in $\mathbb{V}$, $\{\hat{a}_{\hat{n}}: \hat{n}< \hat{\omega}\}$ is an indiscernible set $I \subseteq \mathbf{j}_{\std}(M)$ with $\mbox{Av}(I, \mathbf{j}_{\std}(M)) = p(x)$.
\end{lemma}

\begin{proof}
	Note that we are free to suppose that $T$ eliminates imaginaries (this will only be helpful in case (B)).
	
	Choose $X \subseteq \mathbf{j}_{\std}(M)$ countable such that $p(x)$ is based on $X$, and such that in case (A), $X$ is pseudofinite, and in case (B), $X$ is $\epsilon$-finite, and in case (C), $X$ is finite. 
	
	Let $q(x) = p(x) \restriction_X$, and for each $n$, let $q^\omega(x_i: i < \omega)$ be the type over $X$ of some or any infinite independent set of realizations of $q(x)$. Enumerate $X = \{b_j: j < j_*\}$, so $j_* \leq \omega$.
	
	We now split into cases.

	\noindent \textbf{Case (A)/(B).} Enumerate $q^\omega = \{\phi_n(x_i: i <n, b_j: j < j_*): n < \omega\}$.
	
	\noindent \textbf{FCP Subcase of (A)/(B).} Let $\Gamma(y)$ be the set of formulas over $\hat{V}$ asserting that for some $\hat{n}_* < \hat{\omega}$, $y =(a_{\hat{n}}: \hat{n} < \hat{n}_*)$ is a sequence from $\mathbf{j}(M)$, such that for all $k < \omega$ and for all $\hat{n}_0 < \ldots < \hat{n}_{k-1} < \hat{n}_*$, $\mathbf{j}(M) \models \phi_k(a_{\hat{n}_i}: i < k, b_j: j < j_*)$.  $\Gamma(y)$ is consistent, since $q^\omega(x_i: i < \omega, b_j: j < j_*)$ is consistent. By Lemma~\ref{InterpKeislerLemma0}, it suffices to check that $\Gamma(y)$ is realized. In case (A), use Theorem~\ref{localSaturation2}(A). In case (B), we aim towards applying Theorem~\ref{BaselineEpsilonSmall}(D). So let $M_*$ be a Skolemization of $M$ in $V$ (all we need is a linear order); also, let $\overline{\phi} \in V$ be an enumeration of the set of $\mathcal{L}$-formulas in order-type $\omega$, and let $\hat{\overline{\phi}} = \mathbf{j}(\overline{\phi})$. Note that $\mbox{acl}_{\mathbf{j}_{\std}(M)}(X) \subseteq \mbox{dcl}_{\mathbf{j}_{\std}(M_*)}(X)$, and so in $\hat{V}$, every element $a \in \mbox{acl}_{\mathbf{j}_{\std}(M)}(X)$ is definable over the finite tuple $(b_j: j < j_*)^\frown (\mathbf{j}(M_*), \hat{\overline{\phi}})$ via a $\Delta_0$-formula of set theory (the particular $\Delta_0$-formula needed will depend on $a$). Hence Theorem~\ref{BaselineEpsilonSmall}(D) applies.
	
	\noindent \textbf{NFCP Subcase of (A)/(B).} Let $\Gamma(y)$ be the set of formulas over $\hat{V}$ asserting that $y =(a_{\hat{n}}: \hat{n} < \hat{\omega})$ is a sequence from $\mathbf{j}(M)$, such that for all $k < \omega$ and for all $\hat{n}_0 < \ldots < \hat{n}_{k-1} < \hat{\omega}$, $\mathbf{j}(M) \models \phi_k(a_{\hat{n}_i}: i < k, b_j: j < j_*)$. We check that $\Gamma(y)$ is consistent. For this, it suffices to show that for every fixed $k < \omega$, there is some sequence $(a_{\hat{n}}: \hat{n} < \hat{\omega})$ such that for all $\hat{n}_0 < \ldots < \hat{n}_{k-1} < \hat{\omega}$, $\mathbf{j}(M) \models \phi_k(a_{\hat{n}_i}: i < k, b_j: j < j_*)$ (we are using that $(\phi_k: k < \omega)$ enumerates a complete type, and so each finite subset is implied by a sufficiently large $\phi_k$). This follows from the failure of Theorem~\ref{fcpEquiv}(C). Hence $\Gamma(y)$ is consistent; hence $\Gamma(y)$ is realized in $\hat{V}$, as in the FCP Subcase of (A)/(B).
	
	\noindent \textbf{Case (C).}  Note that $j_* < \aleph_0$, since $X$ is finite. Then $q(x, y_j: j < j_*) \in V$, since $T$ is small (use induction on Cantor-Bendixson rank). Hence $q^\omega(x_i: i < \omega, y_j: j < j_*) \in V$ as it can be computed from $q$. Let $(\phi_n(x_i: i < n, y_j: j < j_*): n < \omega) \in V$ be an enumeration of $q^\omega(x_i: i < \omega, y_j: j <j_*)$. Write $(\hat{\phi}_{\hat{n}}(x_i: i < \hat{n}, y_j: j < j_*): \hat{n} < \hat{\omega}) = \mathbf{j}(\phi_n(x_i: i < n, y_j: j < j_*): n < \omega)$.
	
	\noindent \textbf{FCP Subcase of (C).}  By overflow in $\hat{V}$, we can find some $\hat{n}_* < \hat{\omega}$, such that there are $(a_{\hat{n}}: \hat{n}< \hat{n}_*)$, such that for all $\hat{n} < \hat{n}_*$ and for all increasing sequences $\hat{s}$ of length $\hat{n}$ from $\hat{n}_*$, $\mathbf{j}(M) \models \hat{\phi}_{\hat{n}}(a_{\hat{s}(i)}: i < \hat{n}, b_j: j < j_*)$. Then $\{a_{\hat{n}}: \hat{n}< \hat{n}_*\}$ is the desired indiscernible set, by Lemma~\ref{InterpKeislerLemma0}.
	
	\noindent \textbf{NFCP Subcase of (C).} By the same argument as for the FCP Subcase of (C), it suffices to show that for every $k < \omega$, there is some sequence $(a_{\hat{n}}: \hat{n} < \hat{\omega})$ such that for all $\hat{n}_0 < \ldots < \hat{n}_{k-1} < \hat{\omega}$, $\mathbf{j}(M) \models \phi_k(a_{\hat{n}_i}: i < k, b_j: j < j_*)$; this follows from the failure of Theorem~\ref{fcpEquiv}(C).
\end{proof}

Hence:

\begin{theorem}\label{InterpOnStable2}
	Suppose $V \models ZFC^-$ is transitive, and $\mathbf{j}: V \preceq \hat{V}$, and $T \in V$ is a countable stable theory. Suppose $M \models T$ with $M \in V$. Suppose one of the following conditions is true:
	
	\begin{itemize}
		\item[(A)] $\hat{V}$ is $\aleph_1$-saturated;
		\item[(B)] $\hat{V}$ is $2^\omega$-rich, and $T$ is superstable;
		\item[(C)] $T$ is $\omega$-stable.
	\end{itemize}
	Then $\mathbf{j}_{\std}(M)$ is $\mu_{\hat{V}}$-saturated. If $T$ does not have the finite cover property, then $\mathbf{j}_{\std}(M)$ is $|\hat{\omega}|_{\mathbb{V}}$-saturated. 
\end{theorem}
\begin{proof}
%
%
	We first show that $\mathbf{j}_{\std}(M)$ is $\mu_{\hat{V}}$-saturated. By Lemma~\ref{ShelahStableSatLemma2}, it suffices to show that if $p(x) \in S^1(\mathbf{j}_{\std}(M))$, then there is an indiscernible set $I \subseteq \mathbf{j}_{\std}(M)$ with $|I| \geq \mu_{\hat{V}}$, such that $\mbox{Av}(I, \mathbf{j}_{\std}(M)) = p(x)$. This follows from Lemma~\ref{KeislerOnStableKey} (note that if we are in case (A) here, then case (A) of Lemma~\ref{KeislerOnStableKey} holds, since $p(x)$ is based on a countable set and every countable subset of $\hat{V}$ is pseudofinite).
	
	Similarly, if $T$ does not have the finite cover property, then by Lemma~\ref{ShelahStableSatLemma2}, it suffices to show that if $p(x) \in S^1(M)$, then there is an indiscernible set $I \subseteq \mathbf{j}_{\std}(M)$ with $|I| \geq |\hat{\omega}|_{\mathbb{V}}$, such that $\mbox{Av}(I, \mathbf{j}_{\std}(M)) = p(x)$. Again, this follows from Lemma~\ref{KeislerOnStableKey}.
\end{proof}

\noindent \textbf{Status Update 3.} We have proven Theorem~\ref{SaturationLevelsThm}. 

\begin{proof}
In the setup of the theorem, we need to show that if $\lambda'_{\hat{V}}(T) \geq \aleph_1$ then $\mathbf{j}_{\std}(M)$ is $\lambda'_{\hat{V}}(T)$-saturated. Since $\lambda'_{\hat{V}}(T) \geq \aleph_1$, Definition~\ref{LongDef} implies one of the clauses (A), (B), (C) above must hold. Thus we are done, using that $|\hat{V}| = |\hat{\omega}|_{\mathbb{V}}$ since $\hat{V} \models ZFC^-_{\pre}$.
\end{proof}
%
%
%
%

\begin{theorem}\label{InterpOnStable3}
	Suppose $V \models ZFC^-$ is transitive, and $\mathbf{j}: V \preceq \hat{V}$, and $T \in V$ is a countable stable theory. Suppose one of the following conditions is true:
	
	\begin{itemize}
		\item[(A)] $\mathfrak{p}_{\hat{V}} \geq \aleph_1$;
		\item[(B)] $\hat{V}$ is $2^\omega$-rich, and $T$ is superstable;
		\item[(C)] $T$ is $\omega$-stable.
	\end{itemize}
	
	Then $\hat{V}$ $\mu_{\hat{V}}$-pseudosaturates $T$. If $T$ does not have the finite cover property, then $\hat{V}$ $\lambda^+$-pseudosaturates $T$ for all $\lambda$.
\end{theorem}
\begin{proof}
	
%
	First, suppose $T$ has the finite cover property. If we are in case (B) or (C), then we can conclude by Theorem~\ref{InterpOnStable2}, so suppose $\mathfrak{p}_{\hat{V}} \geq \aleph_1$. Choose $M \models T$ with $M \in V$. Suppose $p_0(x)$ is a complete type over $X$, a pseudofinite set of cardinality less than $\mu_{\hat{V}}$. Let $p(x)$ be a non-forking extension of $p_0(x)$ to $\mathbf{j}_{\std}(M)$. By Lemma~\ref{KeislerOnStableKey}, there is some nonstandard $\hat{n}_* < \hat{\omega}$ and some $(a_{\hat{n}}: \hat{n} < \hat{n}_*) \in \hat{V}$, such that $\hat{Y} := \{a_{\hat{n}}: \hat{n} < \hat{n}_*\}$ is an indiscernible subset of $\mathbf{j}_{\std}(M)$ with $\mbox{Av}(\hat{Y}, \mathbf{j}_{\std}(M)) = p(x)$. By Lemma~\ref{ShelahStableSatLemma2}, $p(x)$ is realized in $\mathbf{j}_{\std}(M)$.
	
	Next, suppose $T$ does not have the finite cover property. Choose $M \models T$ with $M \in V$. Suppose $p_0(x)$ is a complete type over $X \subseteq \mathbf{j}_{\std}(M)$, a pseudofinite set of arbitrary cardinality; choose $\hat{X} \in \hat{V}$ finite in $\hat{V}$ with $X \subseteq \hat{X} \subseteq \mathbf{j}_{\std}(M)$. Let $p(x)$ be a non-forking extension of $p_0(x)$ to $\mathbf{j}_{\std}(M)$. By Lemma~\ref{KeislerOnStableKey}, there is some $(a_{\hat{n}}: \hat{n} < \hat{\omega}) \in \hat{V}$, such that $I := \{a_{\hat{n}}: \hat{n} < \hat{\omega}\}$ is an indiscernible subset of $\mathbf{j}_{\std}(M)$ with $\mbox{Av}(I, \mathbf{j}_{\std}(M)) = p(x)$. 
	
	By overflow we can find a pseudofinite set $\hat{\Delta} \in \hat{V}$ containing all the true formulas of $T$; by Ramsey's theorem in $\hat{V}$, we can find some $\hat{J} \subseteq \hat{\omega}$ infinite in the sense of $\hat{V}$, such that $(a_{\hat{n}}: \hat{n} \in \hat{J})$ is $\hat{\Delta}$-indiscernible over $\hat{X}$. Then in $\mathbb{V}$, $J := \{a_{\hat{n}}: \hat{n} \in \hat{J}\}$ is indiscernible over $\hat{X}$, and thus any element of $J$ realizes $p_0(x)$.
\end{proof}

%

\noindent \textbf{Status Update 3.} We have proven Theorem~\ref{PseudosaturationLevelsThm}. 

\begin{proof}
	In the setup of the theorem, we need to show that if $\lambda_{\hat{V}}(T) \geq \aleph_1$, then $\hat{V}$ $\lambda_{\hat{V}}(T)$-pseudosaturates $T$. But since $\lambda_{\hat{V}}(T) \geq \aleph_1$, one of the clauses (A), (B), (C) above must hold, and so we are done.
\end{proof}

\section{The Interpretability Orders on Stable Theories}\label{KeislerOnStableSection}
In this section, we classify the interpretability orders on the stable theories. The classification of $\trianglelefteq$ on stable theories is due to Shelah \cite{ShelahIso}, specifically he proved Theorem~\ref{InterpOnKeisler1} in the case $\leq = \trianglelefteq$, although we fix a minor gap in his argument. The classification of $\trianglelefteq^*_{\aleph_1}$ on stable theories is due to Malliaris and Shelah \cite{InterpOrders}, specifically they proved Theorem~\ref{InterpOnKeisler1}(A), (B) and (D) for $\leq = \trianglelefteq^*_{\aleph_1}$. The classification of $\trianglelefteq^*_1$ on stable theories is due to the same authors \cite{InterpNew} (a work in preparation), specifically they prove Theorem~\ref{InterpOnKeisler3}(A) for $\leq = \trianglelefteq^*_{1}$. The authors privately communicated a sketch of the proof; our argument is an adaptation to models of $ZFC^-$, and we include it with their permission.

\begin{theorem}\label{InterpOnKeisler1}
	Suppose $\leq$ is one of $\trianglelefteq^*_{\kappa}, \trianglelefteq^{\times}_{\kappa}$ or $\trianglelefteq$, where $\kappa \geq \aleph_1$. 
	
	\begin{itemize}
		\item[(A)] Suppose $T_0, T_1$ are stable. Then $T_0 \leq T_1$ if and only if the following holds: if $T_0$ has the finite cover property, then so does $T_1$.
		\item[(B)] If $T_0$ is stable and $T_1$ is unstable, then $T_1 \not \leq T_0$.
		\item[(C)] If $T_0$ is stable and $T_1$ is unstable, then $T_0 \leq T_1$.
	\end{itemize}
\end{theorem}

\begin{theorem}\label{InterpOnKeisler2}
	Suppose $\leq$ is one of $\trianglelefteq^*_{\aleph_0}, \trianglelefteq^{\times}_{\aleph_0}$. 
	
		\begin{itemize}
		\item[(A)] Suppose $T_0, T_1$ are stable. Then $T_0 \leq T_1$ if and only if the following holds: if $T_0$ has the finite cover property, then so does $T_1$; and if $T_0$ is unsuperstable, then so is $T_1$.
		\item[(B)] If $T_0$ is stable and $T_1$ is unstable, then $T_1 \not \leq T_0$.
		\item[(C)] If $T_0$ is superstable and $T_1$ is unstable, then $T_0 \leq T_1$.
	\end{itemize}
	
\end{theorem}

\begin{theorem}\label{InterpOnKeisler3}
	Suppose $\leq$ is one of $\trianglelefteq^*_{1}, \trianglelefteq^{\times}_{1}$. 
	
	\begin{itemize}
		\item[(A)] Suppose $T_0, T_1$ are stable. Then $T_0 \leq T_1$ if and only if the following holds: if $T_0$ has the finite cover property, then so does $T_1$; and if $T_0$ is unsuperstable, then so is $T_1$; and if $T_0$ is not $\omega$-stable, then neither is $T_1$.
		\item[(B)] If $T_0$ is stable and $T_1$ is unstable, then $T_1 \not \leq T_0$.
		\item[(C)] If $T_0$ is superstable and $T_1$ is unstable, then $T_0 \leq T_1$.
	\end{itemize}
\end{theorem}

To begin, we will finish proving the positive reductions $T_0 \leq T_1$. In view of Theorems~\ref{PseudosaturationLevelsThm} and \ref{SaturationLevelsThm}, the only ones we have left are the respective clause (C)'s. 

We will need some notation.

\begin{definition}
	If $(L, <)$ is a linear order with proper initial segment $\omega$, then let $\mbox{lcf}_{(L, <)}(\omega)$ be the least cardinal $\kappa$ such that there is a descending sequence $(a_\alpha: \alpha < \kappa)$ from $L \backslash \omega$  which is cofinal above $\omega$ (i.e., for every $a \in L$, if $a < a_\alpha$ for each $\alpha < \kappa$, then $a \in \omega$). So $\mbox{lcf}_{(L, <)}(\omega) \leq \mu_{(L,<)}$ always.
	
	Suppose $\hat{V} \models ZFC^-$ is $\omega$-nonstandard. Then let $\mbox{lcf}_{\hat{V}}(\omega) = \mbox{lcf}_{(\hat{\omega}, <)}(\omega)$.
	
\end{definition}

\begin{theorem}\label{KeislerStableLemma0}
	Suppose $V \models ZFC^-$ is transitive, and $\mathbf{j}: V \preceq \hat{V}$ is $\omega$-nonstandard, and suppose $T \in V$ is a countable unstable theory. Write $\lambda = \mbox{lcf}_{\hat{V}}(\omega)$. Then $\hat{V}$ does not $\lambda^+$-pseudosaturate $T$.
\end{theorem}
\begin{proof}
	Choose $\phi(x, \overline{y})$ an unstable $T$-formula (we can suppose $x$ is a single variable by Theorem II.2.13 of \cite{ShelahIso}). Let $M \models T$ with $M \in V$.
	
	Choose $(\overline{a}^n_m: m < n < \omega)$ from $M^{|\overline{y}|}$ such that for each $m_* < n < \omega$, $M \models \exists x \bigwedge_{m < m_*} \phi(x, \overline{a}^n_m) \,\, \land \,\, \bigwedge_{m \geq m_*} \lnot \phi(x, \overline{a}^n_m)$. Let $(\hat{\overline{a}}^{\hat{n}}_{\hat{m}}: \hat{m} < \hat{n} < \omega) = \mathbf{j}((\overline{a}^n_m: m < n < \omega))$. 
	
	Let $\hat{n} < \hat{\omega}$ be nonstandard. Let $(\hat{c}_\alpha: \alpha < \lambda)$ be a decreasing sequence from $\hat{\omega}$ with $\hat{c}_0 = \hat{n}$, which is cofinal above $\omega$; this is possible by the definition of $\lambda$.
	
	Let $p(x)$ be the pseudofinite type over $\mathbf{j}_{\std}(M)$ defined by: $p(x) = \{\phi(x, \hat{\overline{a}}^{\hat{n}}_{i}): i < \omega\} \cup \{\lnot \phi(x, \hat{\overline{a}}^{\hat{n}}_{\hat{c}_\alpha}): \alpha < \lambda\}$. It suffices to show $p(x)$ is omitted by $\mathbf{j}_{\std}(M)$; so suppose towards a contradiction $b \in \mathbf{j}_{\std}(M)$ realized it. Let $Q(\hat{m}_*)$ be the property: for all $\hat{m} < \hat{m}_*$, $\phi(b, \hat{\overline{a}}^{\hat{n}}_{\hat{m}})$. $Q(i)$ holds for all $i < \omega$, but since $(\hat{c}_\alpha: \alpha < \lambda)$ is cofinal above $\omega$, $Q(\hat{m}_*)$ fails for all nonstandard $\hat{m}_*$. This contradicts overflow.
\end{proof}

We have thus proven almost all of the positive reductions in Theorems~\ref{InterpOnKeisler1}, \ref{InterpOnKeisler2}, \ref{InterpOnKeisler3}, except that in Theorem~\ref{InterpOnKeisler3}(C) we only have the result so far when $T_0$ is $\omega$-stable. To pass to superstable, we need the following:

\begin{theorem}\label{KeislerStableLemma0p5}
	Suppose $V \models ZFC^-$ is transitive, and $\mathbf{j}: V \preceq \hat{V}$ is $\omega$-nonstandard, and suppose $T \in V$ is a countable unstable theory. If $\hat{V}$ $\aleph_1$-pseudosaturates $T$, then $\hat{V}$ is $2^\omega$-rich.
\end{theorem}
\begin{proof}
Either $T$ has $SOP_2$ or else $T$ has $IP$, by Theorem~\ref{NonStableDichotomy}. If $T$ has $SOP_2$, then $\mathfrak{p}_{\hat{V}} \geq \aleph_1$ by Theorem~\ref{SOP2Max}, and hence $\hat{V}$ is $2^\omega$-rich by Theorem~\ref{localSaturation2}(A). Suppose on the other hand $T$ has $IP$; choose a formula $\phi(\overline{x}, \overline{y})$ of $T$ with the independence property. For notational convenience we write it as $\phi(x, y)$.

Choose $M \models T$ and $(b_s: s \in 2^{<\omega})$ from $M$, such that $M, (b_s: s \in 2^{<\omega}) \in V$, and for all disjoint $u, v \subseteq 2^{<\omega}$, $\{\phi(x, b_s): s \in u\} \cup \{\lnot \phi(x, b_s): s \in v\}$ is consistent. Let $(\hat{b}_{\hat{s}}: \hat{s} \in 2^{<\hat{\omega}}) = \mathbf{j}(b_s: s \in 2^{<\omega})$.

Now let $\eta \in 2^\omega$ be given. Let $X_0 = \{s \in 2^{<\omega}: s \subseteq \eta\}$ and let $X_1 = \{s \in 2^{<\omega}: s \not \subseteq \eta\}$. Let $p(x) = \{\phi(x, \hat{b}_s): s \in X_0\} \cup \{\lnot \phi(x, \hat{b}_s): s \in X_1\}$. Then $p(x)$ is a countable, consistent partial type over $\mathbf{j}_{\std}(M)$; and $p(x)$ is pseudofinite, since it is a subset of $(\hat{b}_{\hat{s}}: \hat{s} \in 2^{<\hat{n}})$ for any nonstandard $\hat{n}< \hat{\omega}$. Hence $p(x)$ is realized by some $a \in \mathbf{j}_{\std}(M)$. By overflow, we can find some $\hat{n}< \hat{\omega}$ nonstandard such that $\bigcup \{\hat{s} \in 2^{<\hat{n}}: \phi(a, \hat{b}_{\hat{s}})\}$ is an element $\hat{s}_* \in 2^{\hat{n}}$. Let $\hat{\eta} \in 2^{\hat{\omega}}$ extend $\hat{s}_*$; then $\hat{\eta}(n) = \hat{s}_*(n)= \eta(n)$ for each $n < \omega$.  
\end{proof}

\noindent \textbf{Status Update.} We have proved all of the positive reductions in Theorems~\ref{InterpOnKeisler1}, \ref{InterpOnKeisler2}, \ref{InterpOnKeisler3}. 

\vspace{1 mm}

Next, we finish the proof of Theorem~\ref{InterpOnKeisler1}. We will need an ultrafilter construction.  Shelah's proof in \cite{ShelahIso} goes through a precursor to the Existence Theorem of \cite{DividingLine}. We will give a streamlined treatment in terms of complete Boolean algebras.


\begin{definition}
Suppose $\mathcal{B}$ is a complete Boolean algebra. Then $(\omega, <)^{\mathcal{B}}$, the $\mathcal{B}$-valued ultrapower of $\omega$, is a full $\mathcal{B}$-valued model of $\mbox{Th}(\omega, <)$ defined as follows. Its domain is the set of all partitions of $\mathbf{n}$ by $\omega$, that is, the set of all $\mathbf{n}: \omega \to \mathcal{B}$ such that for all $n < m$, $\mathbf{n}(n) \wedge \mathbf{n}(m) = 0$, and such that $\bigvee_n \mathbf{n}(n) = 1$. We have $\|\mathbf{n} < \mathbf{m}\|_{\omega^{\mathcal{B}}} = \bigvee_{n < m} \mathbf{n}(n) \wedge \mathbf{m}(m)$, and similarly for $\leq, =$, etc. If $\mathcal{U}$ is an ultrafilter on $\mathcal{B}$, then we view $(\omega, <)^{\mathcal{B}}/\mathcal{U}$ as an elementary extension of $(\omega, <)$. See \cite{BVModelsUlrich} for more details.
\end{definition}
\begin{definition}\label{KeislerStableLemma3}
	Suppose $\mathcal{U}$ is an $\aleph_1$-incomplete ultrafilter on the complete Boolean algebra $\mathcal{B}$. Then define $\mu_{\mathcal{U}} = \mu_{(\omega, <)^{\mathcal{B}}/\mathcal{U}}$ and define $\mbox{lcf}_{\mathcal{U}}(\omega) = \mbox{lcf}_{(\omega, <)^{\mathcal{B}}/\mathcal{U}}(\omega)$.
\end{definition}

\begin{remark}
$\mbox{lcf}_{\mathcal{U}}(\omega)  \leq \mu_{\mathcal{U}} \leq |\mathcal{B}|$, using that whenever $\mathcal{B}$ is an infinite complete Boolean algebra, then $|\mathcal{B}|^{\aleph_0} = \mathcal{B}$ \cite{CompleteBooleanAlgebraCard}.
\end{remark}

The following theorem and corollary will show that these are robust properties of $\mathcal{U}$.

\begin{theorem}\label{EndExtension0}
	Suppose $\mathcal{B}$ is a complete Boolean algebra, and $(\omega, <) \preceq (\mathbf{L}, <)$ is a $\mathcal{B}$-valued elementary extension such that $(\mathbf{L}, <)$ is $\aleph_1$-saturated. Then there is a unique embedding $\mathbf{i}: (\omega, <)^{\mathcal{B}} \preceq (\mathbf{L}, <)$. The image of $\mathbf{i}$ consists exactly of those $\mathbf{n} \in \mathbf{L}$ such that $\bigvee_{n < \omega} \|\mathbf{n} = n\|_{\mathbf{L}} = 1$. Thus, if $\mathbf{n}$ is in the image of $\mathbf{i}$ and $\|\mathbf{m} \leq \mathbf{n}\|_{\mathbf{L}} = 1$, then $\mathbf{m}$ is in the image of $\mathbf{i}$.
\end{theorem}

\begin{proof}
	Given $\mathbf{n} \in (\omega, <)^{\mathcal{B}}$, we can by $\aleph_1$-saturation find $\mathbf{m} \in \mathbf{L}$ such that each $\|\mathbf{m}= n\|_{\mathbf{L}} = \mathbf{n}(n)$. I claim that this specifies $\mathbf{m}$ uniquely, so that we can set $\mathbf{i}(\mathbf{n}) = \mathbf{m}$. Indeed, suppose $\mathbf{m}'$ were another element of $\mathbf{L}$ with each $\|\mathbf{m}' = n\|_{\mathbf{L}} = \mathbf{n}(n)$. Then $\bigvee_n \|\mathbf{m} = n\|_{\mathbf{L}} = \bigvee_n \|\mathbf{m}' = n\| = 1$, so we can choose a maximal antichain $\mathbf{C}$ of $\mathcal{B}$ such that for each $\mathbf{c} \in \mathbf{C}$, there are $m(\mathbf{c}), m'(\mathbf{c}) < \omega$ with $\mathbf{c} \leq \|\mathbf{m} = m(\mathbf{c}) \land \mathbf{m}' = m'(\mathbf{c})\|_{\mathbf{L}}$. But then each $m(\mathbf{c}) = m'(\mathbf{c})$, as otherwise $\mathbf{c} \leq \mathbf{n}(m(\mathbf{c})) \wedge \mathbf{n}(m'(\mathbf{c})) =0$. Thus each $\mathbf{c} \leq \|\mathbf{m} = \mathbf{m}'\|_{\mathbf{L}}$, and this happens on a maximal antichain. Thus $\|\mathbf{m} = \mathbf{m}'\|_{\mathbf{L}} = 1$, and so $\mathbf{m} = \mathbf{m}'$.
	
	It is simple to see that $\mathbf{i}$ is an elementary embedding, and that we had no choice. Moreover, if $\mathbf{m} \in \mathbf{L}$ has $\bigvee_n \|\mathbf{m}  = n\|_{\mathbf{L}} = 1$, then we can define $\mathbf{n} \in (\omega, <)^{\mathcal{B}}$ with each $\mathbf{n}(n) = \|\mathbf{m} = n\|_{\mathbf{L}}$, so $\mathbf{i}(\mathbf{n}) = \mathbf{m}$.
	
	The final claim follows easily, since we have the identity: $\bigvee_n \|\mathbf{n} = n\|_{\mathbf{L}} = 1$ if and only if $\bigvee_n \|\mathbf{n} \leq n\|_{\mathbf{L}} = 1$.
\end{proof}

\begin{corollary}\label{EndExtension}
	Suppose $\mathcal{B}$ is a complete Boolean algebra and $\mathcal{U}$ is an ultrafilter on $\mathcal{B}$. Suppose $(\mathbf{L}, <) \succeq (\omega, <)$ is $\aleph_1$-saturated. Then $(\mathbf{L}, <)/\mathcal{U}$ is an end extension of $(\omega, <)^{\mathcal{B}}/\mathcal{U}$. Hence $\mu_{(\mathbf{L}, <)/\mathcal{U}} = \mu_{\mathcal{U}}$ and $\mbox{lcf}_{(\mathbf{L}, <)/\mathcal{U}} = \mbox{lcf}_{\mathcal{U}}(\omega)$.
\end{corollary}
\begin{proof}
	It suffices to show the first claim. By the previous theorem, we can suppose after relabeling $(\omega, <)^{\mathcal{B}} \preceq (\mathbf{L}, <)$. Thus we can also view $(\omega,<)^{\mathcal{B}}/\mathcal{U} \preceq (\mathbf{L}, <)/\mathcal{U}$, and so the statement of the corollary makes sense.
	
	Suppose $\mathbf{n} \in \mathbf{L}$ and $\mathbf{m} \in (\omega, <)^{\mathcal{B}}$ with $\|\mathbf{n} \leq \mathbf{m}\|_{\mathbf{L}} \in \mathcal{U}$. Note that $\mbox{min}(\mathbf{n}, \mathbf{m})$ makes sense by fullness of $\mathbf{L}$, and that $[\mbox{min}(\mathbf{n}, \mathbf{m})]_{\mathcal{U}} = [\mathbf{n}]_{\mathcal{U}}$.  By Corollary~\ref{EndExtension0}, $\mbox{min}(\mathbf{n}, \mathbf{m}) \in (\omega,<)^{\mathcal{B}}$, so $[\mathbf{n}]_{\mathcal{U}} \in (\omega,<)^{\mathcal{B}}/\mathcal{U}$.
\end{proof}

Thus:

\begin{corollary}\label{KeislerStableLemma5}
	Suppose $\mathcal{U}$ is an ultrafilter on the complete Boolean algebra $\mathcal{B}$ and $T$ is a complete countable theory. If $T$ does not have the finite cover property, then $\mathcal{U}$ $\lambda^+$-saturates $T$ for all $\lambda$.
	
	Suppose additionally $\mathcal{U}$ is $\aleph_1$-incomplete. If $T$ is stable with the finite cover property, then $\mathcal{U}$ $\lambda^+$-saturates $T$ if and only if $\lambda \leq \mu_{\mathcal{U}}$. If $T$ is unstable then $\mathcal{U}$ does not $\mbox{lcf}_{\mathcal{U}}(\omega)^+$-saturate $T$.
\end{corollary}
\begin{proof}
	Suppose $V \models ZFC^-$ is transitive with $T \in V$, and let $\lambda$ be an infinite cardinal, and choose  $\mathbf{i} := V \preceq \mathbf{V}$ with $\mathbf{V}$ $\lambda^+$-saturated. Write $\hat{V} = \mathbf{V}/\mathcal{U}$. By Theorem~\ref{KeislerAndInterp2}, $\mathcal{U}$ $\lambda^+$-saturates $T$ if and only if $\hat{V}$ $\lambda^+$-pseudosaturates $T$. Note also that $\hat{V}$ is $\aleph_1$-saturated, since $\mathcal{U}$ is $\aleph_1$-good.
	
	If $T$ does not have the finite cover property, then we conclude by Theorem~\ref{PseudosaturationLevelsThm} that $\hat{V}$ $\lambda^+$-pseudosaturates $T$, as desired.
	
	Suppose $\mathcal{U}$ is $\aleph_1$-incomplete. Then by Corollary~\ref{EndExtension} and Lemma~\ref{SatAndPseudo3First}, $\mu_{\mathcal{U}} = \mu_{\hat{V}}$ and $\mbox{lcf}_{\mathcal{U}} = \mbox{lcf}_{\hat{V}}$. Thus we conclude by Theorems \ref{KeislerStableLemma0} and \ref{PseudosaturationLevelsThm}.
\end{proof}

One can ask after the situation for $\aleph_1$-complete $\mathcal{U}$; in fact we will prove in \cite{InterpOrders2Ulrich} (Remark 8.13) that if $\mathcal{U}$ is $\aleph_1$-complete, then $\mathcal{U}$ $\lambda^+$-saturates every stable theory for every $\lambda$, so it is reasonable to set $\mu_{\mathcal{U}} = \infty$.

Now to finish the proof of Theorem~\ref{InterpOnKeisler1} it suffices to verify the following theorem. The argument is based on Theorem 3.12 from \cite{ShelahIso} Chapter 6. To handle the orders $\trianglelefteq^\times_\kappa, \trianglelefteq^*_\kappa$ for $\kappa > \aleph_1$, we arrange additionally that the ultrafilter we construct is $\kappa$-good. For ease of notation, we actually arrange that it is $\kappa^+$-good, which is more than needed.

\begin{remark}\label{GapRemark}
	There is a minor gap in the proof of Theorem 3.12 in \cite{ShelahIso}. Namely, Claim VI.3.18 (4) is false (all further references are to chapter VI). The problem is that the appeal to Claim 3.17 (2) is unjustified unless $D_\beta$ is maximal subject to (ii) alone. It is not hard to come up with counterexamples.
	
	This invalidates Claim 3.19(1), which is used in Claim 3.21, which is used in the proof of Theorem 3.12.
	
	Our proof looks very different due to the Boolean algebra terminology, but this is mostly a matter of presentation. The patch is to delete item (ii) in Claim 3.17, and arrange by hand that the hence clause in Claim 3.19(1) holds at the next stage of the ultrafilter construction. This is what Claim 1 in the following theorem accomplishes.
	
	The above references are all to the \emph{second} edition of Classification Theory. The first edition version is also flawed, but in a different way.
	
	
\end{remark}

\begin{theorem}\label{KeislerOnStableUlt}
	Suppose $\aleph_0 \leq \kappa < \lambda \leq \mu$ are cardinals such that $\mu = \mu^{\kappa}$ and such that $\lambda$ is regular. Then there exists a complete Boolean algebra $\mathcal{B}$ with the $\kappa^+$-c.c. and an $\aleph_1$-incomplete, $\kappa^+$-good ultrafilter $\mathcal{U}$ on $\mathcal{B}$, such that $\mu_{\mathcal{U}} = \mu$ and $\mbox{lcf}_\mathcal{U}(\omega) = \lambda$.
\end{theorem}
\begin{proof}
	
	Let $\alpha_*$ be the ordinal product $\mu \lambda$; so $\mbox{cof}(\alpha_*) = \lambda$, and for all $\alpha < \alpha_*$, $|\alpha_* \backslash \alpha| = \mu$. For each $\alpha \leq \alpha_*$, let $P_\alpha$ denote the set of all finite partial functions from $\alpha$ to $\kappa$, ordered by reverse inclusion; let $\mathcal{B}_\alpha$ be its Boolean algebra completion (see \cite{Jech} for a discussion of Boolean algebra completions in the context of forcing). 
	
	We recall some relevant terminology; see \cite{BVModelsUlrich} for a fuller discussion. A $\kappa$-distribution in $\mathcal{B}$ is a function $\mathbf{A}: [\kappa]^{<\aleph_0} \to \mathcal{B}$ such that $\mathbf{A}(\emptyset) = 1$ and such that $s \subseteq t$ implies $\mathbf{A}(s) \geq \mathbf{A}(t)$. $\mathbf{A}$ is in $\mathcal{U}$ if each $\mathbf{A}(s) \in \mathcal{U}$. $\mathbf{A}$ is multiplicative if each $\mathbf{A}(s) = \bigwedge_{\alpha \in s} \mathbf{A}(\{\alpha\})$. If $\mathbf{B}$ is another $\kappa$-distribution, then $\mathbf{B}$ refines $\mathbf{A}$ if each $\mathbf{B}(s) \subseteq \mathbf{A}(s)$. An ultrafilter $\mathcal{U}$ on $\mathcal{B}_{\alpha_*}$ is $\kappa^+$-good if and only if every $\kappa$-distribution in $\mathcal{U}$ has a multiplicative refinement in $\mathcal{U}$.

	Note that for each $\alpha \leq \alpha_*$, $P_\alpha$ has the $\kappa^+$-c.c., by the $\Delta$-system lemma; hence also $\mathcal{B}_\alpha$ has the $\kappa^+$-c.c., and every element of $\mathcal{B}_\alpha$ can be written as the join of $\kappa$-many elements from $P_\alpha$. Hence, each $|\mathcal{B}_\alpha| \leq |P_\alpha^\kappa| \leq \mu^\kappa = \mu$, and in particular $|\mathcal{B}_{\alpha_*}| = |\alpha_*| = \mu$. Further, since $\mbox{cof}(\alpha_*) = \lambda > \kappa$, $[\mathcal{B}_{\alpha_*}]^\kappa = \bigcup_{\alpha < \alpha_*} [\mathcal{B}_\alpha]^\kappa$. 
	
	Also, note that for all $\alpha < \alpha' \leq \alpha_*$, $(\omega, <)^{\mathcal{B}_\alpha} \preceq (\omega, <)^{\mathcal{B}_{\alpha'}}$, and so if $\mathbf{n}, \mathbf{m} \in (\omega, <)^{\mathcal{B}_\alpha}$, then $\|\mathbf{n} < \mathbf{m}\|_{(\omega, <)^{\mathcal{B}_\alpha}} = \|\mathbf{n} < \mathbf{m}\|_{(\omega, <)^{\mathcal{B}_{\alpha'}}}$, etc. Thus, we can omit the subscripts without ambiguity in what follows.

	Let $I = \{2 \cdot \alpha: \alpha < \alpha_*\}$ and $J = \{2 \cdot \alpha+1: \alpha < \alpha_*\}$.
	\vspace{1 mm}
	
	\noindent \textbf{Claim 0.} We can find ultrafilters $\mathcal{U}_\alpha$ on $\mathcal{B}_\alpha$ for each $\alpha < \alpha_*$, such that:
	
	\begin{itemize}
		\item[(I)]  For $\alpha < \alpha'$, $\mathcal{U}_\alpha \subseteq \mathcal{U}_{\alpha'}$;
		\item[(II)] For each $\alpha \in I$, there is some $\mathbf{n}_{\alpha} \in (\omega, <)^{\mathcal{B}_{\alpha+1}}$ which is $\mathcal{U}_{\alpha+1}$-nonstandard (i.e. for each $m < \omega$, $\|\mathbf{n}_{\alpha} \geq m\| \in \mathcal{U}_{\alpha+1}$), satisfying that for each $\mathcal{U}_\alpha$-nonstandard $\mathbf{m} \in (\omega, <)^{\mathcal{B}_\alpha}$, $\|\mathbf{n}_\alpha < \mathbf{m}\| \in \mathcal{U}_{\alpha+1}$;
		\item[(III)] For each $\alpha < \alpha_*$ and for each $\kappa$-distribution $\mathbf{A}$ in $\mathcal{U}_\alpha$, there is some $\beta \geq \alpha$ such that $\mathbf{A}$ has a multiplicative refinement in $\mathcal{U}_{\beta+1}$.
	\end{itemize}
	
	Towards proving Claim 0, we prove two additional claims; the first will show that we can handle (II), and the second will show that we can handle each individual $\lambda$-distribution $\mathbf{A}$ in (III). This suffices to prove Claim 0, since there are only $|\mathcal{B}_{\alpha_*}^\kappa| = \mu$ $\kappa$-distributions in total.
	
	\vspace{1 mm}
	
	\noindent \textbf{Claim 1}. Suppose $\alpha \in I$ and $\mathcal{U}_\alpha$ is an ultrafilter on $\mathcal{B}_\alpha$. For each $\gamma < \kappa$ let $\mathbf{c}_\gamma = \{\langle \alpha, \gamma \rangle \} \in P_{\alpha+1}$, so $(\mathbf{c}_\gamma: \gamma < \kappa)$ is a maximal antichain of $\mathcal{B}_{\alpha+1}$. Let $f: \kappa \to \omega$ be surjective, and define $\mathbf{n}_\alpha \in (\omega, <)^{\mathcal{B}_{\alpha+1}}$ via $\|\mathbf{n}_\alpha = m\| = \bigvee_{f(\gamma) = m} \mathbf{c}_\gamma$. Then there is an ultrafilter $\mathcal{U}_{\alpha+1}$ on $\mathcal{B}_{\alpha+1}$ extending $\mathcal{U}_\alpha$, such that $\mathbf{n}_\alpha$ is $\mathcal{U}_{\alpha+1}$-nonstandard, and such that for each $\mathcal{U}_\alpha$-nonstandard $\mathbf{m} \in (\omega, <)^{\mathcal{B}_\alpha}$, $\|\mathbf{n}_\alpha < \mathbf{m}\| \in \mathcal{U}_{\alpha+1}$. 
	
	\begin{proof}Define $\mathcal{X} \subseteq \mathcal{B}_{\alpha+1}$ to be $\mathcal{U}_\alpha \cup \{\|\mathbf{n}_\alpha \geq m\|: m < \omega\} \cup \{\|\mathbf{n}_\alpha < \mathbf{n}\|: \mathbf{n} \in (\omega, <)^{\mathcal{B}_\alpha}\}$.  It suffices to show $\mathcal{X}$ has the finite intersection property, since then any ultrafilter on $\mathcal{B}_{\alpha+1}$ extending $\mathcal{X}$ will be as desired.
	
	So suppose towards a contradiction this were not the case. Then there must be some $m_* < \omega$  and some finite tuple $(\mathbf{m}_i: i < i_*)$ of $\mathcal{U}_{\alpha}$-nonstandard elements from $(\omega, <)^{\mathcal{B}_\alpha}$ and some $\mathbf{a} \in \mathcal{U}_\alpha$, such that $\mathbf{a} \wedge \|\mathbf{n}_\alpha \geq m_*\| \wedge \bigwedge_{i < i_*} \|\mathbf{n}_\alpha < \mathbf{m}_i\| = 0$.
	
	Choose some $\gamma < \kappa$ with $f(\gamma) = m_*$; note $\mathbf{c}_\gamma \leq \|\mathbf{n}_\alpha = m_*\|$. By wedging the preceding equation with $\mathbf{c}_\gamma$, we get that $\mathbf{a}\wedge \mathbf{c}_\gamma \wedge \bigwedge_{i < i_*} \|m_* < \mathbf{m}_i\| = 0$. But $\mathbf{a}\wedge \bigwedge_{i < i_*} \|m_* < \mathbf{m}_i\| \in \mathcal{U}_\alpha$ is nonzero, so we can find $g \in P_\alpha$ below it; then $g \cup \{(\alpha, \gamma)\} = g \wedge \mathbf{c}_\gamma \in P_{\alpha+1}$ is nonzero, a contradiction. Thus we have proven Claim 1.
	\end{proof}
	\noindent \textbf{Claim 2.}  Suppose $\alpha \in J$, and $\mathcal{U}_\alpha$ is an ultrafilter on $\mathcal{B}_\alpha$, and $\mathbf{A}$ is a $\kappa$-distribution in $\mathcal{U}_\alpha$. Then there is an ultrafilter $\mathcal{U}_{\alpha+1}$ on $\mathcal{B}_{\alpha+1}$ extending $\mathcal{U}_\alpha$, such that $\mathbf{A}$ has a multiplicative refinement in $\mathcal{U}_{\alpha+1}$.
	
	\begin{proof}Choose a bijection $\rho: [\kappa]^{<\aleph_0} \to \kappa$. For each $s \in [\kappa]^{<\aleph_0}$, let $\mathbf{c}_s = \{(\alpha, \rho(s))\} \in P_{\alpha+1}$; so $(\mathbf{c}_s: s \in [\kappa]^{<\aleph_0})$ is a maximal antichain of $\mathcal{B}_{\alpha+1}$, and whenever $\mathbf{a} \in \mathcal{B}_\alpha$ is nonzero, then $\mathbf{a} \wedge \mathbf{c}_s$ is nonzero for all $s$. For each $s \in [\lambda]^{<\aleph_0}$, define $\mathbf{B}(s) = \bigvee\{\mathbf{A}(t) \wedge \mathbf{c}_t: s \subseteq t \in [\kappa]^{<\aleph_0}$. 
	
	$\mathbf{B}$ is clearly a $\kappa$-distribution. 
	
	I claim that $\mathbf{B}$ is multiplicative; let $s \in [\kappa]^{<\aleph_0}$. Suppose towards a contradiction $\mathbf{e} := \left(\bigwedge_{\alpha \in s} \mathbf{B}_{\{\alpha\}}\right) \wedge (\lnot \mathbf{B}(s))$ were nonzero. Then we can find $\mathbf{e}' \leq \mathbf{e}$ nonzero, and $(s_\alpha: \alpha \in s)$ a sequence from $[\kappa]^{<\aleph_0}$, such that each $\alpha \in s_\alpha$, and such that $\mathbf{e}' \leq \mathbf{A}(s_\alpha) \wedge \mathbf{c}_{s_\alpha}$ for each $\alpha \in s$. Since $(\mathbf{c}_s: s \in [\kappa]^{<\aleph_0})$ is an antichain this implies $s_\alpha = s_{\alpha'} = t$ say, for all $\alpha, \alpha' \in s$. Visibly then $s\subseteq t$, and so $\mathbf{e}' \leq \mathbf{A}(t) \wedge \mathbf{c}_t$, contradicting that $\mathbf{e}' \wedge \mathbf{B}(s) = 0$. 
	
	I claim that $\mathcal{U}_\alpha \cup \{\mathbf{B}(s): s \in [\kappa]^{<\aleph_0}\}$ has the finite intersection property, which suffices.  So suppose towards a contradiction it did not; then we can find $s \in [\kappa]^{<\aleph_0}$ and $\mathbf{a} \in \mathcal{U}_\alpha$ such that $\mathbf{a} \wedge \mathbf{B}(s) = 0$. But then $\mathbf{a} \wedge \mathbf{A}(s) \wedge \mathbf{c}_s = 0$, so $\mathbf{a} \wedge \mathbf{A}(s) = 0$, but $\mathbf{A}(s) \in \mathcal{U}_\alpha$ so this is a contradiction. This finishes the proof of Claim 2.
	\end{proof}

	As remarked above, Claims 1 and 2 suffice to prove Claim 0.
	
	Let $\mathcal{U} = \bigcup_{\alpha< \alpha_*} \mathcal{U}_\alpha$. Since $[\mathcal{B}_{\alpha_*}]^\kappa = \bigcup_{\alpha < \alpha_*} [\mathcal{B}_\alpha]^\kappa$, we have that $\mathcal{U}$ is a $\kappa^+$-good ultrafilter on $\mathcal{B}_{\alpha_*}$, and also $(\omega, <)^{\mathcal{B}_{\alpha_*}} = \bigcup_{\alpha < \alpha_*} (\omega, <)^{\mathcal{B}_\alpha}$. From the latter it follows that $([\mathbf{n}_\alpha]_{\mathcal{U}_{\alpha+1}}: \alpha \in I)$ is a cofinal sequence above $\omega$ in $(\omega, <)^{\mathcal{B}_{\alpha_*}}/\mathcal{U}$, and hence $\mbox{lcf}_{\mathcal{U}}(\omega) = \mbox{cof}(\alpha_*) = \lambda$ and $\mu_{\mathcal{U}} \geq \min\{\alpha_* \backslash \alpha: \alpha < \alpha_*\}= \mu$. But also $\mu_{\mathcal{U}} \leq |\mathcal{B}_{\alpha_*}| \leq |P_{\alpha_*}|^{\aleph_0} = \mu$, and so $\mu_{\mathcal{U}} = \mu$.
\end{proof}

\begin{remark}
The hypotheses in Theorem~\ref{KeislerOnStableUlt} are sharp (except that with additional notational difficulties, if $\kappa$ is an uncountable regular limit cardinal and $\lambda \geq \kappa$, we can replace $\kappa^+$ by $\kappa$). $\lambda$ clearly must be regular. Further, it is not hard to check that if $\hat{V} \models ZFC^-$ and $\mathfrak{p}_{\hat{V}} \geq \aleph_1$, then $\mu_{\hat{V}}^{\aleph_0} = \mu_{\hat{V}}$, and so the hypothesis that $\mu_{\mathcal{U}}^{\aleph_0} = \mu_{\mathcal{U}}$ is necessary; see Exercise 2.10 of \cite{ShelahIso}, or just note that by Theorem~\ref{localSaturation2}, given $\hat{n}< \hat{\omega}$ nonstandard, if we choose $\hat{m} < \hat{n}$ nonstandard with $|\hat{m}^{\hat{m}}|_{\hat{V}} < \hat{n}$, then every $f: \omega \to \hat{m}$ extends to some function $\hat{f}: \hat{m} \to \hat{m}$ with $\hat{f} \in \hat{V}$, and so $|\hat{m}|_{\mathbb{V}}^{\aleph_0} \leq |\hat{n}|_{\mathbb{V}}$.
\end{remark}

\noindent \emph{Proof of Theorem~\ref{InterpOnKeisler1}.} Let $\leq$ be either $\trianglelefteq^*_\kappa$ or $\trianglelefteq^\times_\kappa$ for some $\kappa \geq \aleph_1$, or else $\trianglelefteq$; if $\leq = \trianglelefteq$ then write $\kappa = \aleph_1$. We have noted that all the positive reductions hold  in Theorem~\ref{InterpOnKeisler1}. Thus we have to show that the negative reductions in (A) hold, and also that (B) holds. Suppose $T_0$ is unstable, and $T_1$ is stable with the finite cover property, and $T_2$ does not have the finite cover property. It suffices to show $T_2 \not \leq T_1 \not \leq T_0$. 

Let $\lambda = \kappa^+$ and choose $\mu > \lambda$ with $\mu^\kappa = \mu$. By Theorem~\ref{KeislerOnStableUlt}, we can find a complete Boolean algebra $\mathcal{B}$ with the $\kappa^+$-c.c. and an $\aleph_1$-incomplete, $\kappa^+$-good ultrafilter $\mathcal{U}$ on $\mathcal{B}$ with $\mu_{\mathcal{U}} = \mu$ and with $\mbox{lcf}_{\mathcal{U}}(\omega) = \lambda$. Note that $\mathcal{U}$ $\infty$-saturates $T_2$, and $\mathcal{U}$ $\mu$-saturates $T_1$ but does not $\mu^+$-saturate $T_1$, and $\mathcal{U}$ does not $\lambda^+$-saturate $T_0$. 

Suppose $\leq = \trianglelefteq$. Then by Theorem~\ref{KeislerChar}, we get that $T_2 \not \trianglelefteq T_1 \not \trianglelefteq T_0$. Since $\trianglelefteq^\times_{\aleph_1}$ and $\trianglelefteq^*_{\aleph_1}$ refine $\trianglelefteq$, it follows that if $\leq$ is either of these, then $T_2 \not \leq T_1 \not \leq T_0$.

So it remains to consider the case where $\kappa > \aleph_1$.

Let $i < j < 3$. We wish to show $T_j \not \leq T_i$. If $i = 0$ then write $\sigma = \lambda$, and otherwise write $\sigma = \mu$.

Let $V$ be any transitive model of $ZFC^-$ with $T_i, T_j \in V$. Then we claim that whenever $\mathbf{j}: V \preceq \hat{V}$ and whenever $\hat{V}_*$ is an expansion of $\hat{V}$, then there is some $\mathbf{j}': \hat{V}_* \preceq \hat{W}_*$ such that $\hat{W}_*$ is $\kappa$-saturated, and $\hat{W}_*$ does not $\sigma^+$-pseudosaturate $T_j$, and for every $N \in V$ with $N \models T_i$, $\mathbf{j}_{\std}(N)$ is $\sigma^+$-saturated. Note that this suffices to prove $T_j \not \trianglelefteq^*_\kappa T_i$ and $T_j \not \trianglelefteq^\times_\kappa T_i$ simultaneously.

So let $\hat{V}_*$ be given. Choose an embedding $\mathbf{i}: \hat{V}_* \preceq \mathbf{V}_*$ into a $\sigma^+$-saturated $\mathcal{B}$-valued model of $\mbox{Th}(\hat{V}_*)$. Write $\hat{W}_* = \mathbf{V}_*/\mathcal{U}$ and let $\mathbf{j}': \hat{V}_* \preceq \hat{W}_*$ be the composition $[\cdot]_{\mathcal{U}} \circ \mathbf{i}$. Then this works, by choice of $\mathcal{U}$; here we are using that $\mathcal{U}$ is $\kappa^+$-good, and hence $\hat{W}_*$ is $\kappa^+$-saturated (we just needed $\kappa$-saturated). \qed

\vspace{1 mm}

Next, we finish the proof of Theorem~\ref{InterpOnKeisler2}. We need to show that if $\leq$ is one of $\trianglelefteq^*_{\aleph_0}$ or $\trianglelefteq^\times_{\aleph_0}$, and if $T_0$ is unsuperstable and $T_1$ is superstable, then $T_0 \not \leq T_1$. This follows from the following simple fact, combined with Theorems~\ref{SaturationLevelsThm} and \ref{PseudosaturationLevelsThm}. 

\begin{theorem}\label{SuperStableNonsat} Suppose $\hat{V}_*$ is an expansion of a model of $ZFC^-$, and suppose $\lambda$ is an infinite cardinal. Then we can find an elementary extension $\hat{V}_* \preceq \hat{W}_*$ such that $\hat{W}_*$ is $\aleph_0$-saturated and $\mbox{lcf}_{\hat{W}_*}(\omega) = \aleph_0$ and $\mu_{\hat{W}_*} \geq \lambda$. 
\end{theorem}
\begin{proof}
Write $\hat{V}_* = \hat{W}_0$. For each $n$, choose some $\aleph_0$-saturated elementary extension $\hat{W}_{n+1}$ of $\hat{W}_n$, such that $\hat{W}_{n+1}$ contains a nonstandard natural number below every nonstandard natural number of $\hat{W}_n$. Write $\hat{W}_* = \bigcup_n \hat{W}_n$. Then $\hat{W}_*$ clearly works.
\end{proof}

Finally, we finish the proof of Theorem~\ref{InterpOnKeisler3}. We need to show that if $\leq$ is one of $\trianglelefteq^*_{1}$ or $\trianglelefteq^\times_{1}$, and if $T_0$ is not $\omega$-stable and $T_1$ is $\omega$-stable, then $T_0 \not \leq T_1$. This follows from the following simple fact, combined with Theorems~\ref{SaturationLevelsThm} and \ref{PseudosaturationLevelsThm}:

\begin{theorem}\label{omegaStableNonsat}
	Suppose $\hat{V}_*$ is a countable expansion of a model of $ZFC^-$ in a countable language, and $\lambda$ is an infinite cardinal. Then we can find an elementary extension $\hat{V}_* \preceq \hat{W}_*$ such that $\hat{W}_*$ is not $2^\omega$-rich and $\mu_{\hat{W}_*} \geq \lambda$.
\end{theorem}
\begin{proof}
	Let $\hat{V}'$ be a Skolemization of $\hat{V}_*$. Let $\hat{W}^{\dagger}$ be a sufficiently saturated expansion of $\hat{V}'$. Choose a sufficiently long decreasing sequence $(\hat{n}_\alpha: \alpha < \alpha_*)$ of nonstandard natural numbers of $\hat{W}^{\dagger}$, such that each $\hat{n}_\alpha$ is below every nonstandard natural number in the Skolem-closure of $\hat{V}' \cup \{\hat{n}_\beta: \beta < \alpha\}$. By a typical application of the Erd\"{o}s-Rado theorem (see for instance Lemma 1.1.5 of \cite{SimplicityTheory}), we can find a decreasing sequence $(\hat{m}_\beta: \beta < \lambda^+)$ with the same properties, which is moreover indiscernible over $\hat{V}'$. Let $\hat{W}'$ be the Skolem hull of $\{\hat{m}_\beta: \beta < \lambda^+\} \cup \hat{V}'$; then $\hat{W}'$ (or more precisely, its reduct to the language of $\hat{V}_*$) is as desired: visibly $\mu_{\hat{W}'} = \lambda$, and $\hat{W}'$ realizes only countably many types over $\hat{V}'$, and hence cannot be $2^\omega$-rich.
\end{proof}

\begin{remark}\label{LanguageCardRemark}
	This is one of the rare times where it matters we restrict to countable languages. Namely, if we allowed $\hat{V}_*$ to have size continuum then Theorem~\ref{omegaStableNonsat} is false, since if $\hat{V}_*$ is $2^\omega$-rich then so is every elementary extension.
\end{remark}

\section{Conclusion}\label{DiscussionSec}

We take stock.

First of all, we have proven the following:

\begin{theorem}
Suppose $V \models ZFC^-$ is transitive, and $\mathbf{j}: V \preceq \hat{V} \models ZFC^-_{\pre}$, and $T \in V$ is a complete countable theory. Suppose $M_1, M_2 \in V$ are models of $T$. Then for all infinite cardinals $\kappa$, $\mathbf{j}_{\std}(M_1)$ is $\kappa$-saturated if and only if $\mathbf{j}_{\std}(M_2)$ is.
\end{theorem}
\begin{proof}
If $T$ is stable, then Theorem~\ref{SaturationLevelsThm} shows $\mathbf{j}_{\std}(M_i)$ is $\kappa$-saturated if and only if $\kappa \leq \lambda'_{\hat{V}}(T)$. If $\kappa = \aleph_0$, then by Corollary~\ref{ModKeislerOrder2}, $\mathbf{j}_{\std}(M_i)$ is $\kappa$-saturated if and only if either $T$ is small or else $\hat{V}$ is $2^\omega$-rich. Finally, if $T$ is unstable and $\kappa > \aleph_0$, then by Corollary~\ref{SatAndPseudoCor0}, $\mathbf{j}_{\std}(M_i)$ is $\kappa$-saturated if and only if $\hat{V}$ $\kappa$-pseudosaturates $T$. 
\end{proof}

Thus the following definition makes sense:

\begin{definition}
	Suppose $V \models ZFC^-$ is transitive, and $\mathbf{j}: V \preceq \hat{V} \models ZFC^-_{\pre}$, and $T \in V$ is a complete countable theory. If $\kappa$ is an infinite cardinal, then say that $\hat{V}$ $\kappa$-saturates $T$ if for some or every $M \models T$ with $M \in V$, $\mathbf{j}_{\std}(M)$ is $\kappa$-saturated.
\end{definition}

We also take this moment to extend the definitions of $\lambda_{\hat{V}}(T)$, $\lambda'_{\hat{V}}(T)$ to any theory $T$ (possibly unstable).

\begin{definition}
	Suppose $V \models ZFC^-$ is transitive, and $\mathbf{j}: V \preceq \hat{V}$ is $\omega$-nonstandard, and $T \in V$ is a complete countable theory. Then let $\lambda_{\hat{V}}(T)$ be the least infinite cardinal (possibly $\infty$) such that $\hat{V}$ does not $\lambda_{\hat{V}}(T)^+$-pseudosaturate $T$. If $\hat{V} \models ZFC^-_{\pre}$, then let $\lambda'_{\hat{V}}(T)$ be the least infinite cardinal (at most $|\hat{V}|$) such that $\hat{V}$ does not $\lambda'_{\hat{V}}(T)^+$-saturate $T$. 
\end{definition}

In view of Theorems \ref{PseudosaturationLevelsThm} and \ref{SaturationLevelsThm}, this does indeed generalize our definitions for stable theories.

\begin{remark}
	In the above setup, if $\hat{V} \models ZFC^-_{\pre}$ and $T$ is unstable, then $\lambda_{\hat{V}}(T) = \lambda'_{\hat{V}}(T)$ by Corollary~\ref{SatAndPseudoCor0}. If $T$ is stable, then there are exceptions, which can be read off from Definition~\ref{LongDef}. Hence, it suffices to understand $(\hat{V}, T) \mapsto \lambda_{\hat{V}}(T)$.
\end{remark}

\begin{remark}
	Suppose $V \models ZFC^-$ is transitive, and $\mathbf{j}: V \preceq \hat{V}$ is $\omega$-nonstandard, and $T \in V$ is a complete countable theory. Then $\lambda_{\hat{V}}(T)\geq \mathfrak{p}_{\hat{V}}$, and if $T$ has $SOP_2$ then $\lambda_{\hat{V}}(T)= \mathfrak{p}_{\hat{V}}$. If $\hat{V} \models ZFC^-_{\pre}$ then the same holds for $\lambda'_{\hat{V}}$.
	
\end{remark}

We can now obtain the following characterizations of $\trianglelefteq^*_\kappa$: 

\begin{theorem}\label{SatAndPseudoCor}
	Suppose $T_0, T_1$ are complete countable theories, and suppose $\kappa$ is infinite or $1$. Then the following are equivalent:
	
	\begin{itemize}
		\item[(A)] $T_0 \trianglelefteq^*_{\kappa} T_1$;
		\item[(B)] For all cardinals $\lambda$, there is some countable transitive $V \models ZFC^-$ with $T_0, T_1 \in V$, such that for all $\mathbf{j}: V \preceq \hat{V} \models ZFC^-_{\pre}$ with $\hat{V}$ $\kappa$-saturated, if $\hat{V}$ $\lambda^+$-pseudosaturates  $T_1$, then $\hat{V}$ $\lambda^+$-pseudosaturates $T_0$.
		\item[(C)] For all cardinals $\lambda$, there is some countable transitive $V \models ZFC^-$ with $T_0, T_1 \in V$, such that for all $\mathbf{j}: V \preceq \hat{V} \models ZFC^-_{\pre}$ with $\hat{V}$ $\kappa$-saturated, if $\hat{V}$ $\lambda^+$-saturates  $T_1$, then $\hat{V}$ $\lambda^+$-saturates $T_0$.
	\end{itemize}
	
	In particular, $\trianglelefteq^\times_\kappa \subseteq \trianglelefteq^*_\kappa$.
\end{theorem}
\begin{proof}
	Note that (A) is equivalent to (C) by Theorem \ref{SatAndPseudoThm1}. So it suffices to show that (A) and (B) are equivalent.
	
	If $T_1$ is unstable, use Corollaries~\ref{SatAndPseudoCor0} and \ref{SatAndPseudoCor1}. 
	
	Next, suppose $T_0$ is unstable and $T_1$ is stable. Then (A) fails by Theorem~\ref{InterpOnKeisler1}(B), and (B) fails by the proof of Theorem~\ref{InterpOnKeisler1}(B).
%
%
	
	Thus, we can assume $T_0$ and $T_1$ are both stable. We show (A) implies (B) implies (C).
	
	(A) implies (B): suppose (A) holds. By Theorems \ref{InterpOnKeisler1}, \ref{InterpOnKeisler2} and \ref{InterpOnKeisler3}, we get that $T_0 \trianglelefteq^\times_\kappa T_1$ (since $\trianglelefteq^*_\kappa$ and $\trianglelefteq^\times_\kappa$ coincide on stable theories). From this it follows immediately that (B) holds.
	
	(B) implies (C): suppose (B) holds. Let $\lambda$ be given. Choose a countable transitive $V \models ZFC^-$ with $T_0, T_1 \in V$ witnessing (B) holds. We claim $V$ witnesses (C) holds as well. Indeed, suppose $\mathbf{j}: V \preceq \hat{V} \models ZFC^-_{\pre}$ with $\hat{V}$ $\kappa$-saturated and $\omega$-nonstandard. We can suppose $\hat{V}$ $\lambda^+$-saturates $T_1$. Hence it $\lambda^+$-pseudosaturates $T_1$, and so by choice of $\hat{V}$, it $\lambda^+$-pseudosaturates $T_0$. In other words, $\lambda < \lambda_{\hat{V}}(T_0)$. We need to show that $\lambda < \lambda'_{\hat{V}}(T_0)$ (i.e. $\hat{V}$ $\lambda^+$-saturates $T_0$). Note that $\lambda_{\hat{V}}(T_0)$ is uncountable. We can suppose that $\lambda'_{\hat{V}}(T_0) < \lambda_{\hat{V}}(T_0)$, as otherwise we are done. Considering Definition~\ref{LongDef}, we see either Clause 5 is the only activated clause, or else Clause 3 is activated.
	
	We claim that Clause 3 cannot be activated. Indeed, suppose towards a contradiction that $T_0$ is not superstable and $\hat{V}$ is not $\aleph_1$-saturated.  Note that since $\hat{V}$ $\aleph_1$-saturates $T_1$, it follows that $T_1$ must be superstable (since otherwise it would be nonsupersimple, and so $\hat{V}$ would be $\aleph_1$-saturated Corollary~\ref{supersimpleCor1}). But then Theorem~\ref{SuperStableNonsat} implies that (B) fails, contrary to hypothesis. 
	
	So Clause 5 must be the only activated clause; i.e. $\lambda_{\hat{V}}(T_0) = \infty$ and $\lambda'_{\hat{V}}(T_0) = |\hat{V}|$. But then, since $\hat{V}$ $\lambda^+$-saturates $T_1$, it follows that $\lambda < |\hat{V}|$, and so in any case, $\lambda < \lambda'_{\hat{V}}(T_0)$. 
\end{proof}

We have also obtained the following (in fact, we could have observed it in Section~\ref{InterpOnStableSec2}):

\begin{corollary}\label{supersimpleCor2}
	Suppose $T_0, T_1$ are complete countable theories,  and $T_1$ is unsupersimple. Then $T_0 \trianglelefteq^*_1 T_1$ if and only if $T_0 \trianglelefteq^*_{\aleph_1} T_1$. In particular, $\trianglelefteq^*_1$ and $\trianglelefteq^*_{\aleph_1}$ coincide on unsupersimple theories.
\end{corollary}
\begin{proof}
	Suppose $T_0 \trianglelefteq^*_{\lambda\aleph_1} T_1$; it suffices to show that $T_0 \trianglelefteq^*_{\lambda 1} T_1$. By Theorem~\ref{SatAndPseudoCor}(C) for $\kappa = \aleph_1$, we can choose some countable transitive $V \models ZFC^-$ with $T_0, T_1 \in V$, such that whenever $\mathbf{j}:V \preceq \hat{V} \models ZFC^-_{\pre}$, if $\hat{V}$ is $\aleph_1$-saturated and if $\hat{V}$ $\lambda^+$-saturates $T_1$, then $\hat{V}$ $\lambda^+$-saturates $T_0$. 
	
	It suffices show that whenever $\mathbf{j}: V \preceq \hat{V} \models ZFC^-_{\pre}$, if $\hat{V}$ $\lambda^+$-saturates $T_1$, then $\hat{V}$ $\lambda^+$-saturates $T_0$. But if $\hat{V}$ $\lambda^+$-saturates $T_1$, then by Corollary~\ref{supersimpleCor1}, $\hat{V}$ is $\aleph_1$-saturated, so this follows from our hypothesis.
\end{proof}

As a result of this, we can import a couple of theorems on $\trianglelefteq^*_1$ to the context of $\trianglelefteq^*_{\aleph_1}$:

\begin{corollary}\label{supersimpleCor3}
	It is consistent that $SOP_2$ theories are exactly the maximal theories of $\trianglelefteq^*_{\aleph_1}$.
\end{corollary}
\begin{proof}
	In \cite{pEqualsT2} the corresponding statement for $\trianglelefteq^*_1$ is proven (using results of \cite{SH692}, \cite{SOP2pt2} and \cite{pEqualsTref}). So it suffices to show that if every $NSOP_2$ theories is nonmaximal in $\trianglelefteq^*_1$, then the same holds of $\trianglelefteq^*_{\aleph_1}$. Let $T$ be $NSOP_2$. Let  $T'$ be some strictly stable theory (i.e. a stable theory which is not superstable); let $T \oplus T'$ be the theory obtained by adding a new sort to $T$ and putting a copy of $T'$ in it. Then $\mbox{Th}(\mathbb{Q}, <) \not \trianglelefteq^*_{\aleph_1} T \oplus T'$ by Corollary~\ref{supersimpleCor2}, so $T \oplus T'$ is not maximal in $\trianglelefteq^*_{\aleph_1}$, and hence neither is $T$.
\end{proof}

\begin{corollary}\label{supersimpleCor4}
	Simplicity is a dividing line in $\trianglelefteq^*_{\aleph_1}$ (i.e. it is downward closed: if $T_0$ is not simple and $T_1$ is simple then $T_0 \not \trianglelefteq^*_{\aleph_1} T_1$).
\end{corollary}
\begin{proof}
	The corresponding statement for $\trianglelefteq^*_1$ is proven in \cite{InterpOrders}. Now suppose $T_0$ is simple and $T_1$ is not; we know that $T_1 \not \trianglelefteq^*_1 T_0$, and we wish to show that $T_1 \not \trianglelefteq^*_{\aleph_1} T_0$. As above, we can replace $T_0$ by $T_0 \oplus T'$ for some strictly stable $T'$, and so conclude from Corollary~\ref{supersimpleCor2}.
\end{proof}

\begin{remark}
	Supersimplicity is not a dividing line in $\trianglelefteq^*_{\aleph_1}$, since if $T$ is any stable theory, then $T \trianglelefteq^*_{\aleph_1} T_{rg}$. But supersimplicity is a dividing line in $\trianglelefteq^*_1$, as proved by Malliaris and Shelah \cite{InterpSeparation}.
\end{remark}

\bibliography{mybib}

\begin{thebibliography}{10}

\bibitem{TransferringSaturation}
J.~Baldwin, R.~Grossberg, and S.~Shelah.
\newblock Transfering {S}aturation, the {F}inite {C}over {P}roperty, and
  {S}tability.
\newblock {\em The Journal of Symbolic Logic}, 64(2):678--684, 1999.

\bibitem{Buechler}
S.~Buechler.
\newblock Lascar strong types in some simple theories.
\newblock {\em Journal of Symbolic Logic}, 64(2):817--824, 1999.

\bibitem{Casanovas}
E.~Casanovas.
\newblock The number of types in simple theories.
\newblock {\em Annals of Pure and Applied Logic}, 98:69--86, 1999.

\bibitem{SH692}
M.~D\v{z}amonja and S.~Shelah.
\newblock On $\trianglelefteq^*$-maximality.
\newblock {\em Annals of Pure and Applied Logic}, 125(1-3):119--158, 2004.

\bibitem{ZFCminus}
V.~Gitman, J.~Hamkins, and T.~Johnstone.
\newblock What is the theory {ZFC} without power set?
\newblock {\em Mathematical Logic Quarterly}, 62(4), 2011.

\bibitem{Jech}
T.~Jech.
\newblock {\em Set theory. The third millennium edition, revised and expanded}.
\newblock Springer Monographs in Mathematics. Springer-Verlag, Berlin, 2003.

\bibitem{Keisler}
J.~Keisler.
\newblock Ultraproducts which are not saturated.
\newblock {\em Journal of Symbolic Logic}, 32:23--46, 1967.

\bibitem{KimForking}
B.~Kim.
\newblock Forking in {S}imple {U}nstable {T}heories.
\newblock {\em Journal of the London Mathematical Society}, 57:257--267, 1998.

\bibitem{SimplicityTheory}
B.~Kim.
\newblock {\em Simplicity {T}heory}.
\newblock Oxford University Press, Oxford, 2014.

\bibitem{TP1}
B.~Kim and K.~Kim.
\newblock Notions around tree property 1.
\newblock {\em Annals of Pure and Applied Logic}, 162(9):698--709, 2011.

\bibitem{AlephEpsilon}
M.~C. Laskowski and S.~Shelah.
\newblock {P}-{NDOP} and {P}-decompositions of $\aleph_{\epsilon}$-saturated
  models of superstable theories.
\newblock {\em Fundamenta Mathematicae}, 229(1):47--81, 2015.

\bibitem{InterpNew}
M.~Malliaris and S.~Shelah.
\newblock Manuscript {F}1692 part {II}.
\newblock In preparation.

\bibitem{pEqualsT2}
M.~Malliaris and S.~Shelah.
\newblock Model-theoretic applications of cofinality spectrum problems.
\newblock Preprint.

\bibitem{InterpOrders}
M.~Malliaris and S.~Shelah.
\newblock A new look at interpretability and saturation.
\newblock Preprint.

\bibitem{InterpSeparation}
M.~Malliaris and S.~Shelah.
\newblock A {S}eparation {T}heorem {F}or {S}imple {T}heories.
\newblock Preprint.

\bibitem{DividingLine}
M.~Malliaris and S.~Shelah.
\newblock A dividing line within simple unstable theories.
\newblock {\em Advances in Math}, 249:250--288, 2013.

\bibitem{pEqualsTref}
M.~Malliaris and S.~Shelah.
\newblock Cofinality spectrum theorems in model theory, set theory, and general
  topology.
\newblock {\em Journal of the American Mathematical Society}, 29:237--297,
  2016.

\bibitem{Optimals}
M.~Malliaris and S.~Shelah.
\newblock Existence of optimal ultrafilters and the fundamental complexity of
  simple theories.
\newblock {\em Advances in Math}, 290:614--681, 2016.

\bibitem{CompleteBooleanAlgebraCard}
R.~Pierce.
\newblock A note on complete boolean algebras.
\newblock {\em Proceedings of the American Mathematical Society},
  9(6):892--896, 1958.

\bibitem{pillay}
A.~Pillay.
\newblock {\em Geometric Stability Theory}, volume~32 of {\em Oxford Logic
  Guides}.
\newblock Oxford University Press, 2002.

\bibitem{KeislerShelah}
S.~Shelah.
\newblock Every two elementarily equivalent models have isomorphic ultrapowers.
\newblock {\em Israel Journal of Mathematics}, 10:224--233, 1971.

\bibitem{ShelahIso}
S.~Shelah.
\newblock {\em Classification {T}heory}.
\newblock North-Holland, Amsterdam, 1978.

\bibitem{ShelahSimple}
S.~Shelah.
\newblock Simple unstable theories.
\newblock {\em Annals of Mathematical Logic}, 19(3):177--203, 1980.

\bibitem{SH500}
S.~Shelah.
\newblock Toward classifying unstable theories.
\newblock {\em Annals of Pure and Applied Logic}, 80:229--255, 1996.

\bibitem{SOP2pt2}
S.~Shelah and A.~Usvyatsov.
\newblock More on ${SOP}_1$ and ${SOP}_2$.
\newblock {\em Annals of Pure and Applied Logic}, 155(1):16--31, 2008.

\bibitem{InterpOrders2Ulrich}
D.~Ulrich.
\newblock Cardinal characteristics of models of set theory.
\newblock Preprint.

\bibitem{BVModelsUlrich}
D.~Ulrich.
\newblock Keisler's {O}rder and {F}ull {B}oolean-{V}alued {M}odels.
\newblock Preprint.

\bibitem{pEqualsTUlrich}
D.~Ulrich.
\newblock A {S}treamlined {P}roof of $\mathfrak{p}=\mathfrak{t}$.
\newblock Preprint.

\bibitem{LowDividingLine}
D.~Ulrich.
\newblock Lowness is a {D}ividing {L}ine in {K}eisler's {O}rder.
\newblock Submitted, April 2017.

\end{thebibliography}

\end{document}